\documentclass[11pt]{article}
\usepackage[margin=2cm]{geometry}
\usepackage{etex}
\usepackage{authblk}

\usepackage{amsmath}
\usepackage[usenames,dvipsnames,svgnames,table]{xcolor}
\usepackage{enumerate}
\usepackage{graphicx}
\usepackage{mathrsfs}
\usepackage{hyperref,comment}
\usepackage{subfig,xspace}

\usepackage{tikz}
\usepackage{pgf}
\usepackage{pgflibraryarrows}
\usepackage{pgffor}
\usepackage{pgflibrarysnakes}
\usetikzlibrary{fit} 
\usetikzlibrary{positioning}
\usepgflibrary{shapes}
\usetikzlibrary{snakes,automata}
\usetikzlibrary{shadows}

\tikzset{
  shadowed/.style={preaction={
      transform canvas={shift={(2pt,-1pt)}},draw opacity=.2,#1,preaction={
        transform canvas={shift={(3pt,-1.5pt)}},draw
        opacity=.1,#1,preaction={
          transform canvas={shift={(4pt,-2pt)}},draw
          opacity=.05,#1,
  }}}},
}

\makeatletter
\newif\if@restonecol
\makeatother
\usepackage{ntheorem}

\newtheorem{theorem}{Theorem}	
\renewtheorem*{theorem*}{Theorem}	

\newtheorem{definition}{Definition}

\newtheorem{lemma}{Lemma}

\newtheorem{remark}{Remark}

\newtheorem{corollary}{Corollary}
\newtheorem{proposition}{Proposition}  

\usepackage{amsmath,mathrsfs}
\usepackage{amssymb}

\def\R{\mathbb{R}}

\def\F{\mathcal{F}}

\newcommand{\argmin}{\operatornamewithlimits{argmin}}

\newcommand{\G}{\mathcal{G}}

\usetikzlibrary{patterns}

\tikzset{
  ashadow/.style={opacity=.25, shadow xshift=0.07, shadow yshift=-0.07},
}

\hypersetup{colorlinks=true,linkcolor={blue},citecolor={Maroon}}
           
\newenvironment{proof}{\paragraph{Proof:}}{\hfill$\square$}

\usepackage{bbm}
\usepackage{enumerate}
\usepackage{times}
\usepackage{centernot}
\usepackage{url}
\usepackage{cite}
\usepackage{amsfonts,mathrsfs}
\usepackage{amssymb,amsmath}
\usepackage{verbatim}
\usepackage{acronym}
\usepackage{mathtools}

\usepackage{graphicx}

\usepackage{hyperref}

\newcommand{\remove}[1]{}

\def\edi#1{{\color{black}#1}}

\def\fskip#1{}

\def\R{\mathbb{R}}
\def\E{\mathbb E}
\def\A{\mathcal A}
\def\B{\mathcal B}
\def\S{\mathcal S}
\def\variance{\text{Var}}
\def\F{\mathcal F}
\def\G{\mathcal G}

\def\Pr{\text{Pr}}
\def\x{\mathbf x}
\def\y{\mathbf y}
\def\X{\mathbf X}
\def\XX{\mathbb X}

\def\I{\mathcal I}
\def\Q{\mathbf Q}

\def\infect{\uparrow}
\def\recover{\downarrow}
\def\existence{\uparrow}
\def\nonexistence{\downarrow}
\def\on{_{\text{on} }}
\def\abcomp{_{(\overline{a,b})}}
\def\abagn{_{?(a,b)}}
\usepackage{amsmath}

\DeclarePairedDelimiterX{\infdivx}[2]{(}{)}{%
  #1\;\delimsize\|\;#2%
}

\usepackage[normalem]{ulem}

\def\Pr{\text{Pr}}
\def\N{\mathbb{N}}

\usepackage{algorithm}
\usepackage[noend]{algpseudocode}

\begin{document}

\title{Usefulness of the Age-Structured SIR Dynamics in Modelling COVID-19}

 \author{Rohit Parasnis}
 \author{Ryosuke Kato}
 \author{Amol Sakhale}
 \author{Massimo Franceschetti}
 \author{Behrouz Touri\thanks{Email: rparasni,rkato,asakhale,mfranceschetti,btouri@ucsd.edu}} 
 \affil{Department of Electrical and Computer Engineering, University of California San Diego}
\date{}

\maketitle

\begin{abstract}
We examine the age-structured SIR model, a variant of the classical Susceptible-Infected-Recovered (SIR) model of epidemic propagation, in the context of COVID-19. In doing so, we provide a theoretical basis for the model, perform an empirical validation, and discover the limitations of the model in approximating arbitrary epidemics. We first establish the differential equations defining the age-structured SIR model as the mean-field limits of a continuous-time Markov process that models epidemic spreading on a social network involving random, asynchronous interactions. We then show that, as the population size grows, the infection rate for any pair of age groups converges to its mean-field limit if and only if the edge update rate of the network approaches infinity, and we show how the rate of mean-field convergence depends on the edge update rate. We then propose a system identification method for parameter estimation of the bilinear ODEs of our model, and we test the model performance on a Japanese COVID-19 dataset by generating the trajectories of the age-wise numbers of infected individuals in the prefecture of Tokyo for a period of over 365 days. In the process, we also develop an algorithm to identify the different \textit{phases} of the pandemic, each phase being associated with a unique set of contact rates. Our results show a good agreement between the generated trajectories and the observed ones.
\end{abstract}

\section{INTRODUCTION}
The global COVID-19 death toll has crossed  6 million \cite{worldometer2020covid}, and it is no surprise that researchers all over the world have been forecasting the evolution of this pandemic to propose control policies aimed at minimizing its medical and economic impacts~\cite{ferguson2020report,alvarez2020simple,acemoglu2020optimal,acemoglu2020multi,chatterjee2020healthcare,pare2021multi}. Their efforts have typically relied on classical epidemiological models or their variants (for an overview see~\cite{bullo2020current} and the references therein). 
One such classical epidemic model is the   Susceptible-Infected-Recovered (SIR) model. 
Proposed in~\cite{kermack1927contribution}, the SIR model is a compartmental model in which every individual belongs to one of three possible states at any given time instant: the \textit{susceptible} state, 
the \textit{infected} state,  
and the \textit{recovered} state. 
The continuous-time SIR dynamics models the time-evolution of the fraction of individuals in any of these states using a set of ordinary differential equations (ODEs) parameterized by two quantities: the {infection rate} (the rate at which a given infected individual infects a given susceptible individual) and the recovery rate of infected individuals.

Even though the continuous-time SIR model is a deterministic model, it models an inherently random phenomenon in a large (but discrete) population. To bridge between the deterministic continuous-time SIR model and the underlying random processes over a finite population, researchers have shown that the associated (continuous-time) ODEs are the mean-field limits of continuous-time Markovian epidemic processes over a finite population~\cite{kurtz1970solutions,benaim2008class}. Similar results have been obtained for variants of the original model, such as for the SIR dynamics on a configuration model network~\cite{volz2008sir,decreusefond2012large}. These results theoretically justify the SIR model ODEs. 

Classical SIR models, however, (continuous and discrete-time) are homogeneous -- the same infection and recovery rates apply to the whole social network despite differences in the individuals' age, gender, race, immunity level, and pre-existing medical conditions. For COVID-19, this assumption is  inconsistent with studies showing that the contact rates between individuals and the recovery rates of infected individuals depend on factors such as age and location~\cite{mossong2008social,prem2017projecting,klepac2020contacts,voinsky2020effects}. In addition,~\cite{ellison2020implications} argues that homogeneous models can introduce significant biases in forecasting the epidemic, including overestimation of the number of infections required to achieve herd immunity, overestimation of the strictness of optimal control policies, overestimation of the impact of policy relaxations, and incorrect estimation of the time of onset of the pandemic.

We therefore need to shift our focus to variants of the classical SIR model with heterogeneous contact rates. Examples include the multi-risk SIR model~\cite{acemoglu2020multi} and the age-stratified SIR models considered in~\cite{acemoglu2021optimal,baqaee2020reopening,rampini2020sequential}, in which the population is partitioned into multiple groups and the rates of infection and recovery vary across groups. See~\cite{ellison2020implications} for a survey of these papers.

 However, the models considered in the above works have two main shortcomings. On the one hand, barring exceptions such as~\cite{favero2020restarting}, 
they are typically not validated using real data. On the other hand, they do not have a strong theoretical foundation because the dynamical processes studied in these works have not been established as the mean field limits of stochastic epidemic processes evolving on time-varying random graphs. We emphasize that even the convergence results obtained for homogeneous SIR models~\cite{kurtz1970solutions,benaim2008class,volz2008sir,decreusefond2012large} make the unrealistic assumption that the network of physical contacts (in-person interactions) existing in the  population is time-invariant. As such, we cannot justify the use of these models in designing optimal control policies aimed at minimizing the impact of any epidemic. We therefore address the aforementioned shortcomings using the \textit{age-structured} SIR model, a multi-group SIR model that partitions the population of a given region into different age groups and assigns different infection rates and recovery rates to the age groups. We note that, although we adopt the term \textit{age-structured} in our paper, our analysis also applies to populations  partitioned on the basis of differences in geographical location, sex, immunity level, etc.  Moreover, among existing heterogeneous models~\cite{ellison2020implications}, the age-structured SIR model is the simplest and hence more mathematically and computationally tractable than other models.



 The contributions of this paper are as follows:
 \begin{enumerate}
 \item \textit{\textbf{Modeling}:}  We extend our previously proposed stochastic epidemic model \cite{parasnis2021case} to a more general model that incorporates (a) a random and time-varying network of physical contacts (in-person interactions between pairs of individuals) that are updated asynchronously and at random times, (b) random transmissions of disease-causing pathogens from infected individuals to their susceptible neighbors, and (c) recoveries of infected individuals that occur at random times. 
     We analyze the resulting dynamics and show that under certain independence assumptions, the expected trajectories of the fractions of susceptible/infected/recovered individuals in any age group converge in mean-square to the solutions of the age-structured SIR ODEs as the population size  goes to $\infty$.
 \item \textit{\textbf{Convergence Rate Analysis}:}  We derive a lower bound on the effective infection rate for a given pair of age groups  in the stochastic model. This bound, as we show, is approximately linear in the reciprocal of the network update rate, which leads to the infection rate converging to its limit (specified by the ODEs) as fast as the reciprocal of the network update rate vanishes.
\item \textit{\textbf{Validation}:} We validate our age-structured model empirically by estimating the parameters of our model using a Japanese COVID-19 dataset and, subsequently, by generating the age-wise numbers of infected individuals as functions of time. In this process, we leverage the crucial fact that
the ODEs defining our model are linear in the model parameters (transmission and recovery rates), which enables us to use a least-squares method for the system identification.
\item \textit{\textbf{A Method to Detect Changes in Social Behavior}:} We design a simple algorithm that can be used to detect changes in social behavior throughout the duration of the pandemic. Given the age-wise daily infection counts, the algorithm estimates the dates around which the inter-age-group contact rates change significantly.
\item \textit{\textbf{Insights into Epidemic Spreading}:} We interpret the results of our phase detection algorithm to identify the least and the most infectious age groups and the least and the most vulnerable age groups. 
Additionally, we analyze the data for the entire period from March 2020 to April 2021 to explain how certain social events influenced the propagation of COVID-19 in the prefecture of Tokyo.
 \end{enumerate}

The structure of our paper is as follows: We introduce the age-structured SIR model and our stochastic epidemic model in Section~\ref{sec:formulation}. We establish the age-structured SIR ODEs as the mean-field limits of our stochastic model in Section~\ref{sec:main_result}. We also discuss the limitations of (converse result for) our model in Section~\ref{sec:main_result}. Next, we describe the empirical validation of our model (in the context of the COVID-19 outbreak in Tokyo) in Section~\ref{sec:empirical_validation}. We conclude with a brief summary and future directions in Section~\ref{sec:conclusion}.

\textit{Related Works:}~\cite{tkachenko2021time} proposes a heterogeneous epidemic model with time-varying parameters to show that heterogeneous susceptibility to infection results in a temporary weakening of the COVID-19 pandemic but not in herd immunity. The model is validated using the death tolls (and not the case numbers) reported for New York and Chicago for a period of about 80 days.~\cite{baqaee2020reopening} uses the age-structured SEIQRD model to predict the number of deaths with a reasonable accuracy, but unlike our work, it does not use the proposed model to generate the number of new cases as a function of time.~\cite{dolbeault2020heterogeneous} uses heterogeneous variants of the SEIR model to study the impact of the lockdown policy implemented in France, but it does not validate these models empirically.~\cite{klepac2020contacts} reports contact rate matrices for the population of the UK based on the self-reported data of 36,000 volunteers. However, the study ignores the time-varying nature of these contact rates, which we capture in our phase detection algorithm (Section~\ref{sec:empirical_validation}). Another study that uses time-invariant model parameters is~\cite{contreras2020multi}, which proposes the age-structured SEIRA model and uses it to simulate the number of new infections in different social groups of Chile.

~\cite{viguerie2021simulating} uses a heterogeneous SEIRD model to predict the effects of various relaxation policies on infection counts in certain regions of Italy. The model therein is empirically validated only using the data obtained during the first 60 days of the pandemic. In~\cite{rampini2020sequential}, the authors propose an age-structured SIRD model and calibrate it with the data obtained from~\cite{ferguson2020report}. Unlike our paper, however, \cite{rampini2020sequential} divides the population into only two age groups, and does not compare the model-generated values of the number of infections with the official case counts. Two other studies that use two-age-group SIR models are~\cite{harris2020data} and~\cite{giagheddu2020macroeconomics}. While~\cite{harris2020data} argues that in Florida, old and socially inert adults have been possibly infected by the young,~\cite{giagheddu2020macroeconomics} argues that age-group-targeted policies are more effective than uniform policies in reducing the economic impact of COVID-19.~\cite{janiak2021covid} proposes a heterogeneous SIR model with feedback and forecasts the economic and medical impacts of various policies aimed at controlling the pandemic in Chile. Unlike our study, however,~\cite{janiak2021covid} ignores the time-varying nature of contact rates.~\cite{favero2020restarting} proposes the SEIR-HC-SEC-AGE model, a heterogeneous SEIR model that sub-divides each age-group further into risk sectors with different vulnerabilities to the SARS-CoV-2 virus. The model therein, which is calibrated to predict the effects of different lockdown policies in certain regions of Italy, simulates the time-evolution of the observed death toll with a high accuracy. By contrast, we pick a much simpler heterogeneous model and examine whether it fits the observed case numbers well.~\cite{acemoglu2021optimal} and \cite{acemoglu2020multi} use an age-structured SIR model to show that control policies that target different age groups differently perform better than uniform policies. However, these results assume that inter-age-group contact rates are the same for all pairs of age groups, an assumption that is inconsistent with our empirical results (Section~\ref{sec:empirical_validation}). Hence, deriving optimal policies in the framework of the age-structured SIR model under more general assumptions is an important open problem. 

\textit{Notation: } We let $\N$ denote the set of natural numbers and  $\N_0:=\N\cup\{0\}$. We let $[l]:=\{1,2,\ldots, \ell\}$ for $\ell\in\N$. We  denote the set of real and positive real numbers by $\R$ and $\R_+$, respectively. For $x\in\R$, we let $x_+:=\max\{x,0\}$ denote the positive part of $x$.

The symbols $t$ and $k$ are used as a  continuous-time and discrete-time indices, respectively. We use the notation $z(t)$ for functions $z:\R_+\cup\{0\}\to \R$ and $z[k]$ for functions $z:\N\to \R$. We occasionally omit the time index $(t)$ when the value of $t$ is clear from the context.

We use the Bachmann-Landau asymptotic notation $O(f(n))$ for a given function $f:\N\rightarrow\R$ in the context of $n\rightarrow \infty$. We use $o(\Delta t)$ in the context of $\Delta t\rightarrow 0$. \edi{In addition, for a given function $g:[0,\infty)\rightarrow \R$, we use the notation $g'=g'(t)$ to denote $\frac{dg}{dt}$, the first derivative of $g$ with respect to time. }

For a set $\mathcal S$, we let $|\mathcal S|$ denote the cardinality of $\mathcal S$. In this paper, all random events and random variables are with respect to a  probability space $(\Omega,\F,\text{Pr})$, where $\Omega$ is the sample space, $\F$ is the set of events, and $\text{Pr}(\cdot)$ is the probability measure on this space. We denote random variables and random events using capital letters, and for a random event $C$, we define $1_C$  to be the indicator random variable associated with $C$, i.e, $1_C:\Omega\to\R$ is the random variable with $1_C(\omega)=1$ if $\omega\in C$ and $1_C(\omega)=0$, otherwise. For an event $C\in \F$, $\bar C$ represents the complement of $C$. For a random variable $X$, 
$\E[X]$ denotes the expected value of $X$ and $\E[X\mid  C]$ denotes the conditional expectation of $X$ given the event $C$. For random variables $X$ and $Y$ and a random event $C$, we define
$$
   \E[X\mid Y,C] = \frac{\E[X1_C\mid Y]}{\E[1_C\mid Y]}. 
$$
Therefore, for an event $F\in \F$
$$
    \Pr(F\mid Y, C) = \E[1_F\mid Y, C] = \frac{ \E[1_{F\cap C} \mid Y] }{\E[ 1_C\mid Y ] }.
$$

We denote tuples of length $r>1$ using bold-face letters and random tuples using bold-face capital letters. For a tuple $\x$ of length $r\in\N$ and an index $\ell\in[r]$, we let $x_\ell = (\x)_\ell$ denote the $\ell$-th entry of $\x$.

For $n\in \N$ and $E\subset [n]\times [n]$, we use $G=([n], E)$ to denote the directed graph (digraph) with vertex set $[n]$ and edge set $E$. Finally, for a graph $G=([n], E)$, given two distinct nodes $a,b\in [n]$, we let $\langle a, b\rangle := (a-1)(n-1) + b - \chi_{b-a}$, where $\chi_{\alpha}=1$ if $\alpha>0$ and $\chi_\alpha=0$, otherwise. {Note that $\langle \cdot, \cdot\rangle$ maps the edges between (distinct) nodes of the graph to the numbers $1,\ldots, n^2-n$ in lexicographic order.}
\section{PROBLEM FORMULATION}~\label{sec:formulation}
We now introduce two epidemic models, of which the first describes a deterministic dynamical system and the second describes a stochastic process on a finite population. One of the main objectives of this work is to relate these  models, which is achieved in Section~\ref{sec:main_result}. 

\subsection{The Age-Structured SIR Model}

Consider a population of individuals spanning $m$ age groups\footnote{As mentioned before, throughout this paper, we could generalize the discussions involving \textit{age groups} to \text{subpopulations} distinguished by  geographical locations, pre-existing health conditions, sex, etc.}. Suppose a part of this population contracts a communicable disease at time $t=0$. Let $s_i(t), \beta_i(t)$, and $r_i(t)$ denote, respectively, the fractions of susceptible, infected, and recovered individuals in the $i$-th age group at (a continuous) time $t\geq 0$, so that $s_i(t) + \beta_i(t) + r_i(t)$ equals the fraction of individuals in the $i$-th age group for all $t\geq 0$. As the disease spreads across the population, susceptible individuals get infected, and infected individuals recover in accordance with the system of ODEs given by
\begin{align} \label{eq:age_struct}
\dot{s}_{i}(t)&=-s_{i}(t) \sum_{j=1}^{m} A_{i j} \beta_{j}(t), \cr
\dot{\beta}_{i}(t)&=s_{i}(t) \sum_{j=1}^{m} A_{i j} \beta_{j}(t)-\gamma_{i} \beta_{i}(t), \\
\dot{r}_{i}(t)&=\gamma_{i} \beta(t), \nonumber
\end{align}
where for each $i,j\in [m]$, the constant $A_{ij}$ represents the rate of infection transmission  from an individual in age group $j$ to an individual in age group $i$, and $\gamma_i$ denotes the recovery rate of an infected individual in age group $i$. Hereafter, we refer to $A_{ij}$ as the \textit{contact rate of age group $j$ with age group $i$}. Note that the third equation in~\eqref{eq:age_struct} can be obtained from the first two equations simply by using the fact that $\dot s_i(t) + \dot\beta_i(t) + \dot r_i(t)=0$ for all $t\geq 0$. Also, if $m=1$, the above model reduces to the classical  (homogeneous and continuous-time) SIR model.

\subsection{A Stochastic Epidemic Model}
Let us now define a continuous-time Markov chain that describes an age-structured process of epidemic spreading occurring over a finite (atomic) population composed of individuals that are connected through a random, time-varying network $G(t)$.
 \subsubsection{Age Groups} Let $n\in \N$ denote the total population size, and let $[n]$ be the vertex set of the time-varying graph $G(t)$, so that the vertex set indexes all the individuals/nodes in the network. We assume that $[n]$ is partitioned into $m$ age groups $\{\A_i\}_{i=1}^m$ and that $|\A_i|$ (the number of individuals in the $i$-th age group) scales linearly with $n$ for all $i\in [m]$. In the following, $i,j\in [m]$ are generic age group indices.
 \subsubsection{State Space} The state space of our random process is the space $\mathbb S= \{-1,0,1\}^n\times \{0,1\}^{2n(n-1)}$. The network state is a tuple $\mathbf x=(x_1,x_2,\ldots, x_{2n^2-n}) \in \mathbb S$, where 
 \begin{enumerate}[(i)]
     \item  $\{x_\ell\}_{\ell\in [n]}$ denotes the disease states of the nodes in the network, i.e., for $\ell\in [n]$, we set $x_\ell=0$, $1$, or $-1$ accordingly as node $\ell$ is susceptible, infected, or recovered, respectively. 
     \item For $\ell\in \{n+1, n+2,\ldots, n^2\}$, we let $x_\ell$ denote the \textit{edge state} of the $\ell$-th pair in the following {lexicographic} order of pairs of distinct nodes: $(1,2), \ldots, (1, n), (2, 1), \ldots, (2, n), \ldots, (n, 1), \ldots, (n, n-1)$. In other words, for any node pair $(a,b)\in [n]\times [n]$ such that $a\neq b$, we set $x_{ \langle a,b\rangle }=1$ if there is a directed edge from $b$ to $a$ in the network $G$, and $x_{\langle a, b\rangle  }=0$, otherwise.  
     For notational convenience, we let $1_{(a,b)}(\x):= x_{ \langle a,b\rangle  }$.
     \item For $\ell\in\{n^2+1, \ldots, 2n^2 - n\}$, we let $x_\ell$ be a binary variable whose value flips (becomes $1-x_\ell$) whenever the $(\ell - n^2)$-th edge state gets updated (re-initialized). However, the direction of this flip (whether $x_{\ell}$ changes from 0 to 1 or from 1 to 0) carries no significance.
 \end{enumerate}
 \subsubsection{State Attributes} For all $\x\in\mathbb S$, we let $\S_i(\x):=\{a\in \A_i: x_a=0\}$, $\I_i(\x):=\{a\in \A_i: x_a=1\}$, \edi{and $\mathcal R_i(\x):=\{a\in\A_i: x_a=-1\}$} denote, respectively, the set of susceptible individuals, the set of infected individuals\edi{, and the set of recovered individuals }in $\A_i$ given that the network state is $\x$. 
 We let $\S(\x):=\cup_{i=1}^m\S_i(\x)$ and $\I(\x):=\cup_{i=1}^m\I_i(\x)$.
 Additionally, for every node $a\in [n]$, we let $E_{j}^{(a)}(\x):=\sum_{c\in\I_j(\x)}1_{(a,c)}(\x)$ be the number of arcs from $\I_j(\x)$ to $a$. 
 \subsubsection{The Markov Process}  Let $\X(t)\in\mathbb S$ denote the state of the network at any time $t\geq 0$. Then {we assume} that $\{\X(t):t\geq 0\}$ is a right-continuous time-homogeneous Markov process in which every transition from a state $\x \in \mathbb S$ to a state $\y\in\mathbb S\setminus\{\x\}$  belongs to one of the following categories:
 \begin{enumerate}
     \item \textit{Infection transition:}  This occurs when a node $a\in\S_i(\x)$ gets infected by a node in $\cup_{k=1}^m \I_k(\x)$, while the disease states of all other nodes and the edge states of all the node pairs remain the same. In other words, $x_a=0$, $y_a=1$, and $x_\ell=y_\ell$ for all $\ell\neq a$. Denoting the state-independent rate of pathogen transmission from a node in $\I_k(\x)$ to an adjacent node in $\S_i(\x)$ by $B_{ik}$, we note that the rate of infection transmission from any node $c\in\I_k(\x)$ to $a$ is $B_{ik}1_{(a,c)}(\x)$. Hence, the total rate at which $a$ receives pathogens from $\I_k$ is $\sum_{c\in\I_k(\x)} B_{ik}1_{(a,c)}(\x)=B_{ik}E_k^{(a)}(\x)$, assuming that different edges transmit the infection independently of each other during vanishingly small time intervals. As a result, the effective rate at which $a$ gets infected is $\sum_{k=1}^m B_{ik}E_k^{(a)}(\x)$. We denote the successor state $\y$ of $\x$, where the node $a$ turns from susceptible to infected, by $\x_{\infect a}$.
     \item \textit{Recovery transition:} This occurs when a node $a\in \I_i(\x)$ recovers, i.e., $x_a=1$, $y_a=-1$, and $x_\ell=y_\ell$ for all $\ell\neq a$. We let $\gamma_i$ denote the rate at which an infected node in $\A_i$ (such as $a$) recovers. For such a transition, we denote $\y=\x_{\recover a}$.
     \item \textit{Edge update transition:} This occurs when $x_{\langle a,b\rangle}$, the edge state of \edi{a node pair $(a,b)\in\A_i\times \A_j$}, is updated or re-initialized, i.e., $y_{n^2 + \langle a, b\rangle } = 1-x_{n^2 + \langle a, b\rangle }$, and $y_\ell=x_\ell$ for all $\ell\notin\{ \langle a,b\rangle, n^2 + \langle a, b\rangle \}$. We let $\lambda$ denote the \textit{edge update rate} or the rate at which an edge state is updated. In addition, given that the edge state of $(a,b)$ is updated at time $t\geq 0$, the probability that $1_{(a,b)}(t)=1$ (i.e., the edge $(a,b)$ exists after the re-initialization) equals $\frac{\rho_{ij}}{n}$, where $\rho_{ij}>0$ is constant in time. 
     Therefore, if $y_{\langle a,b\rangle}=1$ (meaning that $(a,b)$ exists as an arc in $G$ in the network state $\y$), then the rate of transition from $\x$ to $\y$ equals $\lambda\frac{\rho_{ij} }{n}$, whereas  if $y_{\langle a,b\rangle}=0$, then the rate of transition from $\x$ to $\y$ equals $\lambda\left(1-\frac{\rho_{ij} }{n} \right)$. In the former case, we write $\y = \x_{\existence (a,b)}$, while in the latter case, we write $\y = \x_{\nonexistence(a,b)}$. Note that the rate of transition from $\x$ to $\x_{\nonexistence(a,b)}$ or $\x_{\existence(a,b)}$ does not depend on $\x$.
     
     The edge update transition of $(a,b)$ can be described informally as follows. Throughout the evolution of the pandemic, $a$ and $b$ decide whether or not to meet each other at a constant rate $\lambda>0$, i.e., their decision times $\{T_\ell^{(a,b)}\}_{\ell=1}^\infty$ form a Poisson  process with rate $\lambda$.
     Each time they make such a decision, they decide to interact with probability $\frac{\rho_{ij}}{n}$, and they decide not to interact with probability $1-\frac{\rho_{ij}}{n}$, independently of their past decisions. The probability of interaction is assumed to scale inversely with $n$ so that the mean degree of every node is constant with respect to $n$. 
 \end{enumerate}
 To summarize, the rate of transition from any state $\x\in\mathbb S$ to any state $\y\in\mathbb S\setminus\{\x\}$ is given by $\Q$, the infinitesimal generator of the Markov chain $\{\X(t):t\geq 0\}$, where \edi{for $\x\not=\y$}
 \begin{align*}
     \Q(\x,\y) :=
     \begin{cases}
      \sum_{k=1}^m B_{ik }E_k^{(a)} (\x) \, &\text{if }\y = \x_{\infect a}\text{ for some } a\in\S_i(\x),i\in [m] \\
      \gamma_i \quad &\text{if }\y = \x_{\recover a}\text{ for some }a\in\I_i(\x),i\in[m] \\
      \lambda \frac{\rho_{ij}}{n} \quad &\text{if } \y = \x_{\existence (a,b)} \text{ for some } (a,b)\in\A_i\times \A_j, i,j\in[m]\\
      \lambda\left(1- \frac{\rho_{ij}}{n}\right) \quad &\text{if }\y = \x_{\nonexistence (a,b), }\text{ for some }(a,b)\in\A_i\times \A_j,i,j\in [m] \\
      0 \quad &\text{otherwise}
     \end{cases},
 \end{align*}  
and $\Q(\x,\x) := -\sum_{\y\in\mathbb S\setminus\{\x\}} \Q(\x,\y)$. In addition, we say that $\y$ \textit{succeeds} $\x$ \textit{potentially} iff $\Q(\x,\y)>0$.
\section{MAIN RESULT}\label{sec:main_result}

 To provide a rigorous mean-field derivation of the dynamics~\eqref{eq:age_struct}, we now consider a \textit{sequence} of social networks with increasing population sizes such that each network obeys the theoretical framework described in Section~\ref{sec:formulation}.
 Given a network from this sequence with  population size $n\in\N$, \edi{we let $\S^{(n)}_j(t):=\S_j(\X(t))$, $\I^{(n)}_j(t):=\I_j(\X(t))$, and $\mathcal R^{(n)}_j(t):=\mathcal R_j(\X(t))$ denote the (random) sets of infected, susceptible, and infected individuals in the $j$-th age group, respectively}, and we let $s_j^{(n)}(t):=\frac{1}{n}|\S_j^{(n)}(t)|$, $\beta_j^{(n)}(t):=\frac{1}{n}|\I_j^{(n)}(t)|$ and $r_j^{(n)}(t):=\frac{1}{n}|\mathcal R_j^{(n)}(t)|$ denote the fractions of susceptible, infected, and recovered individuals in the $j$-th age group, respectively. As for the absolute numbers, we let $S_j^{(n)}(t):=|\S_j^{(n)}(t)|$, $I_j^{(n)}(t):=|\I_j^{(n)}(t)|$, and $R_j^{(n)}(t):=|\mathcal R_j^{(n)}(t)|$. Additionally, we let $E^{(n)}(t)$ denote the edge set of the network at time $t$, and we drop the superscript $^{(n)}$ when the context makes our reference to the $n$-th network clear.
 
 
 Another quantity that varies with $n$ is $\lambda^{(n)}$, the edge update rate. To obtain the desired mean-field limit in Theorem~\ref{thm:main}, we assume that $\lambda^{(n)} \rightarrow \infty$ as $n\rightarrow\infty$. To interpret this assumption, consider any pair of individuals $(a,b)\in \A_i\times \A_j$ that are in contact with each other at time $t\geq0$ during the epidemic. Since the edge state of $(a,b)$ is updated to $0$ (the state of non-existence) at a time-invariant rate of $\lambda^{(n)}\left(1 - \frac{\rho_{ij} }{n} \right)$, the assumption implies that the mean interaction time of $b$ with $a$, which is $\frac{1}{\lambda^{(n)}} + O\left(\frac{1}{n\lambda^{(n)} }\right)$, vanishes as the population size increases. This is a possible real-world scenario, because as $n$ increases, the population density of the given geographical region increases, which could result in overcrowding and rapidly changing interaction patterns in the network. This may be especially true in the case of public places such as supermarkets and subway stations at a time when the society is already aware of an evolving epidemic. Another implication  of $\lim_{n\rightarrow\infty}\lambda^{(n)}=\infty$ is that the rate at which a given infected node contacts and transmits pathogens to a given susceptible node vanishes as the population size goes to $\infty$ (see Remark~\ref{rem:only} for an explanation). This implication is weaker than the often-assumed condition that the rate of pathogen transmission is proportional to the reciprocal of the population size~\cite{armbruster2017elementary,simon2013exact}.
 
 We are now ready to state our main result. Its proof is based on the theory of continuous-time Markov chains and an analysis of how the disease propagation process is affected by random updates occurring in the network structure at random times  (which results in Propositions~\ref{prop:basic} and~\ref{prop:bounds}) in addition to the proof techniques used in~\cite{armbruster2017elementary}. 
 The proofs of all these results are available in the appendix. 

\begin{theorem}\label{thm:main}

Suppose that $\lim_{n\rightarrow\infty}\lambda^{(n)} = \infty$ and that for every $i\in [m]$, there exist $s_{i,0}, \beta_{i,0}\in [0,1]$ such that $\lim_{n\to\infty} s_i^{(n)}(0)= s_{i,0}$ and $\lim_{n\rightarrow\infty}\beta_i^{(n)}(0)=\beta_{i,0}$. Then for each $i\in [m]$,
$$
    \edi{\lim_{n\to\infty}} \mathbb{E}\left[\left\|\left(s_{i}^{(n)}(t), \beta_{i}^{(n)}(t)\right)-\left(y_i(t), w_i(t)\right)\right\|_2^{2}\right] = 0.
$$
on any finite time interval $[0, T_0]$, where $(y_i(t),w_i(t))$ is the solution to \edi{the ODE system given by the first two equations in~\eqref{eq:age_struct}, i.e., $(y_i(t),w_i(t))$ satisfies}
\begin{enumerate}[(I).]
    \item  \label{eq:main_1} $\quad \dot y_{i}=- y_i\sum_{j=1}^m  A_{ij} w_j , \quad y_{i}(0)=s_{i,0}$,
    \item \label{eq:main_2} $\quad \dot w_{i}= y_i\sum_{j=1}^m  A_{ij} w_j -\gamma_i w_i, \quad w_i(0)=\beta_{i,0}$,
\end{enumerate}
and $A\in \R^{m\times m}$ is defined by $A_{ij}:=\rho_{ij}B_{ij}$.
\end{theorem}

  Theorem~\ref{thm:main} relies on the following   proposition. 
 
\begin{proposition}\label{prop:basic}
For each $i,j\in [m]$, let
$$
    \chi_{ij}=\chi_{ij}(t, \S, \I):=\E[1_{(a,b)}(t)\mid \S(t), \I(t)]=\Pr((a,b)\in E(t)\mid \S(t), \I(t))
$$ 
be the \edi{random variable} that denotes
the conditional probability that a pair of nodes $(a,b)\in \S_i(t)\times \I_j(t)$ are in physical contact at time $t$ given the state of the network at time $t$. Then the following equations hold for all $t\geq 0$:
\begin{enumerate}[(i).]
    \item  \label{item:init_1} $\E[s_i]'=-\sum_{j=1}^m B_{ij}\E[n\chi_{ij} s_i \beta_j]$,
    \item  \label{item:init_2}
    $\E[\beta_i]' = \sum_{j=1}^m  B_{ij}\E[n\chi_{ij} s_i\beta_j] - \gamma_i\E[\beta_i]$,
    \item  \label{item:init_3}
    $\E[s_i^2]'=-\sum_{j=1}^m \big( 2 B_{ij}\E[n\chi_{ij} s_i^2 \beta_j ]-  B_{ij} \E[n\chi_{ij} s_i \beta_j]/n \big) $,
    \item  \label{item:init_4} $\E[\beta_i^2]' = \sum_{j=1}^m  B_{ij} (  2\E[n\chi_{ij} s_i\beta_j\beta_i]+\E[n\chi_{ij} s_i\beta_j] / n ) -\gamma_i\left( 2\E[\beta_i^2] - \E[\beta_i]/n \right) $.
\end{enumerate}
\end{proposition}

We point out that if $\chi_{ij}= \frac{\rho_{ij}}{n}$ then Equations~\eqref{item:init_1} and~\eqref{item:init_2} have the same coefficients as~\eqref{eq:main_1} and~\eqref{eq:main_2}. It is then natural to ask: how does  the conditional edge probability $\chi_{ij}$ compare to the unconditional edge probability $\frac{\rho_{ij} }{n}$? The following proposition provides an answer. As we show in Remark~\ref{rem:only}, our answer helps characterize the rate at which the infection transmission rates converge to their respective limits, an analysis missing from other works such as~\cite{armbruster2017elementary} and~\cite{simon2013exact}.
\begin{proposition}\label{prop:bounds}
For all $t\geq 0$, $n\in\N$ and $i,j\in [m]$, 
$$
    \frac{\rho_{ij} }{n} \left(1-\frac{B_{ij} }{\lambda^{(n)} }\left(1-e^{-\lambda^{(n)} t}\right) \right)  \leq \chi_{ij} \leq \frac{ \rho_{ij} }{n}.
$$
\end{proposition}

\begin{remark}\label{rem:only} 
Given $(\S(t), \I(t))$, note that the conditional probability that a given infected node in $\A_j$ infects a given susceptible node in $\A_i$
 during a time interval $[t, t+\Delta t)$ is
$
    B_{ij}\chi_{ij}(t, \S,\I)\Delta t + o(\Delta t).$
In light of Proposition~\ref{prop:bounds}, this means that the associated conditional infection \textit{rate} $B_{ij}\chi_{ij}(t,\S,\I)$ belongs to the interval 
$$
    \left[\frac{1}{n}A_{ij}\left(1-\frac{B_{ij} }{\lambda^{(n)} }(1-e^{-\lambda^{(n)} t}) \right), \frac{1}{n}A_{ij}\right].
$$
On taking expectations, we realize that the same applies to the associated unconditional infection rate as well. Hence, the total rate of infection transmission from all of $\A_j$ to any given node of $\A_i$ is at least $I_j^{(n) }(t)\times\frac{1}{n}A_{ij} \left(1-\frac{B_{ij} }{\lambda^{(n)} }(1-e^{-\lambda^{(n)} t}) \right) = A_{ij} \beta_j^{(n)} (t) \left(1-\frac{B_{ij} }{\lambda^{(n)} }(1-e^{-\lambda^{(n)} t}) \right) $ and at most $A_{ij}\beta_j^{(n)}(t)$. Since we assume $\lim_{n\rightarrow\infty}\lambda^{(n)}=\infty$, this further implies that the concerned rate is approximately $A_{ij}\beta_j^{(n)}(t)$ for large $n$, thereby giving us an interpretation of the `contact rate' $A_{ij}$ as a normalized infection rate. That is, in the limit as $n\rightarrow\infty$, the matrix $A$ quantifies the infection transmission rates between any two age groups \textit{relative} to the level of infectedness (fraction of infected persons) of the transmitting age group. Moreover, Proposition~\ref{prop:bounds} also implies that the difference between the age-wise infection transmission rates and their respective mean-field limits (which exist as per Theorem~\ref{thm:main}) is $O\left(\frac{1}{\lambda^{(n)}}\right)$.
\end{remark}

\section{A CONVERSE RESULT}

The purpose of this section is to argue that the age-structured SIR dynamics does not model an epidemic well if the infection rates $B_{ij}$ are high enough to be comparable to the edge update rate of the network.

\begin{theorem}\label{thm:converse}
Suppose $\lambda_\infty:=\lim_{n\rightarrow\infty}\lambda^{(n)} < \infty$ and that for every $p\in [m]$, there exist $s_{p,0}, \beta_{p,0}\in [0,1]$ such that $s_p^{(n)}(0)\rightarrow s_{p,0}$ and $\beta_i^{(n)}(0)\rightarrow \beta_{p,0}$ as $n\rightarrow\infty$. In addition, let $\{(y_q(t) ,w_q(t))\}_{q\in[m]}$ be the solutions of the ODEs~\eqref{eq:main_1} and~\eqref{eq:main_2}. Then, there exists no interval $[t_1,t_2]\subset[0,\infty)$ for which $\min_{p,q\in[m]}\min_{t\in[t_1,t_2] }y_p(t)w_q(t) >0$ and on which the pairs $\left\{\left(s_{q}^{(n)}(t), \beta_{q}^{(n)}(t)\right)\right\}_{q=1}^m$ uniformly converge in probability to the corresponding pairs in $\{(y_q(t),w_q(t))\}_{q=1}^m$. More precisely, for every interval $[t_1,t_2]\subset\R$ such that $y_p(t)>0$ and $w_p(t)>0$ for all $p\in[m]$ and $t\in[t_1,t_2]$, there exists a $q\in [m]$ and an $\varepsilon_q>0$ such that
$$
   \liminf_{n\rightarrow\infty}\sup_{t\in[t_1,t_2]} \Pr\left(\left\|\left(s_{q}^{(n)}(t), \beta_{q}^{(n)}(t)\right)-\left(y_q(t), w_q(t)\right)\right\|_2> \varepsilon_q\right) >0.
$$
\end{theorem}

\begin{proof}
    Suppose, on the contrary, that there exists a time interval $[t_1,t_2]\subset[0,\infty)$ such that $y_p(t)>0$ and $w_p(t)>0$ for all $p\in[m]$ and $t\in[t_1,t_2]$, and the following holds for all $q\in [m]$ and all $\varepsilon_q>0$:
    $$
        \liminf_{n\rightarrow\infty}\sup_{t\in[t_1,t_2]} \Pr\left(\left\|\left(s_{q}^{(n)}(t), \beta_{q}^{(n)}(t)\right)-\left(y_q(t), w_q(t)\right)\right\|_2> \varepsilon_q\right)=0,
    $$
    i.e., for a fixed $\varepsilon>0$, there exists a sequence $\{\pi(n)\}_{n=1}^\infty\subset\N$ such that 
    $$
        \lim_{n\rightarrow\infty}\sup_{t\in[t_1,t_2]} \Pr\left(\left\|\left(s_{q}^{(\pi( n))}(t), \beta_{q}^{(\pi( n))}(t)\right)-\left(y_q(t), w_q(t)\right)\right\|_2> \varepsilon\right)=0.
    $$
    We then arrive at a contradiction, as shown below. 
    
    
    We first choose an $\eta>2\left(1+\frac{A_{\max}  }{\lambda_\infty }\right)$ (where $A_{\max}:=\max\{A_{pq}:p,q\in[m]\}$) and a $t\in (t_1,t_2)$. 
    By our hypothesis and norm equivalence, for $\kappa_0:=\frac{1}{\eta\lambda_\infty}$ and for every $\delta>0$, there exists an $N_{\varepsilon,\delta}\in \N$ such that
    $$
        \Pr\left(\left\|\left(s_{q}^{(\pi(n))}\left(\tau\right), \beta_{q}^{(\pi(n))}\left(\tau\right)\right)-\left(y_q\left(\tau\right), w_q\left(\tau\right)\right)\right\|_1\leq  \varepsilon\right)\geq 1-\delta
    $$
    for all $n\geq N_{\varepsilon,\delta}$ and all $\tau\in [t-\kappa_0,t]$.

    We now define $\alpha_0:=\min_{p,q\in[m]}\min_{t\in[t_1,t_2]} y_p(t)w_q(t)$ and we let $a$ and $b$ be any two nodes such that $(a,b)\in\S_i(t)\times\I_j(t)$ for arbitrary $i,j\in[m]$. Additionally, we let $K:=t-\inf\{\tau\geq 0:b\in\I(\tau)\}$ denote the (random) time elapsed between the time at which $b$ gets infected and time $t$. We then have
    \begin{align}\label{eq:somewhat_long}
        \Pr\left(K\leq \kappa_0\mid \S(t),\I(t) \right)
        &= \Pr\left(b\text{ is infected during }[t-\kappa_0,t] \mid \S(t),\I(t) \right)\cr
        &\stackrel{(a)}\leq 1 -  e^{-A_{\max} \kappa_0}\cr
        &\stackrel{(b)}\leq A_{\max} \kappa_0\cr
        &=\frac{A_{\max}}{\eta\lambda_\infty},
    \end{align}
    where $(b)$ holds because $1-e^{-u}\leq u$ for all $u\geq 0$, and $(a)$ can be explained as follows: given $(\S(\tau),\I(\tau))$ and an infected node $c\in\I_q(\tau)$ for some time $\tau\in [t-\kappa_0,t)$, and given that $b\in\S_j(\tau)$, we know from Proposition~\ref{prop:bounds} that the conditional probability of the edge $(b,c)$ existing in the network at time $\tau$ is at most $\frac{\rho_{jq}}{\pi(n)}$. Also, as per the definition of our stochastic epidemic model, given that $(b,c)\in E(\tau)$ and given $(\S(\tau),\I(\tau))$ (and hence, also that $(b,c)\in\S_j(\tau)\times\I_q(\tau)$), the conditional rate of infection transmission  from $c$ to $b$ at time $\tau$ is $B_{jq}$. Hence, given $(\S(\tau),\I(\tau))$ (and hence, that $(b,c)\in\S_j(\tau)\times \I_q(\tau)$), the conditional rate of infection transmission from $c$ to $b$ is at most $B_{jq}\frac{\rho_{jq}}{\pi( n)}=\frac{A_{jq}}{\pi( n)}$. Under our modelling assumption that distinct edges transmit the infection independently of each other during vanishingly small time intervals, this means that, conditional on $\S(\tau)$ and $\I(\tau)$, the conditional total rate at which $b$ receives infection is at most
    $$
        \sum_{q\in[m]}\sum_{c\in\I_q(\tau)}\frac{A_{jq}}{\pi(n)}=\sum_{q\in[m]} |\I_q(\tau)|\frac{A_{jq}}{\pi(n)}=\sum_{q\in [m]} \beta_q^{(\pi(n))} (\tau)A_{jq}\leq A_{\max}\sum_{q\in[m]}\beta_{q}^{(\pi(n))}(\tau)\leq A_{\max}\cdot 1.
    $$
    Note that this upper bound is time-invariant and does not depend on $\S(\tau)$ or $\I(\tau)$ for any time $\tau$. It thus follows that, conditional on $(\S(t),\I(t))$, the rate at which $b$ gets infected is at most $A_{\max}$ throughout the interval $[t-\kappa_0,t)$ and hence, the probability that $b$ does not get infected during an interval of length $\kappa_0$ is at least $e^{-A_{\max} \kappa_0}$. This implies $(a)$.
    
    We now infer from~\eqref{eq:somewhat_long} that
    \begin{align}~\label{eq:not_long_at_all}
        \Pr\left(K\geq\kappa_0 \mid\S(t),\I(t) \right) \geq 1 - \frac{A_{\max}}{\eta\lambda_\infty}.
    \end{align}
    Next, we lower-bound $\Pr(T\geq \kappa_0\mid\S(t),\I(t))$. To this end, note from Proposition~\ref{prop:bounds} that $\Pr((a,b)\in E(t)\mid \S(t),\I(t))\leq \frac{\rho_{ij} }{\pi(n)}$. As a result, we have 
    \begin{align*}
        &|\Pr(T\geq \kappa_0\mid \S(t),\I(t))- \Pr(T\geq \kappa_0\mid (a,b)\notin E(t), \S(t),\I(t))|\cr
        &=|\Pr(T\geq \kappa_0\mid (a,b)\notin E(t), \S(t),\I(t))(1-\Pr((a,b)\in E(t)\mid \S(t),\I(t)))\cr
        &\quad+ \Pr(T\geq \kappa_0\mid (a,b)\in E(t), \S(t),\I(t))\cdot\Pr( (a,b)\in E(t) \mid \S(t),\I(t)) \cr
        &\quad- \Pr(T\geq \kappa_0\mid (a,b)\notin E(t),\S(t),\I(t))| \cr
        &\leq \frac{\rho_{ij} }{\pi(n)},
    \end{align*}
    which also means that 
    \begin{gather}\label{eq:not_long}
       | \Pr(T<\kappa_0\mid (\S(t),\I(t))) - \Pr(T<\kappa_0\mid  (a,b)\notin E(t), (\S(t),\I(t)))| 
       \leq \frac{\rho_{ij} }{\pi(n)}.
    \end{gather}
    Moreover,  for any realization $(\S_0,\I_0)$ of $(\S(t),\I(t))$, Remark~\ref{rem:should_be_the_last} asserts that
    \begin{align*}
        \Pr(T\leq \kappa_0\mid K=\kappa,(\S(t),\I(t))=(\S_0,\I_0),(a,b)\notin E(t))\leq 1-e^{-\lambda \kappa_0}
    \end{align*}
    for all $0\leq \kappa\leq t$. Since the right-hand-side above is independent of both $\kappa$ and $(\S_0,\I_0)$, it follows that
    \begin{align}\label{eq:is_it_last}
        \Pr(T\leq \kappa_0\mid (\S(t),\I(t)),(a,b)\notin E(t))\leq 1-e^{-\lambda \kappa_0}\leq \lambda\kappa_0.
    \end{align}
    Therefore, as a consequence of~\eqref{eq:not_long_at_all},~\eqref{eq:not_long},~\eqref{eq:is_it_last}, and the union bound, we have
    \begin{align}\label{eq:really_last}
        \Pr(T\geq \kappa_0, K\geq \kappa_0\mid \S(t),\I(t))&=1-\Pr(\{T<\kappa_0\}\cup\{K< \kappa_0\}\mid \S(t),\I(t))\cr
        &\geq 1 - \Pr(T<\kappa_0\mid \S(t),\I(t))-\Pr(K<\kappa_0\mid \S(t),\I(t))\cr
        &\geq  1 - \frac{A_{\max}}{\eta\lambda_\infty} - \frac{\lambda^{( \pi(n) )}}{\eta\lambda_\infty} - \frac{\rho_{ij} }{\pi(n)}.
    \end{align}
    This further yields,
    \begin{align}\label{eq:why}
        \chi_{ij}(t,\S,\I)&=\Pr((a,b)\in E(t)\mid \S(t),\I(t))\cr
        &=\Pr((a,b)\in E(t)\mid T\geq \kappa_0, K\geq \kappa_0, \S(t),\I(t))\cdot\Pr(T\geq \kappa_0, K\geq \kappa_0 \mid \S(t),\I(t)) \cr
        &\quad+ \Pr((a,b)\in E(t)\mid \{T< \kappa_0\}\cup\{K< \kappa_0\}, \S(t),\I(t))\cdot\Pr(\{T< \kappa_0\}\cup\{K< \kappa_0\}\mid  \S(t),\I(t))\cr
        &\leq \frac{\rho_{ij} }{\pi(n)}e^{-B_{ij}\kappa_0}\left(1 - \frac{A_{\max} }{\eta\lambda_\infty} - \frac{\lambda^{(\pi(n))}}{\eta\lambda_\infty} -\frac{\rho_{ij} }{\pi(n)} \right)+  \left( \frac{A_{\max}}{\eta\lambda_\infty} + \frac{\lambda^{(\pi(n))}}{\eta\lambda_\infty} + \frac{\rho_{ij} }{\pi(n)}\right)\frac{\rho_{ij} }{\pi(n)}\cr
        &=\frac{\rho_{ij} }{\pi(n)}e^{-\frac{B_{ij} }{\eta\lambda_\infty} }\left(1 - \frac{A_{\max} }{\eta\lambda_\infty} - \frac{\lambda^{(\pi(n))}}{\eta\lambda_\infty} -\frac{\rho_{ij} }{\pi(n)} \right)+  \left( \frac{A_{\max}}{\eta\lambda_\infty} + \frac{\lambda^{(\pi(n))}}{\eta\lambda_\infty} + \frac{\rho_{ij} }{\pi(n)}\right)\frac{\rho_{ij} }{\pi(n)}
    \end{align}
    where the inequality is a consequence of~\eqref{eq:really_last} and Remark~\ref{rem:for_converse_result}. Recall that $\lim_{n\rightarrow\infty}\lambda^{(n)}=\lambda_\infty$, which means that the right hand side can be made smaller than $(1+\varepsilon)\frac{\rho_{ij}}{\pi( n)}e^{-\frac{B_{ij} }{\eta\lambda_\infty}}$ by choosing $n$ large enough. Moreover~\eqref{eq:why} holds for all $t\in[t_1,t_2]$.
    
    Proposition~\ref{prop:basic} now implies that for all $t\in[t_1,t_2]$ and large enough $n$,
    \begin{align}\label{eq:last_never_comes}
        \E[s_i]'=-\sum_{j=1}^m B_{ij}\E[n\chi_{ij}s_i\beta_j] > -\sum_{j=1}^m A_{ij}(1+\varepsilon) e^{-\frac{B_{ij} }{\eta\lambda_\infty}}\E[ s_i\beta_j], 
    \end{align}
     Now, observe that for any $t\in[t_1,t_2]$, we have \begin{align}
        &\E[s_i(t)\beta_j(t)]\cr
        &\leq 1\cdot \Pr\left(\left\|\left(s_{q}^{(n)}\left(\tau\right), \beta_{q}^{(n)}\left(\tau\right)\right)-\left(y_q\left(\tau\right), w_q\left(\tau\right)\right)\right\|_1>  \varepsilon\right) \cr
        &\quad+ (w_i(t)+\varepsilon)(y_j(t)+\varepsilon) \cdot\Pr\left(\left\|\left(s_{q}^{(n)}\left(\tau\right),  \beta_{q}^{(n)}\left(\tau\right)\right)-\left(y_q\left(\tau\right), w_q\left(\tau\right)\right)\right\|_1\leq  \varepsilon\right)\cr
        &\leq \delta+ y_i(t)w_j(t)+  2\varepsilon + \varepsilon^2.
    \end{align}
    Therefore, assuming that $\delta$ and $\varepsilon$ are small enough to satisfy $\delta+2\varepsilon+\varepsilon^2<\alpha_0\left(\frac{e^{\frac{B_{ij} }{\eta\lambda_\infty}}  }{1+\varepsilon}-1\right)$,~\eqref{eq:last_never_comes} implies the existence of a constant $\varepsilon'>0$ such that :
    $$
        \E[s_i]'=-\sum_{j=1}^m B_{ij}\E[n\chi_{ij}s_i\beta_j] \geq -\sum_{j=1}^m A_{ij} y_i(t)w_j(t)+\varepsilon'=y_i'(t)+\varepsilon'.
    $$
    Since this holds for all $t\in[t_1,t_2]$, we have
    $$
        \E[s_i^{(\pi(n))}(t_2)]-y_i(t_2)\geq \E[s_i^{(\pi(n))}(t_1)] - y_i(t_1) + (t_2-t_1)\varepsilon'
    $$
    for all sufficiently large $n$. Here, we observe that $\{s_i^{(n)}(t_1),\beta_j^{(n)}(t_1):i,j\in[m],n\in\N\}$ are bounded by the constant function $1$, which is integrable with respect to probability measures. Therefore, $\{s_i^{(n)}(t_1):i,j\in[m],n\in\N\}$ are uniformly integrable. Since they converge in probability to $\{y_i(t_1):i,j\in[m]\}$ (by hypothesis), it follows by Vitali's Convergence Theorem that they also converge in $L^1$-norm. Thus,  $\E[s_i^{(n)}(t_1)]-y_i(t_1)\rightarrow 0$ as $n\rightarrow\infty$, thereby implying that 
    \begin{equation}\label{eq:not_align}
         \liminf_{n\rightarrow\infty} \left(\E[s_i^{( \pi( n))}(t_2)]-y_i(t_2)\right)\geq (t_2-t_1)\varepsilon'.
    \end{equation}
   
    On the other hand, Vitali's Convergence Theorem and our hypothesis also imply that $\E[s_i^{(n)}(t_2)]\rightarrow y_i(t_2)$ as $n\rightarrow\infty$, which contradicts~\eqref{eq:not_align}. Hence, our hypothesis that $\{s_q^{(n)}(t),\beta_{q}^{(n)}(t)\}_{q\in[m]}$  converge in probability to the solutions of~\eqref{eq:main_1} and~\eqref{eq:main_2} uniformly on the interval $[t_1,t_2]$ is false. This completes the proof.
\end{proof}
\noindent\\\\
Before interpreting Theorem~\ref{thm:converse}, we first note that the result only applies to the time intervals on which $\{y_i(t)w_j(t):i,j\in[m]\}$ are positive throughout the interval. Although this condition appears stringent, it is mild from the viewpoint of epidemic spreading in the real world. This is because, in practice we are only interested in time periods during which every age group has infected cases (which ensures that $\beta_j(t)>0$ for all $j\in [m]$), and most epidemics leave behind uninfected individuals (thereby ensuring that $s_i(t)>0$ for all $i\in[m]$). Therefore, Theorem~\ref{thm:converse} applies to all time intervals of practical interest.

Restricting our focus to such intervals, Theorem~\ref{thm:converse} asserts that, if the edge update rate does not go to $\infty$ with the population size, then there exists a positive lower bound on the probability of the age-wise infected and susceptible fractions differing significantly from the corresponding solutions of the age-structured SIR ODEs at one or more points of time in the considered time interval. At this point, we remark that for large populations, the edge update rate $\lambda$ is approximately the reciprocal of the mean duration of every interaction in the network. This means that the greater the value of $\lambda$, the faster will be the changes that occur in the social interaction patterns of the network. Therefore, in conjunction with Theorem~\ref{thm:main}, Theorem~\ref{thm:converse} enables us to draw the  following inference: the age-structured SIR model can be expected to approximate a real-world epidemic spreading in a large population accurately if and only if the social interaction patterns of the network change rapidly with time. This is more likely to be the case in crowded public places such as supermarkets and airports. 

There is another way to interpret Theorems~\ref{thm:main} and~\ref{thm:converse}. Note that we have assumed that the sequence of edge states realized during the timeline of the epidemic are independent for every pair of nodes in the network. Therefore, for greater values of $\lambda$, the network structure becomes more unrecognizable from its past realizations. Thus, the age-structured SIR model can be expected to approximate epidemic spreading well if and only if the network is highly memoryless, i.e., if and only if the network continually ``forgets'' its past interaction patterns throughout the timeline of the epidemic under study.

\begin{remark}
Observe from the proof of Theorem~\ref{thm:converse} that the difference between $\E[s_i]'$, the first derivative of the expected fraction of infected nodes in $\A_i$, and $y_i'$, the first derivative of the corresponding ODE solution $y_i(t)$, is small only if $e^{-\frac{B_{ij} }{\eta\lambda_\infty}}$ is close to 1, which happens when $\lambda_\infty\gg B_{ij}$. Moreover, this observation is consistent with Remark~\ref{rem:only}, according to which the total infection rate from $[n]=\cup_{j=1}^m\A_j$ to any given susceptible node in $\A_i$ is close to $\sum_{j=1}^m A_{ij}\beta_j(t)$ (and hence, in close agreement with the ODEs~\eqref{eq:age_struct}) when $\frac{B_{ij} }{\lambda^{(n)}}\ll 1$. Along with Theorems~\ref{thm:main} and~\ref{thm:converse}, this means that the age-structured SIR model is likely to approximate real-world epidemic spreading well if and only if the infection transmission rates are negligible when compared to the social mixing rate $\lambda$.
\end{remark}
Intuitively, when $\frac{B_{ij}}{\lambda}\ll 1$, the time scales (the mean duration of time) over which the concerned disease spreads from any age group to any other age group are orders of magnitude greater than the time scale over which the network is updated. As a result, the independence of the sequences of edge state updates ensures that most of the possible realizations of the network structure are attained over the time scale of infection transmission. Equivalently, from the viewpoint of the pathogens causing the disease, the effective network structure (the network topology averaged over any of the age-wise infection timescales) is close to being a complete graph. Hence, by extrapolating the existing results on mean-field limits of epidemic processes on complete graphs (such as~\cite{armbruster2017elementary}) to heterogeneous epidemic models, we can assert that the age-structured SIR ODEs are able to approximate the epidemic propagation with a high accuracy. 

On the other hand, if the infection rates $B_{ij}$ are too high (and hence, comparable to the social mixing rate $\lambda$, which is always finite in reality), the pathogens perceive a randomly generated network even on the time scale of infection transmission. Since this random network is sparse (because we assume the expected node degrees to be constant, which results in the edge probability scaling inversely with the population size), it follows that the number of transmissions occurring in any given time period is likely to be smaller than in the case of a complete graph. Thus, the age-structued SIR ODEs   overestimate the rate of growth of age-wise infected fractions. This is further confirmed by the sign of the inequality in~\eqref{eq:not_align}.
\section{EMPIRICAL VALIDATION}\label{sec:empirical_validation}
We now validate the age-structured SIR model in the context of the COVID-19 pandemic in Japan as follows: we first estimate the model parameters using the data provided by the Government of Japan, and we then compare the trajectories generated by the model with the reference data.
\subsection{Dataset}

We use a dataset provided by the Government of Japan at~\cite{tokyodata}. This dataset partitions the population of {the prefecture of Tokyo} into $m=5$ age groups: 0 - 19, 20 - 39, 40 - 59, 60 - 79, and 80+ years old individuals. For each age group $i\in[m]$ and each day $k$ in the year-long timeline $\Gamma=\{\text{March 10, 2020},\ldots, \text{April 9, 2021}\}$, the
dataset lists the total number of people infected in the age group until date $k$. We denote this number by $I_i^T[k]$.

\subsection{Preprocessing}
Due to several factors, such as lack of reporting/testing on the weekends, the raw data has missing information and is contaminated with noise. Therefore, using a moving average filter with a window size of {15} days, we de-noise the raw data to obtain the estimated total number of infected individuals by day $k$ in age group $i$, denoted by $I_i^T[k]$. We then estimate from the smoothed data the number of susceptible, infected, and recovered individuals in age group $i\in [m]$ on day $k$, denoted by $S_i[k]$, $I_i[k]$, and $R_i[k]$, respectively. We do this as follows: for any age group $i\in [m]$ and day $k\in\Gamma$, we have $I_i^T[k]=I_i[k]+R_i[k]$, because the cumulative number of infections $I_i^T[k]$ includes both active COVID-19 cases and closed cases (cases of individuals who were infected in the past but recovered/succumbed by day $k$). Therefore, to estimate $I_i[k]$ and $R_i[k]$ from $I_i^T[k]$, we assume that every infected individual takes exactly {$T_R=14$} days to recover. This assumption is consistent with WHO's criteria for discharging patients from isolation (i.e., discontinuing transmission-based precautions)  \cite{who1} after a period involving the first 10 days from the onset of symptoms and 3 additional symptom-free days (if the patient is originally symptomatic) or after 10 days from being tested positive for SARS-CoV-2 (if the patient is asymptomatic). After the required period, the patients were not required to re-test. Under such an assumption on the recovery time, we have $R_i[k] = I_i^T[k-T_R]$ and $I_i[k] = I_i^T[k] - I_i^T[k-T_R]$. Next, we obtain $S_i[k]$ by subtracting $I_i^T[k]$ from the total population of $\A_i$, which is obtained from the age distribution and the total population of {Tokyo}.

We must mention that in the subsequent analysis, all infected individuals are considered infectious, i.e., they can potentially transmit the SARS-CoV-2 virus to their susceptible contacts. This assumption, on which the classical SIR model and all its variants are based, is consistent with the CDC's understanding of the first wave of SARS-CoV-2 infection, which claims that every infected individual remains infectious for up to about 10 days from the onset of symptoms, though the exact duration of the period of infectiousness remains uncertain~\cite{cdc1}.

\subsection{Parameter Estimation Algorithm}\label{subsec:alg}
Before estimating the parameters of our model, we discretize the ODEs~\eqref{eq:age_struct} with a step size of 1 day and obtain the following:
\begin{align} \label{eq:age_struct_discrete}
s_i[k+1] - s_i[k] &= -s_i[k] \sum^m_{j=1}A_{ij} \beta_j[k], \cr
\beta_i[k+1] - \beta_i[k] &= s_i[k] \sum^m_{j=1}A_{ij} \beta_j[k] - \gamma_i \beta_i[k], \\
r_i[k+1] - r_i[k] &= \gamma_i \beta_i[k],\nonumber
\end{align}


A key observation here is that these equations are linear in the model parameters. Therefore, given the sets of fractions $\{s_i[k]:i\in [m],k\in\Gamma\}$, $\{\beta_i[k]:i\in [m],k\in\Gamma\}$, and $\{r_i[k]:i\in [m],k\in\Gamma\}$ (which we obtain by implementing the data processing steps described above) for all $i\in [m]$, we can express~\eqref{eq:age_struct_discrete} in the form of a matrix equation $Cx=d$, where the column vector $x\in \R^{m^2+m}$ is a stack of the parameters $\{A_{ij}:1\leq i,j\leq m\}$ and $\{\gamma_i:1\leq i\leq m\}$, the column vector $d$ is a stack of the increments $\{s_i[k+1] - s_i[k]:i\in [m], k\in\Gamma\}$, $\{\beta_i[k+1] - \beta_i[k]:i\in [m], k\in\Gamma\}$, and $\{r_i[k+1] - r_i[k]:i\in [m], k\in\Gamma\}$, and $C$ is a matrix of coefficients. Thus, solving the least-squares problem~\eqref{eq:non-neg least square} gives us the best estimates of the model parameters ${\{A_{ij}:i,j\in [m]\}\cup\{\gamma_i:i\in[m]\}}$ in the mean-square sense.
\begin{align}\label{eq:non-neg least square}
    \hat{x} = \underset{ x \geq 0}{\text{argmin}} \| Cx - d \|_2.
\end{align} 

However, the values of the contact rates $A_{ij}$ change as and when the patterns of social interaction in the network change during the course of the pandemic. For this reason, we assume that the pandemic timeline splits up into multiple phases, say $\Gamma_1,\ldots, \Gamma_s$, with the contact rates varying across phases, and we perform the required optimization separately for each phase. At the same time, we do not expect the contact rates to make quantum leaps (or falls) from one phase to the next. Therefore, for every $\ell\geq 2$, 
 in the objective function corresponding to Phase $\ell$ we introduce a regularization term that penalizes any deviation of the optimization variables from the model parameters estimated for the previous phase (Phase $\ell-1$). Adding this term also ensures that our parameter estimation algorithm does not overfit the data associated with any one phase. Our optimization problem for Phase $\ell$ thus becomes 
\begin{align} \label{eq:minimization with regularizing term}
\hat{x}^{(\ell)} = \underset{ x \geq 0}{\text{argmin}} \left( \| Cx - d \|_2 + \lambda \| x^{(\ell)} - x^{(\ell-1)} \|_2\right),
\end{align}
where $x^{(\ell)}$ is the parameter vector estimated for Phase $\ell$. 

We now summarize this parameter estimation algorithm for Phase $\ell\in [s]$ .
\begin{algorithm}[H]
        \caption{Parameter Estimation Algorithm for Phase $\ell$}
        \label{alg:parameter_estimation_algorithm}
        \hspace*{\algorithmicindent} \textbf{Input:} $(s_i[k], \beta_i[k], r_i[k])$ for all $i\in [m]$ and $k\in\Gamma_\ell$ \\
        \hspace*{\algorithmicindent} \textbf{Output:} $\hat{x}$
        \begin{algorithmic}[1]
        \Function{Get\_Parameters}{$(s_i[k], \beta_i[k], r_i[k])$}
        \For{each day, each age group}
            \State Stack the difference equations~\eqref{eq:age_struct_discrete} vertically
        \EndFor
        \State Obtain the matrix equation $Cx=d$
        \State Solve Least Squares Problem~\eqref{eq:minimization with regularizing term}
        \State \Return {$\hat{x}$}
        \EndFunction
        \end{algorithmic}
\end{algorithm}

\subsection{Phase Detection Algorithm}
We now provide an algorithm that divides the timeline of the pandemic into multiple phases in such a way that the beginning of each new phase indicates a significant change in one or more of the contact rates $\{A_{ij}:i,j\in[m]\}$. 

Given the pandemic timeline $\{p_0,\ldots, p_s\}$ (where $p_0$ denotes March 10, 2020 and $p_s$ denotes April 9, 2021), our phase detection algorithm outputs $s-1$ phase boundaries $p_1\leq p_2\leq \cdots\leq p_{s-1}$ 
that divide $[p_0,p_s)$ into $s$ phases, namely $\Gamma_1=[p_0,p_1),\Gamma_2=[p_1,p_2),\ldots, \Gamma_s=[p_{s-1},p_s)$. Central to the algorithm are the following optimization problems:

Problem (a): Unconstrained Optimization
\begin{align} \label{eq:unconstrained}
\underset{x\geq 0}{\text{minimize}}\ \ \ \| C_{[p, p+w)}x - d_{[p, p+w)} \|_2.
\end{align}

Problem (b): Constrained Optimization
\begin{align} \label{eq:constrained}
\underset{x\geq 0}{\text{minimize}}\ \ \ &\| C_{[p, p+w)}x - d_{[p, p+w)} \|_2, \cr
\text{subject to}\ \ \ &\| x - \bar{x}_{[p, p+w)} \|_2 \leq \varepsilon \| \bar{x}_{[p-\Delta p, p-\Delta p+w)} \|_2.
\end{align}


In these problems, $p\in\Gamma$ denotes the \textit{start date} (chosen recursively as described in Algorithm~\ref{alg:phase_detection_algorithm}), $w\in\N$ is the \textit{optimization window}, $\Delta p<w$ is the algorithm step size, $[p,p+w)$ denotes a {$w$}-day period from day $p$, $C_{[p,p+w)}$ and $d_{[p,p+w)}$ are obtained from $\{(s_i[k],\beta_i[k],r_i[k]):i\in [m],k\in\{p,\ldots, p+w\}\}$ by using the procedure described in Section~\ref{subsec:alg}, and $\bar{x}:=\argmin_{x\geq 0}\|C_{[p,p+w)}x-d_{[p,p+w)}\|_2$ is the parameter vector estimated by Problem (a). We set $w=30$ (days), and the quantities $\Delta p$ and $\varepsilon$ are pre-determined algorithm parameters whose choice is discussed in the next subsection.

Observe that both Problem and Problem (b) result in the minimization of the  mean-square error~\eqref{eq:least-squares-error}, where ${\{(\hat s_i[k],\hat \beta_i[k],\hat r_i[k]):i\in [m],k\in\{p,\ldots, p+w\}\}}$ are the model-generated values (estimates) of the susceptible, infected, and recovered fractions  $\{(s_i[k],\beta_i[k],r_i[k]):i\in [m],k\in\{p,\ldots, p+w\}\}$. Also note that Problem (a) performs this minimization while ignoring all the previously estimated model parameters, whereas Problem (b) performs the same minimization while constraining  $x$ to remain close to the parameter vector estimated for the period $[p-\Delta p, p-\Delta p + w)$. However, if the contact rates do not change significantly around day $p$, then the additional constraint imposed in Problem (b) should be satisfied automatically (without imposition) in Problem (a), which should in turn result in the same mean-square error for both the problems. 
\begin{align}\label{eq:least-squares-error}
\mathcal{E}  = \frac{1}{ 3m(w+1) } \sum_{i\in [m]}\sum_{k\in \{p,\ldots, p+w\}} \bigg( (s_i[k]-\hat{s}_i[k])^2  +(\beta_i[k]-\hat\beta_i[k])^2 +(r_i[k]-\hat{r}_i[k])^2\bigg).
\end{align}

Therefore, after solving Problems (a) and (b), our phase detection algorithm compares $\mathcal E_{(a)p}$ (the mean-square error for Problem (a)) with $\mathcal E_{(b)p}$ (the mean-square error for Problem (b)) as follows: using~\eqref{eq:least-squares-error}, the algorithm first computes $\mathcal E_{(a)p}$ and $\mathcal E_{(b)p}$. It then compares $\frac{| \mathcal{E}_{(b)p} - \mathcal{E}_{(a)p} |}{\mathcal{E}_{(a)p} }$ with $\delta$, a positive threshold whose choice is discussed in the next subsection. If
\begin{align}\label{eq:mae_condition}
\frac{| \mathcal{E}_{(b)p} - \mathcal{E}_{(a)p} |}{\mathcal{E}_{(a)p} } > \delta,
\end{align}
then $p$ is identified as a phase boundary. Otherwise,  the algorithm increments the value of $p$ by $\Delta p$, checks whether the interval $[p,p+w)$ is part of the timeline $\Gamma$, and repeats the entire procedure described above.

Finally, the algorithm merges every short phase (length $\leq 20$ days) with its predecessor by deleting the appropriate phase boundary(s). There are two reasons for this step. First, the contact rates are believed to change not instantly but with a transition period of positive duration. Second, 
since the data used is noisy, to avoid overfitting the data it is necessary for the number of data points per phase (given by $2m$ times the number of days per phase) to significantly exceed $m^2 + m$, the number of model parameters to be estimated per phase.

We now provide the pseudocode for the entire  algorithm. Observe that Problems (a) and (b) are both convex optimization problems. This enables us to use the Embedded Conic Solver (ECOS) \cite{domahidi2013ecos} of CVXPY \cite{diamond2016cvxpy,agrawal2018rewriting} to implement our algorithm.

\begin{algorithm}[H]
        \caption{Phase Detection Algorithm}
        \label{alg:phase_detection_algorithm}
        \hspace*{\algorithmicindent} \textbf{Input:} $(s_i[k], \beta_i[k], r_i[k])$ for all $i\in [m]$ and $k\in\Gamma$ \\
        \hspace*{\algorithmicindent} \textbf{Output:} Set of phase boundaries $\mathcal{B}$
        \begin{algorithmic}[1]
        \Function{Detect\_Phases}{$(s_i[k], \beta_i[k], r_i[k])$}
        \State Initialize set of phase boundaries $\mathcal{B}\leftarrow \phi$ 
        \State Initialize start date $p\leftarrow \Delta p$
        \While{ $p \in \Gamma$}
            \State Solve Problem (a) for window $[p, p+w)$
            \State Solve Problem (b) for window $[p, p+w)$ \If{condition~\eqref{eq:mae_condition} holds}
                \State $\mathcal B\leftarrow \mathcal B\cup\{p\}$
            \EndIf
            \State $p\leftarrow p +\Delta p$
        \EndWhile
        \State Initialize $p_{\text{start}}\leftarrow0$
        \State Initialize $\mathbf b\leftarrow\text{list}(\mathcal B)$
        \State Sort $\mathbf{b}$ in ascending order
        \For{$p\in \mathbf b$} 
            \If{$p-p_{\text{start}}\leq 20$}
                \State $\B\leftarrow\B \setminus\{p\}$
            \Else
                \State $p_{\text{start}}=p$
            \EndIf
        \EndFor
        \State \Return {$\mathcal{B}$}
        \EndFunction
        \end{algorithmic}
\end{algorithm}

\subsection{Selection of Algorithm Parameters}



We now explain our parameter choices for the algorithms described above.

\subsubsection{Phase Detection Algorithm}
As mentioned earlier, for Algorithm~\ref{alg:phase_detection_algorithm}, we set $\Delta p=5$ days and the optimization window $w = 30$ days. This ensures that the optimization window is large enough for the number of model parameters to be significantly smaller than the  number of data points used to estimate these parameters in Problems (a) and (b). In addition, we set $\delta = 3$, and $\varepsilon = 10^{-4}$ for the following reasons:
\begin{enumerate}
    \item $\varepsilon = 10^{-4}$: If both $[p,p+w)$ and $[p-\Delta p,p-\Delta p+w)$ are sub-intervals of the same phase, then the same set of contact rates (and hence the same parameter vector $x$) should apply to the network during both the time intervals.
    \item $\delta = 3$: If day $p$ marks the beginning of a new phase (i.e., a new set of contact rates), we expect the least-squares error~\eqref{eq:least-squares-error} to increase significantly upon the imposition of the constraint introduced in~\eqref{eq:constrained}.
\end{enumerate}

\subsubsection{Parameter Estimation Algorithm}
We set $\lambda = 10^{-5}$ in \eqref{eq:minimization with regularizing term}. This small but non-zero value is consistent with our belief that around every phase boundary, contact rates change gradually but significantly during a transition period involving the phase boundary. 

\subsection{Results}
We now present the results of implementing both the algorithms on our chosen dataset. 

\subsubsection{Phase Detection}
Algorithm~\ref{alg:phase_detection_algorithm} detects the following phases.
\begin{table}[H]
\scriptsize
\centering
\fontsize{7.95}{10}\selectfont
  \begin{tabular}{|c|c|c|c|} \hline
    Phase & From & To & Corresponding Events \\ \hline
    1 & Mar 10 2020 & Mar 28 2020 & Closure of Schools\\
    2 & Mar 28 2020 & April 23 2020 & Issuance of State of Emergency\\
    3 & April 23 2020 & May 20 2020 & \\
    4 & May 20 2020 & Jun 22 2020 & \\
    5 & Jun 22 2020 & Jul 24 2020 & Summer Vacation \\ 
    6 & Jul 24 2020 & Aug 25 2020 & Obon, Summer Vacation \\
    7 & Aug 25 2020 & Sep 23 2020 & Summer Vacation\\
    8 & Sep 23 2020 & Oct 20 2020 & ``Go to Travel'' Campaign\\ & & & Relaxation of Immigration Policy\\
    9 & Oct 20 2020 & Nov 14 2020 & 
      ``Go to Eat'' Campaign\\ & & &  ``Go to Travel'' Campaign\\
    10 & Nov 14 2020 & Dec 19 2020 & \\
    11 & Dec 19 2020 & Jan 12 2021 & Issuance of State of Emergency\\
     & & & Winter Vacation\\
    12 & Jan 12 2021 & Feb 07 2021 &\\
    13 & Feb 07 2021 & Apr 09 2021 &\\ \hline
  \end{tabular}
  \caption{Phases Detected by Algorithm~\ref{alg:phase_detection_algorithm}~\cite{wiki,karako2021overview}}
  \label{Tab:Major_events_cal}
\end{table}

Although some of the detected phases can be accounted for by identifying changes in governmental policies and major social events, many of them seem to result from changes in social interaction patterns that cannot be explained using public information sources (such as news websites). However, this is consistent with out intuition that  social behavior is inherently dynamic -- it displays significant changes even in the absence of government diktats and important calendar events. Moreover, except for the first phase, the length of every phase is at least 25 days, which points to the likely scenario that it takes at least 3 to 4 weeks for the contact rates to change significantly. This could be true because social behavior is often unorganized. In particular, the interaction patterns of any one individual are often not in synchronization with those of others.

Another noteworthy inference to be drawn from Table~\ref{Tab:Major_events_cal} and Figure~\ref{fig:daily_cases} is that policy changes initiated by  governments have a delayed effect at times. For example, the ``Go to Travel'' and the ``Go to Eat'' campaigns, launched between mid-September and mid-November (Phases 8 and 9), seem to have caused a spike in daily case counts in the subsequent phases (Phases 10 and 11). Likewise, the State of Emergency issued in Phase 11 seems to have come to fruition in Phase 12 and its effects appear to have remained until the last phase (Phase 13).

\subsubsection{Parameter Estimation and Its Implications}
Figure~\ref{fig:daily_cases} below plots the original and the model-generated fractions of infected individuals in each age group as functions of time.
\begin{figure}[H]
\includegraphics[scale=0.5]{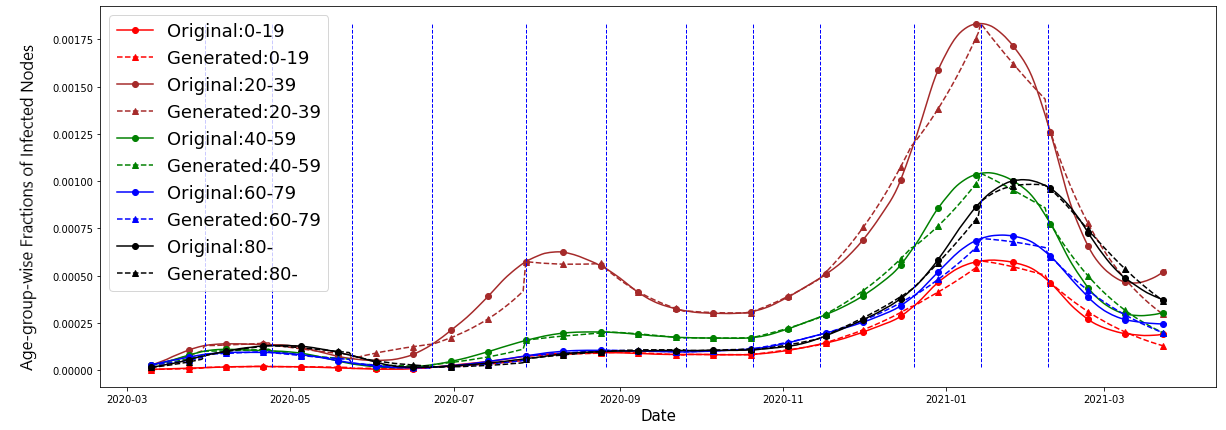}
\caption{Age-wise Daily Fractions of Infected Individuals in Tokyo, Japan: Original and Generated Trajectories}
\label{fig:daily_cases}
\end{figure}
Figure \ref{fig:parameters} plots the estimated contact rates and labels the 10 most significant ones among the 25 rates. 

As seen in Figure \ref{fig:daily_cases}, three COVID-19 surges or ``waves'' occur during the considered timeline. For each wave, we explain below the corresponding contact rate variations and their implications with the help of the mobility data of Tokyo (Figure~\ref{fig:mobility}) collected by Google \cite{google_mobility}.

\begin{figure*}[htb]
\includegraphics[scale=0.44]{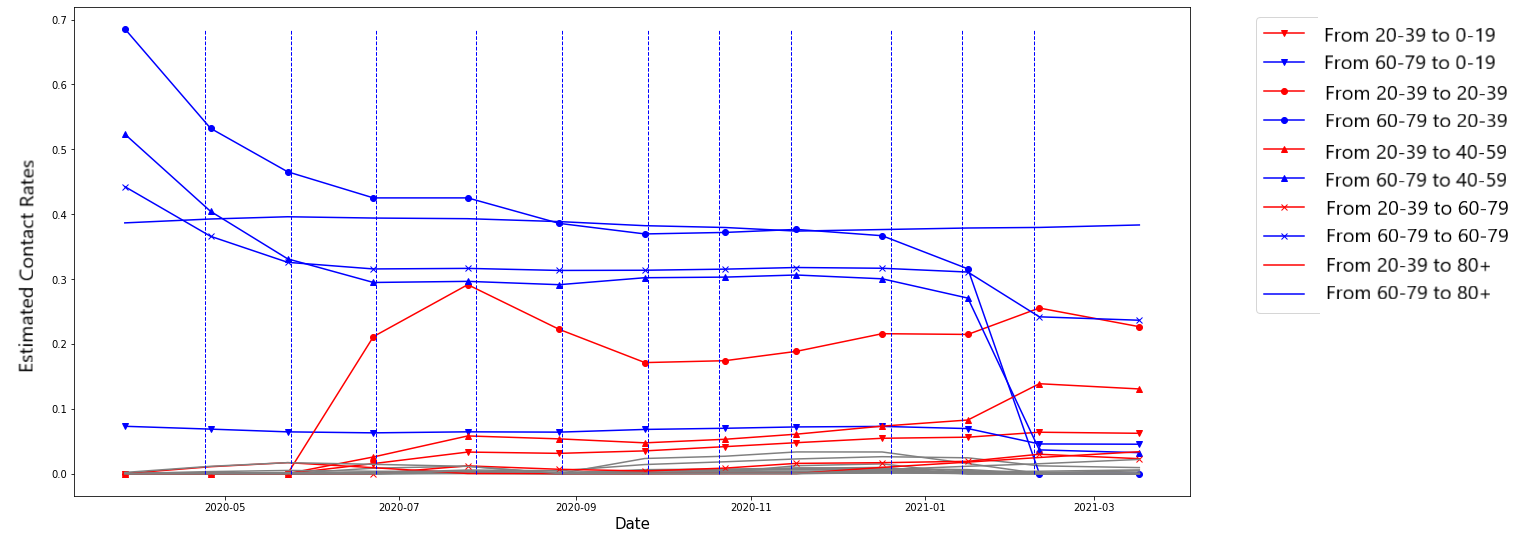}
\caption{Estimated Contact Rate Between Groups}
\label{fig:parameters}
\end{figure*}

\begin{figure*}[htb]
\includegraphics[width=190mm]{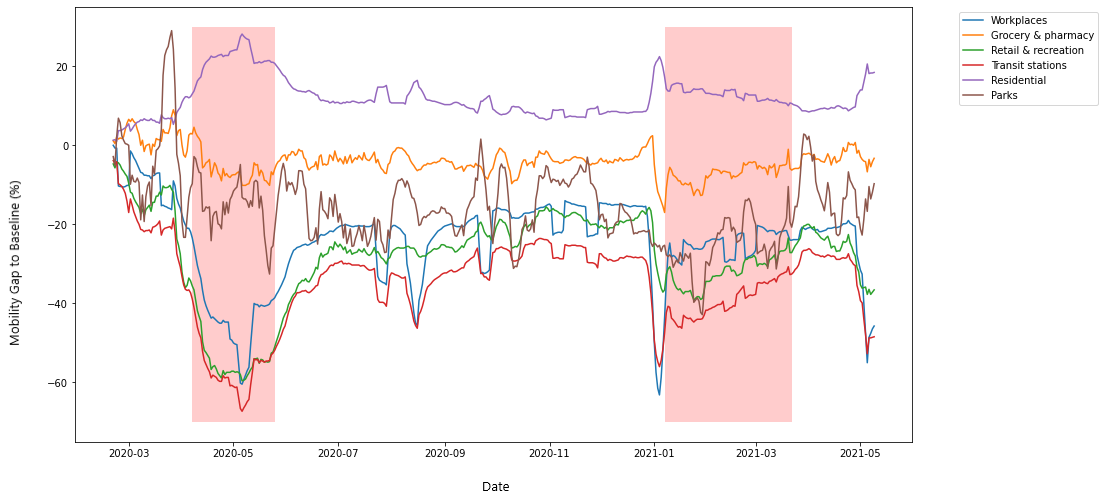}
\caption{Mobility for Each Type of Place by Google}
\footnotesize The period in which state of emergency is issued is highlighted in red.
\label{fig:mobility}
\end{figure*}
\subsubsection*{The First Wave (March 2020 - June 2020, Phases 1 - 3)}
This wave corresponds to a rapid surge in daily cases across the world followed by various governmental measures such as issuance of national emergencies, tightening of immigration policies, home quarantines, and school closures. In Japan, the national emergency consisted of various measures such as restrictions on service times in restaurants and bars, enforcement of work from home, and a limit on the number of people attending public events. As a result of these measures, the mobility of workplaces, retail and recreation, and transit stations  dropped dramatically in April 2020 and remained low  for over a month (Figure \ref{fig:mobility}).

This drop is reflected in our simulation results (Figure~\ref{fig:parameters}), which show that the three greatest contact rates decreased steadily from April to June. However, Figure~\ref{fig:parameters} also shows that contact rates from the age group 60-79 to most other age groups (shown in blue) remained remarkably high throughout the timeline $\Gamma$. This may be because most people admitted to nursing homes are aged above 60 and frequently come in contact with the relatively younger care-taking staff. More strikingly, the contact rate from age group 60-79 to age group 80+ is consistently high. This could be because there is a significant number of married couples with members from both these age groups (thereby resulting in a high value of $\rho_{ij}$ for $(i,j)=(m,m-1)$)  and because the age group 80+ has the lowest immunity levels, which leads to a large effective $B_{ij}$ (infection rate) for $(i,j)=(m,m-1)$. 
\subsubsection*{The Second Wave (July 2020 - September 2020, Phases 4 - 7)}
The most intriguing aspect of the second wave is that the wave subsided without any significant governmental interventions (such as the issuance of a nationwide emergency). To explain this phenomenon, some researchers point out that (i) the rate of PCR testing increased {in July} and thus more infections were detected in the first few weeks of the second wave, and (ii) people's mobility decreased in August during the Japanese summer vacation period called ``Obon'' \cite{karako2021overview}. As we can infer from Figure~\ref{fig:mobility}, this decrease in mobility occurs primarily at workplaces and transit stations~\cite{karako2021overview}.  

Besides Figure~\ref{fig:mobility}, our simulation results provide some insight into the second wave. Figure \ref{fig:daily_cases} shows that the contact rates from age group 60-79 to other age groups do not show any increase during the first few weeks of the wave. However, the intra-group contact rate of the age group 20-39 increases rapidly before this period and drops significantly in August, corresponding to a decrease in daily cases. This strongly suggests that the social activities of those aged between 20-39 played a key role in the second wave. Meanwhile,  contact rates from the age group 60-79  decreased after the first wave, possibly because of an increase in the proportion of  quarantined individuals among the elderly, which in turn could have resulted from an increased public awareness of older age groups' higher susceptibility to the virus.

\subsubsection*{The Third Wave (October 2020 - January 2021, Phases 8 - 11)}
This wave was the most severe of the three because in October, a policy promoting domestic travel (the ``Go to Travel'' campaign) was implemented in the Tokyo prefecture and eating out was promoted as well (as part of the ``Go to Eat'' campaign). In addition, Japan started relaxing its immigration policy in October. \cite{karako2021overview} points out that the ``major factors for this rise include the government’s implementation of further policies to encourage certain activities, relaxed immigration restrictions, and people not reducing their level of activity''. This observation is supported by Figure~\ref{fig:mobility}, which shows that there is no drop in mobility in any category during the third wave. As a result, daily infection counts dropped only after the second state of emergency was issued by the government on January 7, 2021.

In agreement with these observations are our simulation results (Figure~\ref{fig:daily_cases}), which show that the age group 60-79 remained the most infectious throughout the third wave, and that the contact rates from the age group 20-39 gradually increased in the early weeks of the wave. This was followed by a remarkable decrease in the intra-group contact rate of the age group 60-79 from mid-January onwards. \\

\subsection*{Comparing the Age Groups on the Basis of Infectiousness and Susceptibility}

It is evident from Figure~\ref{fig:parameters} that among all the five age groups, members of the youngest age group (0-19) are the least likely to contract COVID-19. This validates the current understanding of the scientific community that children and teenagers are more immune to the disease than adults. At the opposite extreme, the age groups 80+ and 20-39 appear to be the most vulnerable, possibly because members of the former group have the lowest immunity levels and the latter group exhibits the highest levels of mobility and social activity.

Besides throwing light on how the likelihood of receiving infection varies across age groups, Figure~\ref{fig:parameters} also throws light on how the likelihood of transmitting the infection varies across age groups. From the figure, the two most infectious age groups are clearly 60-79 and 20-39. Surprisingly, the age group 80+  is found to be less infectious than the group 60-79, perhaps because of the lower social mobility of the former. The figure also shows that the age groups 0-19, 40-59 and 80+ are remarkably less infectious than the other two age groups. However, we need additional empirical evidence to validate these findings, and it would be interesting to see whether our inferences  are echoed by future empirical studies.
\section{CONCLUSION AND FUTURE DIRECTIONS}\label{sec:conclusion}
\color{black}

We have analyzed the age-structured SIR model of epidemic spreading from both theoretical and empirical viewpoints. Starting from a stochastic epidemic model, we have shown that the ODEs defining the age-structured SIR model are the mean-field limits of a continuous-time Markov process evolving over a time-varying network that involves random, asynchronous interactions if and only if the social mixing rate grows unboundedly with the population size. We have also provided a lower-bound on the associated convergence rates in terms of the social mixing rate. As for empirical validation, we have proposed two algorithms: a least-square method to estimate the model parameters based on real data and a phase detection algorithm to detect changes in contact rates and hence also the most significant social behavioral changes that possibly occurred during the observed pandemic timeline. We have validated our model empirically by using it to approximate the trajectories of the numbers of susceptible, infected, and recovered individuals in the prefecture of Tokyo, Japan, over a period of more than 12 months. Our results show that for the purpose of forecasting the future of the COVID-19 pandemic and designing appropriate control policies, the age-structured SIR model is likely to be a strong contender among compartmental epidemic models.

Our analysis, however, has a few limitations. First, it is not clear whether the large number of phases detected by Algorithm~\ref{alg:phase_detection_algorithm} indicates rapidly changing social interaction patterns or simply that our model is unable to approximate the pandemic over timescales significantly longer than a month. Second, the outputs of our algorithms have a few surprising implications that are as yet unconfirmed by independent empirical studies. For example, the estimated contact rates indicate that the age group 60-79  is consistently more infectious than the  age group 20-39, a finding that is inconsistent with the widely held belief that younger age groups are significantly more mobile than the older ones. Such apparent anomalies highlight the need for age-stratified mobility datasets that would enable further investigation into the dynamic interplay between social behavior and epidemic spreading.

\bibliographystyle{ieeetr}
\bibliography{bib}
\section*{Appendix}

Our first aim is to prove Proposition~\ref{prop:basic}, which is based on the Lemma~\ref{lem:obvious}. This lemma describes a known property of continuous-time Markov chains, but we prove it nevertheless. The proof is based on the concept of \textit{jump times}, defined below.

\begin{definition} [Jump Times] The \emph{jump times} of the Markov chain $\{\X(\tau):\tau\geq 0\}$ are the random times defined by $J_0:=0$ and $J_\ell := \inf\{\tau\geq 0: \X(J_{\ell-1}+\tau)\neq \X(J_{\ell-1})\}$ for all $\ell\in\N$.
\end{definition}

Note that jump times are simply the times at which the Markov chain jumps or transitions to a new state.

\begin{lemma}\label{lem:obvious}
Let $\x\in\mathbb S$ and let $[t,t+\Delta t)\subset [0,\infty)$. Given that $\X(t)=\x$, the conditional probability that more than one state transitions occur during $[t,t+\Delta t)$ is $o(\Delta t)$. 
\end{lemma}

\begin{proof}
 Let $\y, \mathbf z\in\mathbb S$ be any two states such that $\y$ and $\mathbf z$ potentially succeed $\x$ and $\y$, respectively. Also, let $\{\X(J_i)\}_{i=0}^\infty$ be the embedded jump chain of $\{\X(\tau)\}_{\tau\geq 0}$ (where $J_0:=0$). Then, given that $\X(0)=\x$, $\X(J_1)=\y$, and $\X(J_2)=\mathbf z$, the holding times $J_1$ and $J_2 - J_1$ are conditionally independent exponential random variables with parameters $q_x := |\Q(\x,\x)|$ and $q_y:=|\Q(\y,\y)|$, respectively. Therefore, given that the original Markov chain makes its first and second transitions from $\x$ to $\y$ and from $\y$ to $\mathbf z$  respectively, the conditional probability that both of these transitions occur during $[0, \Delta t)$ is given by
 \begin{align}
     &\Pr(J_2 < \Delta t\mid (\X(0),\X(J_1), \X(J_2)) = (\x,\y,\mathbf z))\cr
     &\leq \Pr(J_2 - J_1 <\Delta t, J_1<\Delta t \mid (\X(0),\X(J_1), \X(J_2)) = (\x,\y,\mathbf z))\cr
     &=\Pr(J_2-J_1<\Delta t \mid (\X(0),\X(J_1), \X(J_2)) = (\x,\y,\mathbf z))\cdot \Pr (J_1<\Delta t\mid (\X(0),\X(J_1), \X(J_2)) = (\x,\y,\mathbf z))\cr
     &=(1- e^{-q_y \Delta t})(1 - e^{-q_x\Delta t})\cr
     &= o(\Delta t).
 \end{align}
 Therefore, $\Pr(J_2<\Delta t\mid \X(0)=\x )$, which is at most $ \max_{\y,\mathbf z\in\mathbb S}\Pr(J_2<\Delta t\mid (\X(0),\X(J_1), \X(J_t)=(\x,\y, \mathbf z ))$, is $ o(\Delta t)$. Hence, given that $\X(0)=\x$, the conditional probability that at least two state transitions occur during $[0,\Delta t)$ is $o(\Delta t)$. By time-homogeneity, this means the following: given that $\X(t)=\x$, the conditional probability that at least two state transitions occur during $[t,t+\Delta t)$ is  $o(\Delta t)$. 
\end{proof}

\subsection*{Proof of Proposition~\ref{prop:basic}}

\begin{proof}
We derive the equations one by one.
\subsubsection*{\textit{Proof of \eqref{item:init_1}:}} Consider any state $\x\in\mathbb S$. Then, by the definition of $\Q$,  for any $i\in [n]$ and $a\in\S_i(\x)$, we have
\begin{align}~\label{eq:basic_infection}
    \Pr(\X(t+\Delta t) = \x_{\uparrow a} \mid \X(t) = \x)&= \Q(\x,\x_{\uparrow a})\Delta t + o(\Delta t)= \left(\sum_{k=1}^m B_{ik} E_{k}^{(a)}(\x)\right) \Delta t + o(\Delta t).
\end{align}

We now use~\eqref{eq:basic_infection} to evaluate the probability of the event $\{S_i(t+\Delta t) = S_i(t)-1\}$. To this end, let $D_\ell(U,t, \Delta t) )$ denote the event that exactly $\ell$ nodes in a given set $U\subset [n]$ recover during $[t,t+\Delta t)$ (i.e., there exist exactly $\ell$ indices $r_1,\ldots, r_\ell$ in $U$ such that $X_{r_k}(t)=1$ and $X_{r_k} (t+\Delta t)=-1$). Similarly, let $\mathscr I_\ell (U,t,\Delta t)$ denote the event that exactly $\ell$ nodes in $U$ get infected during $[t,t+\Delta t)$. Then, 
\begin{align}
    &\Pr(S_i(t+\Delta t) - S_i(t) = -1 \mid \X(t) =\x)\cr
    &= \Pr( \mathscr I_1(\A_i,t,\Delta t) \mid \X(t)=\x ) \cr
    &\stackrel{(a)}{=} \Pr(D_0([n],t,\Delta t)\cap \mathscr I_0([n]\setminus\A_i,t,\Delta t)\cap \mathscr I_1(\A_i,t,\Delta t)\mid \X(t)=\x ) + o(\Delta t)\cr
    &=\Pr\left(\cup_{a\in \S_i(\x)} \{\X(t+\Delta t) = \x_{\uparrow a}\}\mid \X(t)=\x\right) + o(\Delta t)\cr
    &=\sum_{a\in \S_i(\x)} \Pr(\X(t+\Delta t) = \x_{\uparrow a}\mid \X(t)=\x) + o(\Delta t)\cr
    &\stackrel{(b)}{=}\left(\sum_{a\in\S_i(\x)}\sum_{j=1}^m B_{ij} E_{j}^{(a)}(\x)\right) \Delta t + o(\Delta t)\cr
    &=\left( \sum_{j=1}^m \sum_{a\in\S_i(\x)}\sum_{b\in\I_j(\x)} B_{ij} 1_{(a,b)}(\x) \right) \Delta t + o(\Delta t).
\end{align}
where $(a)$ is a straightforward consequence of Lemma~\ref{lem:obvious}, and $(b)$ follows from~\eqref{eq:basic_infection}. Since this holds for all $\x\in\mathbb S$, we have
\begin{align}\label{eq:to_be_recast}
    \Pr(S_i(t+\Delta t)-S_i(t)=-1\mid \X(t) ) =\left( \sum_{j=1}^m \sum_{a\in\S_i(t )}\sum_{b\in\I_j(t )} B_{ij} 1_{(a,b)}(t ) \right) \Delta t + o(\Delta t),
\end{align}
where $\S_i(t)$, $\I_i(t)$, and $1_{(a,b)}(t)$ stand for $\S_i(\X(t))$, $\I_i(\X(t))$, and $1_{(a,b)}(\X(t))$, respectively. Since $\S(t)$ and $\I(t)$ are determined by $\X(t)$, we may express~\eqref{eq:to_be_recast} as
\begin{align}
    \Pr(S_i(t+\Delta t)-S_i(t)=-1\mid \S(t),\I(t), \X(t) ) =\left( \sum_{j=1}^m \sum_{a\in\S_i(t )}\sum_{b\in\I_j(t )} B_{ij} 1_{(a,b)}(t ) \right) \Delta t + o(\Delta t).
\end{align}
As a result, we have
\begin{align}
    \Pr(S_i(t+\Delta t)-S_i(t)=-1\mid \S(t),\I(t) ) =\left( \sum_{j=1}^m \sum_{a\in\S_i(t )}\sum_{b\in\I_j(t )} B_{ij} \E[1_{(a,b)}(t )\mid \S(t),\I(t)] \right) \Delta t+ o(\Delta t).
\end{align}
At this point we note that
$$
    \E[1_{(a,b)}(t )\mid \S(t),\I(t)] = \Pr((a,b)\in E(t)\mid \S(t), \I(t)) = \chi_{ij}(t,\S,\I).
$$
 We thus have the following for $\Delta_t S_i:= S_i(t+\Delta t)-S_i(t)$:
\begin{align}\label{eq:nine}
    \E[\Delta_t S_i \mid \S(t),\I(t)]
    &= -\Pr(\Delta_t S_i =-1\mid \S(t),\I(t) ) -\sum_{\ell = 2}^n \ell\cdot \Pr(\Delta_t S_i = -\ell) \cr
    &\stackrel{(a)}{=}-\Pr(\Delta_t S_i=-1\mid \S(t),\I(t) ) + o(\Delta t) \cr
    &=-\left( \sum_{j=1}^m \sum_{a\in\S_i(t )}\sum_{b\in\I_j(t )} B_{ij} \chi_{ij}(t) \right) \Delta t + o(\Delta t)\cr
    &=-\left( \sum_{j=1}^m  B_{ij} \chi_{ij}(t) S_i(t) I_j(t) \right) \Delta t + o(\Delta t),
\end{align}
where (a) follows from Lemma~\ref{lem:obvious}. Taking expectations on both sides of~\eqref{eq:nine} and dividing the resulting relation by $\Delta t$ now yields
\begin{align}\label{eq:last}
    \E\left[\frac{S_i(t+\Delta t)-S_i(t) }{\Delta t}\right ] &= - \sum_{j=1}^m  B_{ij} \E[n\chi_{ij}(t) s_i(t) I_j(t)] + \frac{o(\Delta t)}{\Delta t},
\end{align}
where we used that $S_i(t) = ns_i(t)$. On letting $\Delta t\rightarrow 0$ and then dividing both the sides of~\eqref{eq:last} by $n$,  we obtain \eqref{item:init_1}.

\subsubsection*{\textit{Proof of~\eqref{item:init_2}:}} Observe that for any $\x\in\mathbb S$, we have \begin{align}\label{eq:eleven}
    &\Pr(I_i(t+\Delta t) - I_i(t) = 1 \mid \X(t) =\x)\cr &=\Pr(|\I_i(\X(t+\Delta t))| - |\I_i(\X(t))| = 1\mid \X(t) = \x)\cr
    &=\Pr(\cup_{\ell= 0}^n (D_\ell(\A_i, t, \Delta t) \cap \mathscr I_{\ell+1}( \A_i, t,\Delta t ) )\mid \X(t)=\x    )\cr 
    &\stackrel{(a)}{=} \Pr( D_0(\A_i, t, \Delta t)\cap \mathscr I_1(\A_i,t,\Delta t) \mid \X(t)=\x )+o(\Delta t ) \cr
    &\stackrel{(b)}{=} \Pr(D_0([n],t,\Delta t)\cap \mathscr I_0([n]\setminus\A_i,t,\Delta t) \cap \mathscr I_1(\A_i,t,\Delta t)\mid \X(t)=\x ) + o(\Delta t)\cr
    &=\Pr\left(\cup_{c\in \S_i(\x)}\{ \X(t+\Delta t) = \x_{\uparrow c}\}\mid \X(t)=\x\right) + o(\Delta t)\cr
    &=\sum_{c\in \S_i(\x)} \Pr(\X(t+\Delta t) = \x_{\uparrow c}\mid \X(t)=\x) + o(\Delta t)\cr
    &\stackrel{(c)}{=}\left(\sum_{c\in\S_i(\x)}\sum_{j=1}^m B_{ij} E_{j}^{(c)}(\x)\right) \Delta t + o(\Delta t),
\end{align}
where $(a)$ and $(b)$ follow from Lemma~\ref{lem:obvious} and $(c)$ follows from~\eqref{eq:basic_infection}. 

On the other hand, 
\begin{align}\label{eq:twelve}
    &\Pr(I_i(t+\Delta t) - I_i(t) = -1 \mid \X(t) =\x)\cr
    &=\Pr(\cup_{\ell= 0}^n (D_{\ell+1} (\A_i, t, \Delta t) \cap \mathscr I_{\ell}( \A_i, t,\Delta t ) )\mid \X(t)=\x    )\cr 
    &\stackrel{(a)}{=} \Pr( D_1(\A_i, t, \Delta t)\cap \mathscr I_0(\A_i,t,\Delta t) \mid \X(t)=\x )+o(\Delta t ) \cr
    &\stackrel{(b)}{=} \Pr(D_1(\A_i,t,\Delta t)\cap D_0([n]\setminus\A_i,t,\Delta t) \cap \mathscr I_0([n],t,\Delta t) \mid \X(t)=\x ) + o(\Delta t)\cr
    &=\sum_{c\in \I_i(\x)} \Pr(\X(t+\Delta t) = \x_{\downarrow c}\mid \X(t)=\x) + o(\Delta t)\cr
    &=\sum_{c\in\I_i(\x)} (\Q(\x,\x_{\downarrow c})\Delta t + o(\Delta t)) + o(\Delta t)\cr
    &=\sum_{c\in\I_i(\x)}\gamma_i\Delta t + o(\Delta t)\cr
    & = \gamma_i |I_i(\x)|\Delta t + o(\Delta t).
\end{align}
As a result of~\eqref{eq:eleven},~\eqref{eq:twelve}, and Lemma~\ref{lem:obvious}, we have
\begin{align*}
    \E[ I_i(t+\Delta t) - I_i(t) \mid \X(t)  ]=\left(\sum_{c\in\S_i(\x)}\sum_{j=1}^m B_{ij}E_j^{(c)}(\X(t) )  - \gamma_i|\I_i(\X(t))|\right)\Delta t+ o(\Delta t).
\end{align*}
By repeating the arguments used in the proof of~\eqref{item:init_1}, we can use the above to prove that
\begin{align*}
    \E[ I_i(t+\Delta t) - I_i(t) \mid \S(t), \I(t)  ]=\left(\sum_{j=1}^m B_{ij} \chi_{ij}(t) S_i(t) I_j(t)  - \gamma_iI_i( t )\right)\Delta t + o(\Delta t),
\end{align*}
which implies that
\begin{align*}
    \E\left[\frac{ I_i(t+\Delta t) - I_i(t) }{\Delta t} \right]=\sum_{j=1}^m B_{ij} \E[n\chi_{ij}(t) s_i(t) I_j(t)] - \gamma_i\E[I_i( t )] + \frac{o(\Delta t)}{\Delta t}.
\end{align*}
On dividing both sides by $n$ and then letting $\Delta t\rightarrow 0$, we obtain~\eqref{item:init_2}.

\subsubsection*{\textit{Proof of~\eqref{item:init_3}:}} 
Observe that when $\Delta_t S_i = -1$, we have $S_i^2(t+\Delta t) - S_i^2(t) = 1 - 2S_i(t)$. Therefore,
    \begin{align*}
        \E[S_i^2(t+\Delta t) - S_i^2(t)\mid\S(t), \I(t)] &= (1-2S_i(t))\cdot \Pr(\Delta_t S_i = -1 \mid \S(t), \I(t)) + o(\Delta t)  \cr
        &= \left(\sum_{j=1}^m B_{ij}\chi_{ij} S_i I_j\Delta t - 2\sum_{j=1}^m B_{ij}\chi_{ij} S_i^2 I_j\Delta t\right) + o(\Delta t).
    \end{align*}
    Taking expectations on both sides, dividing by $\Delta t$, letting $\Delta t\rightarrow 0$, and dividing both sides by $n^2$ yields \eqref{item:init_3}.
    
    \subsubsection*{\textit{Proof of \eqref{item:init_4}:}} Observe that if $\Delta_t I_i(t) = -1$, we have $I_i^2(t+\Delta t) - I_i^2(t) = 1 - 2I_i(t)$, and if $\Delta_t I_i(t) = 1$, we have $I_i^2(t+\Delta t) - I_i^2(t) = 1 + 2I_i(t)$.
    
    Thus, 
    \begin{align*}
        &\E[ I_i^2(t+\Delta t) - I_i^2(t)  \mid\S(t) , \I(t) ]\cr
        &= (1-2I_i(t))\cdot \Pr( \Delta_t I_i=-1\mid \S(t), \I(t))+(1+2I_i(t))\cdot\Pr(\Delta_t I_i=1\mid\S(t), \I(t) ) + o(\Delta t)
    \end{align*}
    On substituting the probabilities above with the expressions derived earlier, taking expectations on both sides, dividing by $n^2\Delta t$ and letting $\Delta t\rightarrow 0$, we obtain \eqref{item:init_4}.
\end{proof}

\begin{lemma} \label{lem:difficult}
Let $a\in\A_i$ and $b\in\A_j$ be any two nodes, let $t\in[0,\infty)$ be any time, and let $T\in [0,t]$ be the random variable such that $t-T$ is the time at which $1_{(a,b)}$ is updated for the last time during the interval $[0,t]$. Then the random variables $T$ and $1_{(a,b)}(t)$ are independent.
\end{lemma}

\begin{proof}
For $\tau\in [0,t]$, let $N_\tau$ denote the number of times $1_{(a,b)}$ is updated in the open interval $(t-\tau,t)$, and let $U_\tau$ denote the zero-probability event that $1_{(a,b)}$ is updated at time $t-\tau$. Note that $\Q(\x,\x_{\existence(a,b)})+\Q(\x,\x_{\nonexistence(a,b)})=\lambda$ for all $\x\in\mathbb S$, which means that the  rate at which $1_{(a,b)}$ is updated is time-invariant and independent of the network state. This means that the sequence of times at which  $1_{(a,b)}$ is updated is a Poisson process, which further means that the updates of $1_{(a,b)}$ occurring in disjoint time intervals are independent. It follows that $N_\tau$ is a Poisson random variable (with mean $\lambda\tau$) that is independent of ${U_\tau}$ and $1_{(a,b)}(t-\tau)$. As a result,
\begin{align*}
    \Pr((a,b)\in E(t)\mid T=\tau)&\stackrel{(a)}{=}\Pr((a,b)\in E(t-\tau)\mid U_\tau, N_\tau=0)\cr
    &\stackrel{(b)}{=}\frac{\Pr(N_\tau=0\mid (a,b)\in E(t-\tau), U_\tau)}{\Pr(N_\tau=0\mid U_\tau)}\cdot\Pr((a,b)\in E(t-\tau)\mid U_\tau)\cr
    &=\frac{\Pr(N_\tau=0\mid 1_{(a,b)}(t-\tau)=1, U_\tau)}{\Pr(N_\tau=0\mid U_\tau)}\cdot\Pr((a,b)\in E(t-\tau)\mid U_\tau)\cr
    &=\frac{\Pr(N_\tau=0)}{\Pr(N_\tau=0)}\cdot\Pr((a,b)\in E(t-\tau)\mid U_\tau)\cr
    &\stackrel{(c)}{=}\frac{\rho_{ij}}{n},
\end{align*}
where (a) follows from the definition of $T$, $(b)$ follows from Bayes' rule, and $(c)$ follows from the model definition (Section~\ref{sec:formulation}). Thus, $\Pr(1_{(a,b)}(t)=1\mid T=\tau)$ and $\Pr(1_{(a,b)}(t)=0\mid T=\tau)$ do not depend on $\tau$, which means that $T$ and $1_{(a,b)}(t)$ are independent.
\end{proof}

The proof of Proposition~\ref{prop:bounds} is based on the concepts of transition sequences and agnostic transition sequences, which we define below.
\begin{definition} [Transition Sequence] Consider any time $t\geq 0$, integer $r\in\N_0$, tuples $\x^{(1)},\x^{(2)},\ldots, \x^{(r)}\in\mathbb S$, and times $0\leq t_1<t_2<\cdots<t_r\leq t$. Let $F=\{\x^{(0)}\stackrel{t_1}{\rightarrow}\x^{(1)}\stackrel{t_2}{\rightarrow}\cdots \stackrel{t_r}{\rightarrow} \x^{(r)}\stackrel{t}{\rightarrow}\x^{(r)}\}$ denote the event that the embedded jump chain $\{\X(J_\ell):\ell\in\N_0\}$ satisfies $\X(J_\ell)=\x^{(\ell)}$ and $J_\ell = t_\ell$ for all $\ell\in [r]$, and $J_{r+1}>t$. Then $F$ is said to be a \textit{transition sequence} for the time interval $[0,t]$. 
\end{definition} 

Note that if $F$ is a transition sequence for $[0,t]$, then for every tuple $\x\in\mathbb S$, we either have $F\subset \{\X(t)=\x\}$ or $F\subset \{\X(t)\neq\x\}$.

\begin{definition} [$(a,b)$-Complement] Let $\x\in\mathbb S$. Then the $(a,b)$\textit{-complement} of $\x$, denoted by $\x\abcomp$, is defined by
\begin{align*}
    (\x\abcomp)_\ell=
    \begin{cases}
        x_\ell\quad &\text{if }\ell\in [2n^2-n]\setminus\{\langle a,b\rangle\},\\
        1-x_\ell\quad&\text{if }\ell =\langle a,b\rangle.
    \end{cases}
\end{align*}
\end{definition}

\begin{definition} [$(a,b)$-Agnostic Transition Sequence] Let $a,b\in[n]$, and let $F=\{\x^{(0)}\stackrel{t_1}{\rightarrow}\x^{(1)}\stackrel{t_2}{\rightarrow}\cdots \stackrel{t_r}{\rightarrow} \x^{(r)}\stackrel{t}{\rightarrow}\x^{(r)}\}$ be a transition sequence for a time interval $[0,t]$. Further, define 
$$
    \Lambda_{(a,b)}(F)=
    \begin{cases}
        \max\left\{\ell\in [r]: \x^{(\ell)}\in\{\x^{(\ell-1)}_{\existence(a,b)},\x^{(\ell-1)}_{\nonexistence(a,b)}\}  \right\} \quad &\text{if } \left\{\ell\in [r]: \x^{(\ell)}\in\{\x^{(\ell-1)}_{\existence(a,b)},\x^{(\ell-1)}_{\nonexistence(a,b)}\}  \right\}\neq\emptyset\\
        0 \quad &\text{otherwise.}
    \end{cases}
$$
Then the $(a,b)$\textit{-agnostic transition sequence for} $F$ is the event $F\abagn=F\cup F\abcomp$, where $F\abcomp$, defined by
\begin{align*}
     F\abcomp:=&\Big\{\x^{(0)}\stackrel{t_1}{\rightarrow}\cdots\stackrel{t_{\Lambda_{(a,b)}(F)-1  } }{\rightarrow} \x^{({\Lambda_{(a,b)}(F)-1  })}\stackrel{t_{\Lambda_{(a,b)}(F) } }{\rightarrow}\x^{({\Lambda_{(a,b)}(F) })}\abcomp\\
    &\quad\stackrel{t_{\Lambda_{(a,b)}(F)+1 } }{\rightarrow} \x^{({\Lambda_{(a,b)}(F) +1 })}\abcomp\stackrel{t_{\Lambda_{(a,b)}(F)+2} }{\rightarrow}\cdots\stackrel{t_r}{\rightarrow} \x^{(r)}\abcomp\stackrel{t}{\rightarrow}\x^{(r)}\abcomp \Big\},
\end{align*}
is called the $(a,b)$\textit{-complement} of $F$.
\end{definition}

Given that $F$ occurs, $t_{\Lambda_{(a,b)}(F)}$ denotes the time at which the edge state of $(a,b)$ is updated for the last time during the time interval $[0,t]$ (note that, if the edge state of $(a,b)$ is not updated during $[0,t]$, then $t_{\Lambda_{(a,b)}(F)}=t_0:=0$). Therefore, the only difference between $F$ and $F\abcomp$ is that the last edge state of $(a,b)$ to be realized during the interval $[0,t]$ is different for $F$ and $F_{(a,b)}$. Stated differently, if $F\subset \{(a,b)\in E(t)\}$, then  $F\abcomp\subset\{(a,b)\notin E(t)\}$, and vice-versa. As a result, the event $F\abagn$ is $(a,b)$-agnostic in that the occurrence of this event does not provide any information about the edge state of $(a,b)$ at time $t$. 

Note that if $F$ is a transition sequence, then $F$, $F\abcomp$, and $F\abagn$ are all zero-probability events. We now approximate these events with the help of suitable positive-probability events.

\begin{definition} [$\delta$-Approximation] Let $F=\{\x^{(0)}\stackrel{t_1}{\rightarrow}\x^{(1)}\stackrel{t_2}{\rightarrow}\cdots \stackrel{t_r}{\rightarrow} \x^{(r)}\stackrel{t}{\rightarrow}\x^{(r)}\}$ be a transition sequence. Then, for a given $\delta >0$, the $\delta$\textit{-approximation} of $F$ is the event $F^\delta:=\{\X(0)=\x^{(0)}, \ldots, \X(J_r)=\x^{(r)}, \Delta_1 J\in [\Delta_1 t,\Delta_1 t+\delta), \ldots, \Delta _r J\in [\Delta_r t,\Delta_r t+\delta), \Delta_{r+1} J >t-t_r\}$, where $\Delta_\ell J:= J_\ell - J_{\ell-1}$, $\Delta_\ell t:=t_\ell - t_{\ell-1}$, and $t_0:=0$. Also, the $\delta$-approximation of $F\abagn$ is the event $F\abagn^\delta:= F^\delta \cup F\abcomp^\delta$ (where $F^\delta\abcomp$ is the $\delta$-approximation of $F\abcomp$).
\end{definition}

The following lemma evaluates the probability of occurrence of a $\delta$-approximation event.
\begin{lemma} \label{lem:delta_approx}
Let $F=\{\x^{(0)}\stackrel{t_1}{\rightarrow}\x^{(1)}\stackrel{t_2}{\rightarrow}\cdots \stackrel{t_r}{\rightarrow} \x^{(r)}\stackrel{t}{\rightarrow}\x^{(r)}\}$ be a transition sequence. Then for all sufficiently small $\delta>0$, the ratio $\frac{\Pr(F^\delta)}{\Pr(\X(0)=\x^{(0)} )} $ equals
$$
     e^{-q_r (t-t_r)}\prod_{\ell=1}^{r} q_{\ell-1,\ell} (e^{-q_{\ell-1}( t_{\ell} - t_{\ell-1})}\delta + o(\delta)),
$$
where $t_0:=0$, $q_{\ell}:=-\Q(\x^{(\ell)}, \x^{(\ell)})$, and $q_{\ell, \ell+1}:= \Q(\x^{(\ell)}, \x^{(\ell+1)})$ for each $\ell\in [r]$.
\end{lemma}

\begin{proof}
Observe that 
\begin{align*}
    \frac{\Pr(F^\delta)}{\Pr(\X(0)=\x^{(0)})}&=\Pr(\cap_{\ell\in[r] } \{\X(J_\ell)=\x^{(\ell)},\Delta_\ell J \in [\Delta_\ell t, \Delta_\ell t +\delta) \}\mid \X(0)=\x^{(0)})\cr
    &\quad\times\Pr(\Delta_{r+1} J>t-t_r\mid \cap_{\ell\in[r] } \{\X(J_\ell)=\x^{(\ell)},\Delta_\ell J \in [\Delta_\ell t, \Delta_\ell t +\delta) \})\cr
    &\stackrel{(a)}{=}\prod_{\ell=1}^r \Pr(\Delta_\ell J\in [\Delta_\ell t,\Delta_\ell t+\delta),  \X(J_{\ell})=\x^{(\ell)}\mid \X(J_{\ell-1})=\x^{(\ell-1)})\cr
    &\quad\times \Pr(\Delta_{r+1} J>t-t_r\mid \X(J_r)=\x^{(r)})\cr
    &\stackrel{(b)}{=}\prod_{\ell=1}^r \Pr(\Delta_1 J\in [\Delta_\ell t,\Delta_\ell t+\delta),  \X(J_1)=\x^{(\ell)}\mid \X(0)=\x^{(\ell-1)})\cr
    &\quad\times \Pr(\Delta_{1} J>t-t_r\mid \X(0)=\x^{(r)})\cr
    &\stackrel{(c)}{=}\prod_{\ell=1}^r \Big(\Pr(\Delta_1 J\in [\Delta_\ell t, \Delta_\ell t+\delta)\mid \X(0)=\x^{(\ell-1)} )\cdot \Pr(\X(J_1)=\x^{(\ell)}\mid \X(0)=x^{(\ell-1)})\Big)\cr
    &\quad\times \Pr(\Delta_{1} J>t-t_r\mid \X(0)=\x^{(r)})\cr
    &\stackrel{(d)}{=}\prod_{\ell=1}^r \left(  \left(q_{\ell-1} e^{-q_{\ell-1}\Delta_\ell t}\delta + o(\delta)\right) \frac{q_{\ell-1,\ell} }{q_{\ell-1}} \right)\times e^{-q_r(t-t_r)}\cr
    &= e^{-q_r(t-t_r)}\prod_{\ell=1}^r q_{\ell-1,\ell} \left( e^{-q_{\ell-1}(t_\ell - t_{\ell-1}) }\delta+o(\delta)\right),
\end{align*}
where $(a)$ follows from the strong Markov property and the fact that jump times are stopping times, $(b)$ follows from Proposition 3.2 of~\cite{lalley}, $(c)$ follows from the fact that $\X(J_1)$ and $\Delta_1 J$ (which equals $J_1$) are conditionally independent given $\X(0)$ (see Proposition 3.1 of~\cite{lalley}), and $(d)$ follows from the following two facts:
\begin{enumerate}
    \item $\Delta_1 J$ is conditionally exponentially distributed with mean $q_{\ell-1}^{-1}$ given that $\X(0)=\x^{(\ell-1)}$.
    \item For the embedded jump chain, the probability of transitioning from $\x\in\mathbb S$ to $\y\in\mathbb S$ is $\frac{\Q(\x,\y)}{|\Q(\x,\x)|}$.
\end{enumerate}

\end{proof}

For the rest of the appendix, let $a\in\A_i$ and $b\in\A_j$ be any two nodes, let $t\in[0,\infty)$ be any time instant, let $T\in [0,t]$ be the random variable such that $t-T$ is the time at which $1_{(a,b)}$ is updated for the last time during the interval $[0,t)$, and let $K:=t-\min\{t,\inf\{\tau:b\in\I(\tau)\}\}$, where $\inf\{\tau: b\in \I(\tau) \}$ is the time at which $b$ gets infected.  

\begin{lemma}\label{lem:pdf_of_T}
The PDF of $T$ has $[0,t]$ as its support and is given by
$$
    f_T(\tau) = \lambda e^{-\lambda \tau}+e^{-\lambda t}\delta_D(\tau-t),
$$
where $\delta_D(\cdot)$ is the Dirac-delta function.
\end{lemma}

\begin{proof}
The definition of $T$ implies that the support of its PDF is $[0,t]$. To derive the required closed-form expression for this PDF, recall that $\Q(\x,\x_{\existence(a,b)})+\Q(\x,\x_{\nonexistence(a,b)})=\lambda$ for all $\x\in\mathbb S$, which means that the edge state of $(a,b)$ is updated at a constant rate of $\lambda$ at all times. Therefore, for any $\tau\in[0,t)$, the quantity $\Pr(T>\tau)$ (the probability that $1_{(a,b)}$ is not updated during $[t-\tau,t]$) is given by $e^{-\lambda\tau}$. However, $\Pr(T>t)=0$, implying that $\Pr(T=t)=\Pr(T\geq t)=\lim_{\tau\rightarrow t^-}\Pr(T>\tau)=e^{-\lambda t}$. Hence, the CDF of $T$ is  $F(\tau)=1-e^{-\lambda\tau}$ for all $\tau\in [0,t)$, and $F(t)=1$. Taking the first derivative of this CDF now yields the required expression for $f_T$.
\end{proof}

To prove the next lemma, we need the notion of agnostic superstates, which is defined below.

\begin{definition} [$(a,b)$-Agnostic Superstate] Given a node pair $(a,b)\in [n]\times [n]$, a collection of states $\mathbb X\subset \mathbb S$ is an $(a,b)$\emph{-agnostic superstate} if $\mathbb X$ can be expressed as $\mathbb X=\left\{\x, \x\abcomp, \y, \y\abcomp\right\}$ for a pair of states $\x,\y\in\mathbb S$ satisfying $y_{n^2+\langle a,b\rangle}=1-x_{n^2 + \langle a,b \rangle }$ and $x_\ell = y_\ell$ for all $\ell\in [2n^2-n]\setminus\{n^2 + \langle a,b\rangle \}$.
\end{definition}

Note that an $(a,b)$-agnostic superstate specifies the disease states of all the nodes and the edge states of all the node pairs except $(a,b)$.

\begin{definition} [$(a,b)$-Agnostic Jump Times] Given $(a,b)\in [n]\times [n]$, the $(a,b)$\emph{-agnostic jump times of the chain $\{\X(\tau):\tau\geq 0\}$}, denoted by $\{L_k\}_{k=0}^\infty$, are defined by $L_0:=0$ and $L_k:=\inf\{J_\ell: \ell\in \N, J_\ell> L_{k-1}, \X(J_\ell)\notin\{ (\X(J_{\ell-1}))_{\existence(a,b)}, (\X(J_{\ell-1}))_{\nonexistence(a,b)}\} \}$ for all $k\in\N$.
\end{definition}

Note that $\{L_k\}_{k=0}^\infty\subset\{J_\ell\}_{\ell=0}^\infty$ and that the $(a,b)$-agnostic jump times of $\{\X(\tau)\}$ are the jump times of the chain at which the edge state of $(a,b)$ is not updated.

\begin{lemma}\label{lem:simple_but_difficult} $K$ is independent of $(T,1_{(a,b)}(t))$.
\end{lemma}

\begin{proof}
Note that $K$ is a function of $\tilde K:=\inf\{\tau\geq 0:b\in\I(\tau)\}$, the time at which $b$ gets infected. Hence, it suffices to prove that $\tilde K$ is independent of $(T,1_{(a,b)}(t))$.

Consider now any $\kappa\geq 0$ and note that $\{\tilde K\geq \kappa\}=\cup_{N=0}^\infty \left(\{\tilde K=L_N\}\cap\{L_N\geq \kappa\}\right)$. To examine the probability of $\{\tilde K=L_N\}$, we let $\mathcal F_N(\kappa)$ denote the set of all the events of the form $F=\{\X(0)\in \XX^{(0)}, \X(L_1)\in \XX^{(1)},\ldots, \X(L_N)\in\XX^{(N)}\}$ (where $\XX^{(0)},\ldots, \XX^{(N)}$ are $(a,b)$-agnostic superstates satisfying $\XX^{(k)}\neq \XX^{(k-1)}$ for all $k\in [N]$) that satisfy $F\subset \{\tilde K=L_N\}$, and we observe that $\{\tilde K=L_N\}=\cup_{F\in\F_{N}(\kappa)}F$.

We now examine $\Pr(F)$ for an arbitrary $F=\{\X(0)\in \XX^{(0)}, \X(L_1)\in \XX^{(1)},\ldots, \X(L_N)\in\XX^{(N)}\}\in\F_N(\kappa)$. Pick any $k\in[N]$ and $\x\in\XX^{(k-1)}$. Note that $F\subset\{\tilde K=L_N\}$ implies that $b\in\S(\x)$. In view of our definition of $\Q$, this means that $\Q(\x,\x)=-\sum_{\mathbf z\in\mathbb S\setminus\{\x\}}\Q(\x,\mathbf z)$, which possibly depends on the disease states $\{x_1,\ldots, x_n\}$ and on the edge states $\{1_{(c,d)}(\x): (c,d)\in[n]\times [n]:d\in\I(\x)\}$, does not depend on $1_{(a,b)}(\x)$. We next observe that, by the definitions of $(a,b)$-agnostic states and jump times, none of the possible transitions from $\XX^{(k-1)}$ to $\XX^{(k)}$ involves an edge state update for $(a,b)$, 
which means that the values of both $X_{\langle a,b\rangle}=1_{(a,b)}(\X)$ and $X_{n^2+\langle a,b\rangle}$ are preserved in such transitions. Therefore, for every $\x\in \XX^{(k-1)}$, there exists at most one state $\y \in\XX^{(k)}$ that potentially succeeds $\x$. For such a state $\y$, the transition rate $\Q(\x,\y)$, given by 
\begin{align*}
     \Q(\x,\y) =
     \begin{cases}
      \sum_{q=1}^m \sum_{d\in\I_q(\x)} B_{pq }1_{(c,d)}(\x) \, &\text{if }\exists\,\, c\in\S_p(\x) \text{ such that }\y = \x_{\infect c} \\
      \gamma_p \quad &\text{if }\exists\,\, c\in\I_p(\x) \text{ such that }\y = \x_{\recover c},\\
      \lambda \frac{\rho_{pq}}{n} \quad &\text{if }\exists\,\, (c,d)\in\A_p\times \A_q\setminus \{(a,b)\} \text{ such that } \y = \x_{\existence (c,d),} \\
      \lambda\left(1- \frac{\rho_{pq}}{n}\right) \quad &\text{if }\exists\,\, (a,b)\in\A_p\times \A_q\setminus \{(a,b)\} \text{ such that } \y=\x_{\nonexistence(c,d)},
     \end{cases}
 \end{align*}
 does not depend on $x_{\langle a,b\rangle}=y_{\langle a,b\rangle}$ or on $x_{n^2+\langle a,b\rangle}=y_{n^2+\langle a,b\rangle}$, because $b\notin \I(\x)$ implies that $(a,b)\notin\cup_{p=1}^m\cup_{c\in\S_p(\x)}\cup_{q=1}^m\{(c,d):d\in\I_q(\x)\}$. It follows that the rate at which the Markov chain $\{\X(\tau)\}$ transitions from $\x$ to a state in $\XX^{(k)}$, given by $\sum_{\mathbf z\in\XX^{(k)}}\Q(\x, \mathbf z)=\Q(\x,\y)$, is time-invariant and takes the same value for every $\x\in\XX^{(k-1)}$. This means that, as long as the Markov chain $\{\X(\tau):\tau\geq 0\}$ does not leave the $(a,b)$-agnostic superstate $\XX^{(k-1)}$, the rate at which the chain transitions to $\XX^{(k)}$ remains the same regardless of transitions within $\XX^{(k-1)}$. We can express this formally as 
 $$
    \lim_{\Delta \tau\rightarrow 0}\frac{\Pr(\X(\tau+\Delta \tau)\in\XX^{(k)}\mid \X(\tau)\in\XX^{(k-1) }, \X(\tau)=\x)}{\Delta \tau}=\Q(\XX^{(k-1)},\XX^{(k)})
 $$
 for all $\tau\geq 0$ and $\x\in\XX^{(k-1)}$, where
 $$
    \Q(\XX^{(k-1)},\XX^{(k)}):= \lim_{\Delta \tau\rightarrow 0}\frac{\Pr(\X(\tau+\Delta \tau)\in\XX^{(k)}\mid \X(\tau)\in\XX^{(k-1) })}{\Delta \tau}
 $$
 denotes the time-invariant rate of transitioning from  $\XX^{(k-1)}$ to $\XX^{(k)}$.
 
 
 By Markovity, this implies that\footnote{In this paper, conditioning an event $H$ on $\{\X(\tau'):0\leq \tau'\leq \tau\}$ means conditioning $H$ on every $\X(\tau')$ for $0\leq \tau'\leq \tau$, i.e., conditioning H on the trajectory traced by the Markov chain during the interval $[0,\tau]$ and not just on the random set of tuples $\{\X(\tau'):0\leq \tau'\leq \tau\}$. Conditioning on the set $\{\X(\tau'):0\leq \tau'\leq \tau\}$ is not sufficient because sets, by definition, are unordered.}
 $$
    \lim_{\Delta \tau\rightarrow 0}\frac{\Pr(\X(\tau+\Delta \tau)\in\XX^{(k)}\mid \X(\tau)=\x, \{\X(\tau'):0\leq \tau'< \tau\}) }{\Delta\tau} = \Q(\XX^{(k-1)},\XX^{(k)} ),
 $$
 for all $\tau\geq 0$ and $\x\in\XX^{(k-1)}$, which means that the conditional rate at which the chain leaves $\XX^{(k-1)}$ time $\tau$ is independent of the history $\{\X(\tau'):\tau'\in[0,\tau]\}$. Since $\{1_{(a,b)}(\tau'):0\leq \tau'\leq \tau\}$ are determined by $\{\X(\tau'):0\leq \tau'\leq \tau\}$, it follows that
 $$
    \lim_{\Delta \tau\rightarrow 0}\frac{\Pr(\X(\tau+\Delta \tau)\in\XX^{(k)}\mid \X(\tau)=\x, \{1_{(a,b)}(\tau'):0\leq \tau'\leq \tau\}) }{\Delta\tau} = \Q(\XX^{(k-1)},\XX^{(k)} )
 $$
 for all $\tau\geq 0$ and $\x\in\XX^{(k-1)}$. Now, let $T_\existence(\tau):=\inf\{\tau'\geq 0:\X(\tau+\tau')=(\X((\tau+ \tau')^-))_{\existence(a,b)}\}$ be the (random) time elapsed between time $\tau$ and the first of the updates of $1_{(a,b)}$ that occur after time $\tau$ and result in $1_{(a,b)}=1$. Then,  the following holds for all $\x\in\XX^{(k-1)}$, $\tau\geq 0$,  $\sigma>0$ and sufficiently small $\Delta\tau>0$:
 \begin{align*}
    &\Pr(T_\uparrow(\tau)\geq \sigma\mid \X(\tau)=\x,\{1_{(a,b)}(\tau'):0\leq \tau'\leq \tau\}, \X(\tau+\Delta \tau)\in \XX^{(k)})\cr
    &=\Pr(T_\uparrow(\tau+\Delta\tau)\geq \sigma-\Delta \tau\mid \X(\tau)=\x,\{1_{(a,b)}(\tau'):0\leq \tau'\leq \tau\}, \X(\tau+\Delta \tau)\in \XX^{(k)}, T_\uparrow(\tau)\geq \Delta \tau )\cr
    &\quad\cdot \Pr(T_\uparrow(\tau)\geq \Delta \tau\mid \X(\tau)=\x,\{1_{(a,b)}(\tau'):0\leq \tau'\leq \tau\}, \X(\tau+\Delta \tau)\in \XX^{(k)})\cr
    &=\Pr(T_\uparrow(\tau+\Delta\tau)\geq \sigma-\Delta \tau\mid \X(\tau)=\x,\{1_{(a,b)}(\tau'):0\leq \tau'\leq \tau\}, \X(\tau+\Delta \tau)\in \XX^{(k)},\cr
    &\quad\quad\quad\quad\quad\quad\quad\quad\quad\quad\quad\quad\quad X_{n^2 + \langle a,b\rangle}(\tau')= X_{n^2+\langle a,b\rangle }(\tau)\,\forall\,\tau'\in [\tau,\tau+\Delta \tau) )\cr
    &\quad\cdot \Pr(T_\uparrow(\tau)\geq \Delta \tau\mid \X(\tau)=\x,\{1_{(a,b)}(\tau'):0\leq \tau'\leq \tau\}, \X(\tau+\Delta \tau)\in \XX^{(k)})\cr
    &\stackrel{(a)}= e^{-\lambda\frac{\rho_{ij} }{n}(\sigma-\Delta \tau)}\cdot \Pr(T_\uparrow(\tau)\geq \Delta \tau\mid \X(\tau)=\x,\{1_{(a,b)}(\tau'):0\leq \tau'\leq \tau\}, \X(\tau+\Delta \tau)\in \XX^{(k)})\cr
    &\stackrel{\Delta \tau\rightarrow 0}{\longrightarrow} e^{-\lambda\frac{\rho_{ij} }{n}\sigma }\cdot 1\cr
    &\stackrel{(b)}=\Pr(T_\uparrow(\tau)\geq \sigma\mid \X(\tau)=\x,\{1_{(a,b)}(\tau'):0\leq \tau'\leq \tau\}),
 \end{align*}
 where $(a)$ and $(b)$ follow from Markovity and the fact that $\Q(\mathbf z,\mathbf z_{\existence(a,b)})=\lambda\frac{\rho_{ij} }{n}$ for all $\mathbb z\in \mathbb S$. Now, let $T^{(1)}_{\existence}(\tau):=T_{\existence}(\tau)$ and $T^{(\ell)}_{\existence}(\tau):=T_\existence(T^{(\ell-1)}_{\existence }(\tau))$ for all $\ell\in\N$. Then, since $\{T_{\existence}^{(\ell)}\}_{\ell=1}^\infty$ are stopping times, similar arguments can be used to show the following for all $\sigma_1,\sigma_2,\ldots, \sigma_\ell\geq 0$ and all $\ell\in\N$:
 \begin{align*}
     \lim_{\Delta \tau\rightarrow0}& \Pr\left(T_\existence^{(1)}(\tau)\geq \sigma_1,\ldots, T_\existence^{(\ell)}(\tau)\geq \sigma_\ell \mid \X(\tau)=\x,\{1_{(a,b)}(\tau'):0\leq \tau'\leq \tau\}, \X(\tau+\Delta \tau)\in \XX^{(k)}\right)\cr
     &=\Pr\left(T_\existence^{(1)}(\tau)\geq \sigma_1,\ldots, T_\existence^{(\ell)}(\tau)\geq \sigma_\ell \mid \X(\tau)=\x,\{1_{(a,b)}(\tau'):0\leq \tau'\leq \tau\}\right).
 \end{align*}
 Similarly, if we let $T_\nonexistence^{(1)}:= \inf\{\tau'\geq 0:\X(\tau+\tau')=(\X((\tau+ \tau')^-))_{\nonexistence(a,b)}\}$ and $T^{(\ell)}_{\nonexistence}(\tau):=T_\nonexistence(T^{(\ell-1)}_{\nonexistence }(\tau))$ for all $\ell\in\N$, then we can show that for all $\sigma_{\existence1},\ldots,\sigma_{\existence\ell},\sigma_{\nonexistence_1},\ldots,\sigma_{\nonexistence\ell'}\geq 0$ and all $\ell,\ell'\in \N$,
  \begin{align*}
     &\lim_{\Delta \tau\rightarrow0} \Pr\Big(T_\existence^{(1)}(\tau)\geq \sigma_{\existence 1},\ldots, T_\existence^{(\ell)}(\tau)\geq \sigma_{\existence\ell}, T_\nonexistence^{(1)}(\tau)\geq \sigma_{\nonexistence1},\ldots, T_\nonexistence^{(\ell')}(\tau)\geq \sigma_{\nonexistence\ell'}\cr
     &\quad\quad\quad\quad\mid \X(\tau)=\x,\{1_{(a,b)}(\tau'):0\leq \tau'\leq \tau\}, \X(\tau+\Delta \tau)\in \XX^{(k)}\Big)\cr
     &=\Pr\Big(T_\existence^{(1)}(\tau)\geq \sigma_{\existence 1},\ldots, T_\existence^{(\ell)}(\tau)\geq \sigma_{\existence\ell}, T_\nonexistence^{(1)}(\tau)\geq \sigma_{\nonexistence1},\ldots, T_\nonexistence^{(\ell')}(\tau)\geq \sigma_{\nonexistence\ell'}\cr
     &\quad\quad\quad\mid \X(\tau)=\x,\{1_{(a,b)}(\tau'):0\leq \tau'\leq \tau\}\Big).
 \end{align*}
 As a result, we have the following for all $\sigma_{\existence1},\ldots,\sigma_{\existence\ell},\sigma_{\nonexistence_1},\ldots,\sigma_{\nonexistence\ell'}\geq 0$ and all $\ell,\ell'\in \N$:
 \begin{align*}
     &\frac{\Pr\Big(  \X(\tau+\Delta \tau)\in \XX^{(k)}  \mid \X(\tau)=\x,\{1_{(a,b)}(\tau')\}_{\tau'\in [0,\tau]}, \{T_\existence^{(\xi)}(\tau)\geq \sigma_{\existence\xi}\}_{\xi=1}^\ell,  \{T_\nonexistence^{(\xi)}(\tau)\geq \sigma_{\nonexistence\xi}\}_{\xi=1}^{\ell'} \Big) }{\Delta \tau} \cr
     &\stackrel{(a)}=\frac{ \Pr\Big( \{T_\existence^{(\xi)}(\tau)\geq \sigma_{\existence\xi}\}_{\xi=1}^\ell,  \{T_\nonexistence^{(\xi)}(\tau)\geq \sigma_{\nonexistence\xi}\}_{\xi=1}^{\ell'}    \mid \X(\tau)=\x,\{1_{(a,b)}(\tau')\}_{\tau'\in [0,\tau]}, \X(\tau+\Delta \tau)\in \XX^{(k)} \Big) }{\Pr\Big( \{T_\existence^{(\xi)}(\tau)\geq \sigma_{\existence\xi}\}_{\xi=1}^\ell,  \{T_\nonexistence^{(\xi)}(\tau)\geq \sigma_{\nonexistence\xi}\}_{\xi=1}^{\ell'}    \mid \X(\tau)=\x,\{1_{(a,b)}(\tau')\}_{\tau'\in [0,\tau]}\Big)}\cr
     &\quad\times \frac{\Pr\Big(  \X(\tau+\Delta \tau)\in \XX^{(k)}  \mid \X(\tau)=\x,\{1_{(a,b)}(\tau')\}_{\tau'\in [0,\tau]}\Big)  }{ \Delta \tau }\cr
     &\stackrel{\Delta\tau\rightarrow0}\longrightarrow 1\times \Q(\XX^{(k-1)},\XX^{(k)}),
 \end{align*}
 i.e., 
 \begin{gather*}
     \lim_{\Delta\tau\rightarrow0}\frac{\Pr\Big(  \X(\tau+\Delta \tau)\in \XX^{(k)}  \mid \X(\tau)=\x,\{1_{(a,b)}(\tau')\}_{\tau'\in [0,\tau]}, \{T_\existence^{(\xi)}(\tau)\}_{\xi=1}^\ell,  \{T_\nonexistence^{(\xi)}(\tau)\}_{\xi=1}^{\ell'} \Big) }{\Delta \tau} \cr
     = \Q(\XX^{(k-1)},\XX^{(k)})
 \end{gather*}
 for all $\ell,\ell'\in\N$. Now, observe that if we are given $\{1_{(a,b)}(\tau'):0\leq \tau'\leq \tau\}$, then $\{1_{(a,b)}(\tau'):\tau\leq \tau'\leq t\}$ are determined by a subset of the random variables $\{T_\existence^{(\ell)}\}_{\ell=1}^\infty \cup\{T_{\nonexistence}^{(\ell)}\}_{\ell=1}^\infty$ and this subset is random but almost surely finite. Hence, the above limit implies the following for all $\x\in\XX^{(k-1)}$ and $\tau\geq 0$:
 $$
    \lim_{\Delta \tau\rightarrow 0}\frac{\Pr\left(\X(\tau+\Delta\tau)\in\XX^{(k)}\mid \X(\tau)=\x,\{1_{(a,b)}(\tau'):0\leq \tau'\leq t\} \right)  }{\Delta\tau}=\Q(\XX^{(k-1)},\XX^{(k)}).
 $$
 Moreover, since the above arguments remain valid if we replace $\XX^{(k)}$ with an arbitrary $(a,b)$-agnostic superstate $\mathbb Y\neq\XX^{(k-1)}$, we can generalize the above to 
  $$
    \lim_{\Delta \tau\rightarrow 0}\frac{\Pr\left(\X(\tau+\Delta\tau)\in\mathbb Y\mid \X(\tau)=\x,\{1_{(a,b)}(\tau'):0\leq \tau'\leq t\} \right)  }{\Delta\tau}=\Q(\XX^{(k-1)},\mathbb Y )
 $$
 for all $(a,b)$-agnostic superstates $\mathbb Y\neq\XX^{(k-1)}$. It follows that
 \begin{align}\label{eq:time_cond_indep}
     \lim_{\Delta \tau\rightarrow 0}\frac{\Pr\left(\X(\tau+\Delta\tau)\notin\XX^{(k-1)}\mid \X(\tau)=\x,\{1_{(a,b)}(\tau'):0\leq \tau'\leq t\} \right)  }{\Delta\tau}=\sum_{\mathbb Y\neq \XX^{(k-1)}}\Q(\XX^{(k-1)},\mathbb Y )
 \end{align}
 for all $\x\in\XX^{(k-1)}$ and all $\tau\geq 0$. This means that, given $\{\X(L_{k-1})=\x\}$ for some $\x\in\XX^{(k-1)}$, the random quantity $L_{k}-L_{k-1}$, which is the duration of time spent by the Markov chain in $\XX^{(k-1)}$, is conditionally exponentially distributed with rate $\sum_{\mathbb Y\neq \XX^{(k-1)}}\Q(\XX^{(k-1)},\mathbb Y )$ and it is conditionally independent of $\{1_{(a,b)}(\tau'):0\leq \tau'\leq t\}$. Besides, the above deductions also imply the following: given $\{1_{(a,b)}(\tau'):0\leq\tau'\leq t\}$ and given  that the chain exits $\XX^{(k-1)}$ from state $\x$ at time $\tau\geq 0$, the conditional probability that it enters $\XX^{(k)}$ at time $\tau$ is 
\begin{align*}
     &\Pr(\X(\tau)\in\XX^{(k)}\mid \X(\tau^-)=\x,\{1_{(a,b)}(\tau'):0\leq \tau'\leq t\}, \X(\tau)\notin\XX^{(k-1)})\cr   
     &=\lim_{\Delta\tau\rightarrow0}\Pr(\X(\tau)\in\XX^{(k)}\mid \X(\tau-\Delta\tau)=\x,\{1_{(a,b)}(\tau'):0\leq \tau'\leq t\}, \X(\tau)\notin\XX^{(k-1)})\cr
     &=\lim_{\Delta \tau\rightarrow 0}\frac{\Pr(\X(\tau)\in\XX^{(k)}\mid \X(\tau-\Delta\tau)=\x,\{1_{(a,b)}(\tau'):0\leq \tau'\leq t\}) }{\Pr(\X(\tau)\notin\XX^{(k-1)}\mid \X(\tau-\Delta\tau)=\x,\{1_{(a,b)}(\tau'):0\leq \tau'\leq t\})}\cr
     &=\lim_{\Delta\tau\rightarrow0}\frac{\Pr(\X(\tau)\in\XX^{(k)}\mid \X(\tau-\Delta\tau)=\x,\{1_{(a,b)}(\tau'):0\leq \tau'\leq t\})}{\Delta\tau}\cr
     &\quad\times\lim_{\Delta\tau\rightarrow0}\left(\frac{  \Pr(\X(\tau)\notin\XX^{(k-1)}\mid \X(\tau-\Delta\tau)=\x,\{1_{(a,b)}(\tau'):0\leq \tau'\leq t\}) }{\Delta\tau}\right)^{-1}\cr
     &=\frac{ \Q(\XX^{(k-1)},\XX^{(k)})  }{ \sum_{\mathbb Y\neq \XX^{(k-1)}} \Q(\XX^{(k-1)},\mathbb Y )}.
 \end{align*}
 By invoking Markovity in the preceding arguments, the above can be generalized to 
 \begin{align}\label{eq:state_cond_indep}
     \Pr(\X(\tau)\in \XX^{(k)}\mid \X(\tau^-)=\x,\{\X(\tau')\}_{\tau'\in[0,\tau)}, \{1_{(a,b)}(\tau')\}_{\tau'\in[0,t]}, \X(\tau)\notin \XX^{(k-1)})=\frac{ \Q(\XX^{(k-1)},\XX^{(k)})  }{ \sum_{\mathbb Y\neq \XX^{(k-1)}} \Q(\XX^{(k-1)},\mathbb Y )},
 \end{align}
 which implies that
 \begin{gather*}
    \Pr(\X(\tau)\in \XX^{(k)}\mid \X(\tau')=\x\,\forall\,\tau'\in [L_{k-1},\tau), \{\X(\tau')\}_{\tau'\in[0,L_{k-1})} \{1_{(a,b)}(\tau')\}_{\tau'\in[0,t]}, \X(\tau)\notin \XX^{(k-1)} )\cr
    =\frac{ \Q(\XX^{(k-1)},\XX^{(k)})  }{ \sum_{\mathbb Y\neq \XX^{(k-1)}} \Q(\XX^{(k-1)},\mathbb Y )}.
 \end{gather*}
 Equivalently, 
 $$
    \Pr(\X(L_k)\in\XX^{(k)}\mid \X(L_{k-1})=\x, \{\X(\tau')\}_{\tau'\in[0,L_{k-1})}, \{1_{(a,b)}(\tau')\}_{\tau'\in[0,t]},L_k=\tau)=\frac{ \Q(\XX^{(k-1)},\XX^{(k)})  }{ \sum_{\mathbb Y\neq \XX^{(k-1)}} \Q(\XX^{(k-1)},\mathbb Y )}.
 $$
 for all $\tau> 0$ and $\x\in\XX^{(k-1)}$. Hence,
 \begin{align}\label{eq:state_cond_indep}
     \Pr(\X(L_k)\in \XX^{(k)}\mid  \X(L_{k-1})\in\XX^{(k-1)},\{\X(\tau')\}_{\tau'\in[0,L_{k-1})},\{1_{(a,b)}(\tau')\}_{\tau'\in[0,t]}, L_k)=\frac{ \Q(\XX^{(k-1)},\XX^{(k)})  }{ \sum_{\mathbb Y\neq \XX^{(k-1)}} \Q(\XX^{(k-1)},\mathbb Y )}.
 \end{align}
 Since the entire analysis above holds for all $k\in [N]$, we have the following for all $\sigma_1,\sigma_2,\ldots, \sigma_N\geq 0$:
 \begin{align}\label{eq:last_leg}
     &\Pr\left(\left(\cap_{k=1}^N\{L_k-L_{k-1}\geq \sigma_k\}\right)\cap\left(\cap_{k=0}^N\{\X(L_k)\in\XX^{(k)}\}\right)\mid \X(0)\in\XX^{(0)}, \{1_{(a,b)}(\tau')\}_{\tau'\in[0,t]} \right)\cr
     &=\prod_{k=1}^N\Pr(\X(L_{k})\in\XX^{(k)}, L_k - L_{k-1}\geq \sigma_k\mid  \{\X(L_{\xi})\in\XX^{(\xi)}, L_\xi-L_{\xi-1}\geq \sigma_\xi \}_{\xi=0}^{k-1},\{1_{(a,b)}(\tau')\}_{\tau'\in[0,t]} )\cr
     &=\prod_{k=1}^N\Pr( L_k - L_{k-1}\geq \sigma_k\mid  \{\X(L_{\xi})\in\XX^{(\xi)}, L_\xi-L_{\xi-1}\geq \sigma_\xi \}_{\xi=0}^{k-1},\{1_{(a,b)}(\tau')\}_{\tau'\in[0,t]} )\cr
     &\quad\times\prod_{k=1}^N\Pr(\X(L_{k})\in\XX^{(k)}\mid  L_k - L_{k-1}\geq \sigma_k,  \{\X(L_{\xi})\in\XX^{(\xi)}, L_\xi-L_{\xi-1}\geq \sigma_\xi \}_{\xi=0}^{k-1},\{1_{(a,b)}(\tau')\}_{\tau'\in[0,t]} )\cr
     &\stackrel{(a)}=\prod_{k=1}^N\Pr( L_k - L_{k-1}\geq \sigma_k\mid  \X(L_{k-1})\in\XX^{(k-1)},\{1_{(a,b)}(\tau')\}_{\tau'\in[0,t]} )\cr
     &\quad\times\prod_{k=1}^N\Pr(\X(L_k)\in\XX^{(k)}\mid  L_k - L_{k-1}\geq \sigma_k,  \{\X(L_{\xi})\in\XX^{(\xi)}, L_\xi-L_{\xi-1}\geq \sigma_\xi \}_{\xi=0}^{k-1},\{1_{(a,b)}(\tau')\}_{\tau'\in[0,t]} )\cr
     &\stackrel{(b)}= \prod_{k=1}^N \exp\left(-\sigma_k\sum_{\mathbb Y\neq \XX^{(k-1)}}\Q(\XX^{(k-1)}, \mathbb Y) \right)\times\prod_{k=1}^N \frac{\Q(\XX^{(k-1)},\XX^{(k)}) }{  \sum_{\mathbb Y\neq \XX^{(k-1)}}\Q(\XX^{(k-1)}, \mathbb Y)},
 \end{align}
 where $(a)$ is a consequence of the strong Markov property and the fact that $\{L_k\}_{k=1}^N$ are stopping times, and (b) follows from~\eqref{eq:time_cond_indep} and~\eqref{eq:state_cond_indep}. Since $\sigma_1,\ldots, \sigma_N$ are arbitrary and since the above expression is independent of $\{1_{(a,b)}(\tau):0\leq \tau\leq t\}$, we have shown that for the event
 $$
    F\cap\{L_N\geq \kappa\} = \left\{\X(0)\in\XX^{(0)},\ldots, \X(L_N) \in\XX^{(N)}, \sum_{k=1}^N (L_k-L_{k-1})\geq \kappa \right\},
 $$
 we have $\Pr(F\cap\{L_N\geq \kappa\}\mid \{1_{(a,b)}(\tau):0\leq \tau\leq t\}) =\Pr(F\cap\{L_N\geq \kappa\})$. As a result,
 \begin{align*}
     &\Pr(\{\tilde K=L_N\}\cap\{L_N\geq \kappa\}\mid \{1_{(a,b)}(\tau):0\leq \tau\leq t\})\cr
     &=\Pr\left(\cup_{F\in\F_N(\kappa)} (F\cap\{L_N\geq \kappa\})\mid \{1_{(a,b)}(\tau):0\leq \tau\leq t\} \right)\cr
     &\stackrel{(a)}=\sum_{F\in\F_N(\kappa)} \Pr\left(F\cap\{L_N\geq \kappa\}\mid \{1_{(a,b)}(\tau):0\leq \tau\leq t\} \right)\cr
     &=\sum_{F\in\F_N(\kappa)} \Pr(F\cap \{L_N\geq \kappa\})\cr
     &\stackrel{(b)}=\Pr\left(\cup_{F\in\F_N(\kappa)} (F\cap\{L_N\geq \kappa\})\right)\cr
     &=\Pr(\{\tilde K=L_N\}\cap\{L_N\geq \kappa\}),
 \end{align*}
 where $(a)$ and $(b)$ hold because the definition of $\F_N(\kappa)$ implies that  $\F_N(\kappa)$ is a collection of disjoint events. Since $\{\tilde K\geq \kappa\}=\cup_{N=0}^\infty\left(\{\tilde K=L_N\}\cap\{L_N\geq \kappa\}\right)$ and since $\{\tilde K=L_1\}\cap\{L_1\geq \kappa\}, \{\tilde K=L_2\}\cap\{L_2\geq \kappa\},\ldots$ are disjoint events, it follows that $\Pr(\tilde K\geq \kappa \mid \{1_{(a,b)}(\tau):0\leq \tau\leq t\})=\Pr(\tilde K\geq \kappa)$. Moreover, since $\kappa\geq 0$ is arbitrary, this means that $\tilde K$ is independent of $\mid \{1_{(a,b)}(\tau):0\leq \tau\leq t\}$. Finally, since $K$ and $(T,1_{(a,b)}(t))$ are functions of $\tilde K$ and $ \{1_{(a,b)}(\tau):0\leq \tau\leq t\}$, respectively, it follows that $K$ and $(T,1_{(a,b)}(t))$ are independent.
\end{proof}

\begin{remark}\label{rem:tilde_K}
Observe that in the proof of Lemma~\ref{lem:simple_but_difficult}, $\eqref{eq:last_leg}$ implies that the event $\{L_N\geq \kappa\}\cap\left(\cap_{k=0}^N\{\X(L_k)\in\XX^{(k)}\}\right)$ is independent of $\{1_{(a,b)}(\tau):0\leq \tau\leq t\}$ (since $L_N=\sum_{k=1}^N (L_{k}-L_{k-1})$ and since the initial state $\X(0)$ is assumed to be non-random). Note that this is true for all the choices of $(a,b)$-agnostic superstates $\{\XX^{(k)}\}_{k=0}^{N}$ that satisfy $\cap_{k=0}^N\{\X(L_k)\in\XX^{(k)}\}\subset\{\tilde K=L_N\}$ and hence also for all $\{\XX^{(k)}\}_{k=0}^{N}$ that satisfy $\cap_{k=0}^N\{\X(L_k)\in\XX^{(k)}\}\subset\{\tilde K=L_N\}\cap\{\X(\tilde K)\in\mathbb Y\}$, where $\mathbb Y$ is an arbitrary $(a,b)$-agnostic superstate. Now, let us by $\mathcal X$ the set of all $\{\XX^{(k)}\}_{k=0}^N$ satisfying $\cap_{k=0}^N\{\X(L_k)\in\XX^{(k)}\}\subset\{\tilde K=L_N\}\cap\{\X(\tilde K)\in\mathbb Y\}$, we have 
$$
    \cup_{\{\XX^{(k)}\}_{k=0}^N\in \mathcal X} \left( \cap_{k=0}^N\{\X(L_k)\in\XX^{(k)}\}  \right) = \{\tilde  K=L_N \}\cap\{\X(\tilde K)\in\mathbb Y\}.
$$
Then, by the preceding arguments we have \begin{align}
    &\Pr(\{\tilde K=L_N\}\cap \{\X(\tilde K)\in\mathbb Y\}\cap\{L_N\geq \kappa\} \mid \{1_{(a,b)}(\tau):0\leq \tau\leq t\})\cr
    &=\Pr\left(\cup_{\{\XX^{(k)}\}_{k=0}^N\in \mathcal X} \left( \cap_{k=0}^N\{\X(L_k)\in\XX^{(k)}\}  \right)\cap\{L_N\geq \kappa\}\mid \{1_{(a,b)}(\tau):0\leq \tau\leq t\} \right)\cr
    &= \sum_{\{\XX^{(k)}\}_{k=0}^N\in \mathcal X} \Pr\left(\cap_{k=0}^N \{\X(L_k)\in\XX^{(k)}\} \cap\{L_N\geq \kappa\} \mid \{1_{(a,b)}(\tau):0\leq \tau\leq t\}  \right)\cr
    &=\sum_{\{\XX^{(k)}\}_{k=0}^N\in \mathcal X} \Pr\left(\cap_{k=0}^N \{\X(L_k)\in\XX^{(k)}\} \cap\{L_N\geq \kappa\} \right)\cr
    &=\Pr\left(\cup_{\{\XX^{(k)}\}_{k=0}^N\in \mathcal X} \left( \cap_{k=0}^N\{\X(L_k)\in\XX^{(k)}\} \cap\{L_N\geq \kappa\} \right)\right)\cr
    &=\Pr(\{\tilde K=L_N\}\cap \{\X(\tilde K)\in\mathbb Y\}\cap\{L_N\geq \kappa\}),
\end{align}
which shows that $\{\tilde K=L_N\}\cap \{\X(\tilde K)\in\mathbb Y\}\cap\{L_N\geq \kappa\}$ is independent of $\{1_{(a,b)}(\tau):0\leq \tau\leq t\}$. Since $\{\tilde K\geq \kappa\}\cap\{\X(\tilde K)\in\mathbb Y\}=\cup_{N=0}^\infty\left(\{\tilde K=L_N\}\cap \{\X(\tilde K)\in\mathbb Y\}\cap\{L_N\geq \kappa\}\right)$, it follows that $\{\tilde K\geq \kappa\}\cap\{\X(\tilde K)\in\mathbb Y\}$ is independent of $\{1_{(a,b)}(\tau):0\leq \tau\leq t\}$. As a consequence of this observation, the fact that $\mathbb Y$ is an arbitrary $(a,b)$-agnostic superstate and the fact that $\kappa$ is an arbitrary non-negative number, we have that $(\XX(\tilde K),\tilde K)$ are independent of $\{1_{(a,b)}(\tau):0\leq \tau\leq t\}$, where $\XX(\tilde K)$ denotes the $(a,b)$-agnostic superstate of the chain at time $\tilde K$.
\end{remark}

In order to state the remaining lemmas, 
we need to introduce some additional notation. For two nodes $(a,b)\in [n]\times [n]$, we let $b\stackrel{t}{\rightsquigarrow}a$ denote the event that $b$ transmits pathogens to $a$ at time $t$. For a given time interval $[t, t+\Delta t)\subset [0,\infty)$, we let $\left\{b\stackrel{t, \Delta t}{\rightsquigarrow}a\right\}:=\cup_{\tau\in [t, t+\Delta t)} \{ b\stackrel{\tau }{\rightsquigarrow}a\}$. The complement of this event is denoted by $\left\{b\stackrel{t, \Delta t}{\centernot\rightsquigarrow}a\right\}$. For two given node sets $A, B\subset [n]$, we use $\{B\stackrel{t}{\rightsquigarrow} A\}$ to denote the event that some node(s) of $B$ infect(s) one or more nodes in $A$ at time $t$. 

We now provide a sequence of lemmas that we later use to prove Proposition~\ref{prop:bounds}.

\begin{lemma}\label{lem:obvious_two}
Suppose $a\in\A_i$, $b\in\A_j$, $y\in\{0,1\}$, and $t_1,t_2\in [0,\infty)$ such that $t_1<t_2$. Given that 
$b\in \I_j(t_1):=\I_j( \X(t_1) )$ and that $1_{(a,b)}(\tau) := 1_{(a,b)}(\X(\tau))=y$ for all $\tau\in [t_1,t_2)$, the conditional probability that $b$ neither recovers nor infects $a$ during the interval $[t_1,t_2)$ is $e^{-(B_{ij}\delta_{1y}+\gamma_j)(t_2 - t_1) }$, where $\delta_{ij}$ is the Kronecker delta.
\end{lemma}

\begin{proof}
Let $\mathbb X:=\{\x\in\mathbb S:b\in \I_j(\x),1_{(a,b)}(\x)=y\}$. Also, let $\Delta t>0$.  Since the rate of infection transmission from $b$ to $a$ at time $t_1$ is $B_{ij}1_{(a,b)}(\X(t_1))$, we have the following for all $\x\in\mathbb X$:
$$
    \Pr\left (b\stackrel{t_1,\Delta t}{ \rightsquigarrow} a \Bigm \vert \X(t_1) =\x\right) = B_{ij}\delta_{1y}\Delta t + o(\Delta t).
$$
On the other hand, denoting the event that $b$ recovers during $[t_1, t_1+\Delta t)$ by $D_b$, we have 
$$
    \Pr(D_b\mid \X(t)=\x) = {\gamma_j}\Delta t + o(\Delta t). 
$$
Similarly, if we let $F_{(a,b)}$ denote the event that the edge state $1_{(a,b)}$ flips (i.e., changes from $y$ to $1-y$) during $[t_1,t_1+\Delta t)$, we have $\Pr(F_{(a,b)}\mid \X(t)=\x)=\lambda \left( y\left(1-\frac{\rho_{ij} }{n} \right)+(1-y)\frac{\rho_{ij} }{n}\right) \Delta t+o(\Delta t)$.
As a result, we have
\begin{align*}
    &\Pr\left(\left\{b\stackrel{t_1,\Delta t}{\centernot\rightsquigarrow}a\right\}\cap \bar D_b\cap\bar F_{ (a,b) } \Bigm \vert \X(t_1) = \x\right) \cr
    &= 1 - \Pr\left(\left\{b\stackrel{t_1, \Delta t }{ \rightsquigarrow  } a \right\}\cup D_b\cup F_{(a,b)} \Bigm \vert \X(t_1)=\x \right)\cr
    &\stackrel{(a)}{=} 1 - \Pr\left(b\stackrel{t_1, \Delta t }{ \rightsquigarrow  } a\Bigm \vert \X(t_1)= \x\right) - \Pr(D_b\mid \X(t_1)=\x) - \Pr(F_{(a,b)}\mid \X(t_1)=\x \}) + o(\Delta t) \cr
    &=1 - (B_{ij}\delta_{1y}\Delta t + o(\Delta t)) - (\gamma_j\Delta t + o(\Delta t))-\left(\lambda \left( y\left(1-\frac{\rho_{ij} }{n} \right)+(1-y)\frac{\rho_{ij} }{n}\right) \Delta t + o(\Delta t)\right) + o(\Delta t) \cr
    &= 1 - \left(B_{ij}\delta_{1y}+\gamma_j+\lambda \left( y\left(1-\frac{\rho_{ij} }{n} \right)+(1-y)\frac{\rho_{ij} }{n}\right) \right)\Delta t + o(\Delta t),
\end{align*}
where (a) follows from Lemma~\ref{lem:obvious} and the Inclusion-Exclusion principle. Since this holds for all $\x\in\mathbb X$, the above implies that
\begin{align*}
    &\Pr\left(\left\{b\stackrel{t_1,\Delta t}{\centernot\rightsquigarrow}a\right\}\cap \bar D_b \cap \bar F_{(a,b)} \Bigm \vert \X(t_1) \in \mathbb X\right)\cr
    &=1 - \left(B_{ij}\delta_{1y}+\gamma_j+\lambda \left( y\left(1-\frac{\rho_{ij} }{n} \right)+(1-y)\frac{\rho_{ij} }{n}\right) \right)\Delta t + o(\Delta t).
\end{align*}
Now, consider any $\ell\in\N_0$. By replacing $t_1$ with $t_1+\ell\Delta t$ in the above relation, we obtain
\begin{align*}
    &\Pr\left(\left\{b\stackrel{t_1+\ell \Delta t,\Delta t}{\centernot\rightsquigarrow}a\right\}\cap \bar D_b^{(\ell)} \cap \bar F^{(\ell)}_{(a,b)} \Bigm \vert \X(t_1+\ell\Delta t) \in\mathbb X \right)\cr
    &=1 - \left(B_{ij}\delta_{1y}+\gamma_j+\lambda \left( y\left(1-\frac{\rho_{ij} }{n} \right)+(1-y)\frac{\rho_{ij} }{n}\right) \right)\Delta t + o(\Delta t),
\end{align*}
where $D_b^{(\ell)}$ is the event that $b$ recovers during $[t_1 + \ell \Delta t , t_1+(\ell+1)\Delta t)$ and $F_{(a,b)}^{(\ell)}$ is the event that $1_{(a,b)}$ flips during $[t_1+\ell\Delta t, t_1+(\ell+1)\Delta t)$. Therefore, on setting $\Delta t=\frac{t_2-t_1}{N}$ for an arbitrary $N\in \N$, it follows that
\begin{align}\label{eq:long_long}
    &\Pr\left(\{b\in\I(t)\}\cap\left\{b\stackrel{t_1, t_2-t_1 }{\centernot\rightsquigarrow}a\right\}\cap\{1_{(a,b)}(\tau)=y\,\forall\,\tau\in [t_1,t_2)\}  \Bigm \vert \X(t_1)=\x \right)\cr
    &=\prod_{\ell = 1}^{N-1} \Pr\left(\left\{b\stackrel{t_1+ \ell\Delta t, \Delta t }{\centernot\rightsquigarrow}a\right\}\cap \bar D_b^{(\ell)}\cap \bar F_{(a,b)}^{(\ell)}  \Bigm \vert b\stackrel{t_1,\ell\Delta t }{\centernot\rightsquigarrow}a, \bar D_b, \bar D_b^{(1)}, \ldots, \bar D_b^{(\ell-1)}, \bar F_{(a,b)},\ldots, \bar F_{(a,b)}^{(\ell-1)}, \X(t_1)=\x \right)\cr
    &\quad\quad\times\Pr\left( \left\{b\stackrel{t_1, \Delta t }{\centernot\rightsquigarrow}a\right\}\cap \bar D_b \cap \bar F_{(a,b)} \Bigm \vert \X(t_1)=\x \right)\cr
    &= \prod_{\ell = 1}^{N-1} \Pr\left(\left\{b\stackrel{t_1+ \ell\Delta t, \Delta t }{\centernot\rightsquigarrow}a\right\}\cap \bar D_b^{(\ell)}\cap \bar F_{(a,b)}^{(\ell)}  \Bigm \vert \X(t_1 +\ell \Delta t)\in\mathbb X,  b\stackrel{t_1,\ell\Delta t }{\centernot\rightsquigarrow}a, \{\bar D_b^{(\sigma)}\}_{\sigma=0}^{\ell-1} , \{\bar F_{(a,b)}^{(\sigma)}\}_{\sigma=0}^{\ell-1}, \X(t_1)=\x \right)\cr
    &\quad\quad\times\left(1 - \left(B_{ij}\delta_{1y} +\gamma_j+\lambda \left( y\left(1-\frac{\rho_{ij} }{n} \right)+(1-y)\frac{\rho_{ij} }{n}\right)\right)\Delta t + o(\Delta t) \right)\cr
    &\stackrel{(a)}{=}\prod_{\ell=1}^{N-1} \left(1 - \left(B_{ij}\delta_{1y} +\gamma_j+\lambda \left( y\left(1-\frac{\rho_{ij} }{n} \right)+(1-y)\frac{\rho_{ij} }{n}\right)\right)\Delta t + o(\Delta t) \right)\cr
    &\quad\times\left(1 - \left(B_{ij}\delta_{1y} +\gamma_j+\lambda \left( y\left(1-\frac{\rho_{ij} }{n} \right)+(1-y)\frac{\rho_{ij} }{n}\right)\right)\Delta t + o(\Delta t) \right)\cr
    &=\left(1 - \left(B_{ij}\delta_{1y} +\gamma_j+\lambda \left( y\left(1-\frac{\rho_{ij} }{n} \right)+(1-y)\frac{\rho_{ij} }{n}\right)\right)\left(\frac{t_2-t_1}{N} \right) \right)^N + o\left(\frac{1}{N}\right),
\end{align}
where (a) follows from the following observation: for any $\y\in\mathbb X$, Markovity implies that 
\begin{align*}
    &\Pr\left(\left\{b\stackrel{t_1+ \ell\Delta t, \Delta t }{\centernot\rightsquigarrow}a\right\}\cap \bar D_b^{(\ell)}\cap \bar F_{(a,b)}^{(\ell)}  \Bigm \vert \X(t_1+\ell\Delta t)=\y,b\stackrel{t_1, \ell\Delta t }{\centernot\rightsquigarrow}a, \{\bar D_b^{(\sigma)}\}_{\sigma=0}^{\ell-1}, \{\bar F_{(a,b)}^{(\sigma )} \}_{\sigma=0}^{\ell-1}, \X(t_1)=\x \right)\cr
    &=\Pr\left(\left\{b\stackrel{t_1+ \Delta t, \Delta t }{\centernot\rightsquigarrow}a\right\}\cap \bar D_b^{(1)}\cap \bar F_{(a,b)}^{(1)}  \Bigm \vert \X(t_1+\Delta t)=\y \right)\cr
    &=1 - \left(B_{ij}\delta_{1y} +\gamma_j+\lambda \left( y\left(1-\frac{\rho_{ij} }{n} \right)+(1-y)\frac{\rho_{ij} }{n}\right)\right)\Delta t + o(\Delta t),
\end{align*}
which further implies that
\begin{align*}
    &\Pr\left(\left\{b\stackrel{t_1+ \ell\Delta t, \Delta t }{\centernot\rightsquigarrow}a\right\}\cap \bar D_b^{(\ell)}\cap \bar F_{(a,b)}^{(\ell)}  \Bigm \vert \X(t_1+\ell\Delta t)\in\mathbb X,b\stackrel{t_1, \ell\Delta t }{\centernot\rightsquigarrow}a, \{\bar D_b^{(\sigma)}\}_{\sigma=0}^{\ell-1}, \{\bar F_{(a,b)}^{(\sigma )} \}_{\sigma=0}^{\ell-1}, \X(t_1)=\x \right)\cr
    &=1 - \left(B_{ij}\delta_{1y} +\gamma_j+\lambda \left( y\left(1-\frac{\rho_{ij} }{n} \right)+(1-y)\frac{\rho_{ij} }{n}\right)\right)\Delta t + o(\Delta t).
\end{align*}
Now, since~\eqref{eq:long_long} holds for all $N\in\N$, it follows that 
\begin{align}\label{eq:obtain_exp_1}
    &\Pr\left(\{b\in\I(t)\}\cap\left\{b\stackrel{t_1, t_2-t_1 }{\centernot\rightsquigarrow}a\right\}\cap\{1_{(a,b)}(\tau)=y\,\forall\,\tau\in [t_1,t_2)\}  \Bigm \vert \X(t_1)=\x \right)\cr
    &=\lim_{N\rightarrow\infty}\left( \left(1 - \left(B_{ij}\delta_{1y} +\gamma_j+\lambda y\left( 1 - \frac{\rho_{ij} }{n}  \right) + \lambda\left(1-y\right)\frac{\rho_{ij} }{n} \right)\left(\frac{t_2-t_1}{N} \right) \right)^N + o\left(\frac{1}{N}\right)\right)\cr
    &=e^{ -\left( B_{ij}\delta_{1y} +\gamma_j + \lambda \left( y\left(1-\frac{\rho_{ij} }{n} \right)+(1-y)\frac{\rho_{ij} }{n}\right) \right)\left(t_2-t_1 \right) }.
\end{align}
Similarly, we can show that  
\begin{align}\label{eq:obtain_exp_2}
    \Pr\left(1_{(a,b)}(\tau)=y\,\forall\,\tau\in [t_1,t_2)  \Bigm \vert \X(t_1)=\x \right)=e^{ - \lambda \left( y\left(1-\frac{\rho_{ij} }{n} \right)+(1-y)\frac{\rho_{ij} }{n}\right) \left(t_2-t_1 \right) }.
\end{align}
As a result of~\eqref{eq:obtain_exp_1} and~\eqref{eq:obtain_exp_2},
\begin{align*}
    &\Pr\left(\{b\in\I(t)\}\cap\left\{b\stackrel{t_1, t_2-t_1 }{\centernot\rightsquigarrow}a\right\} \Bigm \vert \{1_{(a,b)}(\tau)=y\,\forall\,\tau\in [t_1,t_2)\},  \X(t_1)=\x \right)\cr
    &=\frac{ \Pr\left(\{b\in\I(t)\}\cap\left\{b\stackrel{t_1, t_2-t_1 }{\centernot\rightsquigarrow}a\right\}\cap\{1_{(a,b)}(\tau)=y\,\forall\,\tau\in [t_1,t_2)\}  \Bigm \vert \X(t_1)=\x \right)  }{ \Pr\left(1_{(a,b)}(\tau)=y\,\forall\,\tau\in [t_1,t_2)  \Bigm \vert \X(t_1)=\x \right) }\cr
    &=e^{-(B_{ij}\delta_{1y}+\gamma_j)(t_2-t_1)}.
\end{align*} 
Since the above holds for all $\x\in\mathbb X$, it follows that 
$$
    \Pr\left(\{b\in\I(t)\}\cap\left\{b\stackrel{t_1, t_2-t_1 }{\centernot\rightsquigarrow}a\right\} \Bigm \vert \{1_{(a,b)}(\tau)=y\,\forall\,\tau\in [t_1,t_2)\},  \X(t_1)\in\mathbb X \right) = e^{- (B_{ij}\delta_{1y}+\gamma_j) (t_2-t_1)},
$$
which proves the lemma.
\end{proof}

\begin{lemma}\label{lem:T_on}
Let $T\on:=\int_{t-K}^{t-T} 1_{(a,b)}(\sigma)d\sigma$ denote the total duration of time for which the edge $(a,b)$ exists in the network during $[t-K,t-T]$. Then, for all $\kappa,\tau\in [0,t]$ and all $\tau\on\in\left[0,(\kappa-\tau)_+\right]$, we have
$$
    \Pr\left(b\stackrel{0,t }{\centernot\rightsquigarrow}a\,\bigg\lvert\, (K,T,T\on)=(\kappa,\tau,\tau\on), b\in\I(t^-), (a,b)\notin E(t)\right) = e^{-B_{ij}\tau\on},
$$
where we define $\I(\sigma ^-):=\cup_{\varepsilon>0}\cap_{\tau'\in[\sigma-\varepsilon, \sigma)} \I(\tau')$ for all $\sigma\geq 0$. In other words, $c\in\mathcal I(\sigma^-)$ iff there exists an $\varepsilon>0$ such that $c\in\I(\tau')$ for all $\tau'\in[\sigma-\varepsilon,\sigma)$.
\end{lemma}

\begin{proof}
We first show that $\left\{b\stackrel{t-\kappa,\kappa-\tau}{\centernot\rightsquigarrow}a\right\}$ is conditionally independent of $\{(a,b)\notin E(t)\}$  given $(K,T,T\on)=(\kappa,\tau,\tau\on)$ and $b\in\I(t-\tau)$:
\begin{align}\label{eq:cond_indep}
    &\Pr\left((a,b)\in E(t)\,\bigg\lvert\, (K,T,T\on)=(\kappa,\tau,\tau\on), b\in\I(t-\tau),b\stackrel{t-\kappa,\kappa-\tau}{\centernot\rightsquigarrow}a \right)\cr
    &\stackrel{(a)}{=}\Pr\left((a,b)\in E(t-\tau)\,\bigg\lvert\, (K,T,T\on)=(\kappa,\tau,\tau\on), b\in\I(t-\tau),b\stackrel{t-\kappa,\kappa-\tau}{\centernot\rightsquigarrow}a \right)\cr
    &\stackrel{(b)}{=}\frac{\rho_{ij} }{n}\cr
    &=\Pr\left((a,b)\in E(t-\tau)\mid T=\tau\right)\cr
    &\stackrel{(c)}{=}\Pr((a,b)\in E(t)\mid T=\tau ),
\end{align}
where $(a)$ and $(c)$ hold because $1_{(a,b)}$ is not updated during the interval $[t-\tau, t)$, and (b) follows from the modeling assumption that the probability of the edge $(a,b)$ existing in the network following an edge state update  is $\frac{\rho_{ij}
}{n}$ (independent of the past states $\{\mathbf X(\tau'):0\leq \tau'< t-\tau\}$), the fact that $\{b\in\I(t-\tau)\}=\{b\in\I((t-\tau)^- )\}$ almost surely, and from the observation that $t-\tau$ is an update time for $1_{(a,b)}$ given $T=\tau$.

In view of~\eqref{eq:cond_indep}, the definitions of $K$, $T$, and $T\on$ imply that 
\begin{align}\label{eq:start_now}
    &\Pr\left(b\stackrel{0,t }{\centernot\rightsquigarrow}a\,\bigg\lvert\, (K,T,T\on)=(\kappa,\tau,\tau\on), b\in\I(t^-), (a,b)\notin E(t)\right)\cr
    &=\Pr\left(b\stackrel{t-\kappa,\kappa-\tau}{\centernot\rightsquigarrow}a\,\bigg\lvert\, (K,T,T\on)=(\kappa,\tau,\tau\on), b\in\I(t^-), (a,b)\notin E(t)\right)\cr
    &=\Pr\left(b\stackrel{t-\kappa,\kappa-\tau}{\centernot\rightsquigarrow}a\,\bigg\lvert\, (K,T,T\on)=(\kappa,\tau,\tau\on), b\in\I(t^-) \right)\cr
    &=\Pr\left(b\stackrel{t-\kappa,\kappa-\tau}{\centernot\rightsquigarrow}a\,\bigg\lvert\, (K,T,T\on)=(\kappa,\tau,\tau\on), b\in\I(t-\tau), b\notin\cup_{\tau'\in (t-\tau,t)} \mathcal R(\tau') \right)\cr
    &=\Pr\left(b\stackrel{t-\kappa,\kappa-\tau}{\centernot\rightsquigarrow}a\,\bigg\lvert\, (K,T,T\on)=(\kappa,\tau,\tau\on), b\in\I(t-\tau)\right),
\end{align}
where the last step holds because $\Q(\x,\x_{\recover b})=\gamma_j$ for all $\x\in\mathbb S$ satisfying $x_b=1$, which implies that, given $b\in\I(t-\tau)$ and any other conditioning event,  node $b$ recovers during $(t-\tau,t)$ at a constant rate of $\gamma_j$ independently of all past edge states and past disease states (and therefore independently of past transmissions as well). Hence, $\left\{b\stackrel{t-\kappa,\kappa-\tau}{\centernot\rightsquigarrow}a\right\}$ and $\{b\notin \cup_{\tau'\in(t-\tau,t)}\mathcal R(\tau') \}$ are conditionally independent given $(K,T,T\on)=(\kappa,\tau,\tau\on)$ and $b\in\I(t-\tau)$.

We now evaluate the right-hand side of~\eqref{eq:start_now} as follows. Let $C$ denote the (random) number of times $1_{(a,b)}$ flips (changes) during $[t-K, t-T]$, and let the times of these changes be $T_1<\cdots<T_C$. We assume that $C$ is even (as the case of $C$ being odd is handled similarly) and that $1_{(a,b)}(\tau')=0$ for $\tau'\in[t-K,T_1]$ (the case $1_{(a,b)}(\tau')=1$ for $\tau'\in[t-K,T_1]$ is handled similarly). Then, for a given $c\in\N$ and a collection of times $t_1, \ldots, t_c, \tau\on$, we have $\{C=c,T_1=t_1,\ldots, T_c=t_c\}\subset \{T\on=\tau\on\}$ iff $\sum_{k=1}^{c/2} (t_{2k} - t_{2k-1})=\tau\on$. Suppose this condition holds. Then, observe that
\begin{align}\label{eq:infect_and_recover_prob}
    &\Pr\left(b\stackrel{t-\kappa,\kappa-\tau}{\centernot\rightsquigarrow}a , b\in\I(t-\tau)\, \bigg\lvert\, (K,T,T\on)=(\kappa,\tau,\tau\on), C=c, (T_1,\ldots, T_c)=(t_1,\ldots, t_c)\right)\cr
    &=\Pr\left(b\stackrel{t-\kappa,\kappa-\tau}{\centernot\rightsquigarrow}a, b\in\I(t-\tau)\, \bigg\lvert\, (K,T,T\on)=(\kappa,\tau,\tau\on), 1_{(a,b)}(\tau')=1\text{ iff }\tau'\in \cup_{k=1}^{c/2} [t_{2k-1},t_{2k}] \right)\cr
    &\stackrel{(a)}{=}\prod_{k=1}^{c/2}\Pr\left ( b\stackrel{t_{2k-1}, t_{2k}-t_{2k-1} }{\centernot\rightsquigarrow}a, b\in \I(t_{2k})\,\bigg\lvert\, 1_{(a,b)}(\tau')=1\,\forall\,\tau'\in[t_{2k-1},t_{2k}], b\in \I(t_{2k-1} ) \right) \cr
    &\quad\times\prod_{k=1}^{c/2+1}\Pr\left ( b\stackrel{t_{2k-2}, t_{2k-1}-t_{2k-2} }{\centernot\rightsquigarrow}a, b\in \I(t_{2k-1})\,\bigg\lvert\, 1_{(a,b)}(\tau')=0\,\forall\,\tau'\in[t_{2k-2},t_{2k-1}], b\in \I(t_{2k-1} ) \right) \cr
    &\stackrel{(b)}{=}\prod_{k=1}^{c/2}e^{-(B_{ij}+\gamma_j)(t_{2k}-t_{2k-1})}\times\prod_{k=1}^{c/2+1} e^{-\gamma_j (t_{2k-1}-t_{2k-2})}\cr
    &=e^{-B_{ij}\sum_{k=1}^{c/2}(t_{2k}-t_{2k-1})}\cdot e^{-\gamma_j\sum_{k=1}^{c+1}(t_k-t_{k-1})} \cr
    &= e^{-B_{ij}\tau\on} e^{-\gamma_j(\kappa-\tau)},
\end{align}
where $(b)$ follows from Lemma~\ref{lem:obvious_two}, and $(a)$ follows from the following fact: since the definition of our epidemic model implies that the rate of pathogen transmission from $b$ to $a$ at any time instant $t'$ depends only on $1_{(a,b)}(t')$ and the disease state of $b$ at time $t'$, transmission events corresponding to disjoint time intervals are conditionally independent if we are given $1_{(a,b)}$ and the disease state of $b$ as functions of time.

On the other hand, we have
\begin{align}\label{eq:recovery_probability}
    &\Pr(b\in \I(t-\tau)\mid (K,T,T\on)=(\kappa,\tau,\tau\on), C=c, (T_1, \ldots, T_c)=(t_1,\ldots, t_c))\cr
    &=\Pr(b\notin \R(t-\tau)\mid (K,T,T\on)=(\kappa,\tau,\tau\on), C=c, (T_1, \ldots, T_c)=(t_1,\ldots, t_c))\cr
    &=e^{-\gamma_j((t-\tau) -(t-\kappa)) }\cr
    &=e^{-\gamma_j(\kappa-\tau)},
\end{align}
where the second equality holds because our model assumes that the rate of recovery of an infected node is time-invariant and independent of all the edge states and the disease states of other nodes (precisely, $\Q(\x,\x_{\recover b})=\gamma_j$ for all $\x\in\mathbb S$ such that $b\in \I(\x)$).

As a result of~\eqref{eq:infect_and_recover_prob} and~\eqref{eq:recovery_probability}, we have
\begin{align*}
    &\Pr\left(b\stackrel{t-\kappa, \kappa-\tau}{\centernot\rightsquigarrow}a\,\bigg\lvert\, (K,T,T\on)=(\kappa,\tau,\tau\on), b\in\I(t-\tau), (C, T_1,\ldots, T_C)=(c,t_1,\ldots, t_c) \right)\cr
    &=\frac{ \Pr\left(b\stackrel{t-\kappa, \kappa-\tau}{\centernot\rightsquigarrow}a,b\in\I(t-\tau)\,\bigg\lvert\, (K,T,T\on)=(\kappa,\tau,\tau\on),(C, T_1,\ldots, T_C)=(c,t_1,\ldots, t_c)\right)  }{ \Pr\left(b\in\I(t-\tau)\,\bigg\lvert\, (K,T,T\on)=(\kappa,\tau,\tau\on),(C, T_1,\ldots, T_C)=(c,t_1,\ldots, t_c)\right)  }\cr
    &= e^{-B_{ij}\tau\on}.
\end{align*}
Since $(c,t_1,\ldots, t_c)$ was an arbitrary tuple satisfying $\{(C,T_1,\ldots, T_C)=(c,t_1,\ldots, t_c)\}\subset\{T\on=\tau\on\}$, it follows that $$
    \Pr\left(b\stackrel{t-\kappa, \kappa-\tau}{\centernot\rightsquigarrow}a\,\bigg\lvert\, (K,T,T\on)=(\kappa,\tau,\tau\on), b\in\I(t-\tau)\right) = e^{-B_{ij}\tau\on}.
$$
Invoking~\eqref{eq:start_now} now completes the proof.
\end{proof}

Observe that in the above proof, given that $K=\kappa$ and that $(a,b)\notin E(t)$, $(C,T_1,\ldots, T_C)$  uniquely determines $\{1_{(a,b)}(\tau):t-K\leq \tau\leq t\}$. Therefore, as an implication of the above proof, we have
$$
    \Pr\left(b\stackrel{t-K,t}{\centernot\rightsquigarrow}a\mid K=\kappa, \{1_{(a,b)}(\tau):t-K\leq \tau\leq t\}, b\in\I(t^-),(a,b)\notin E(t)\right)=e^{-B_{ij}T\on}.
$$
The dependence on the random variable $T\on$ holds because $T\on$ is a function of $\{1_{(a,b)}(\tau):t-K\leq \tau\leq t\}$. By invoking Markovity, this result can be extended to 
$$
    \Pr\left(b\stackrel{t-K,t}{\centernot\rightsquigarrow}a\mid K=\kappa, \{1_{(a,b)}(\tau):0\leq \tau\leq t\}, b\in\I(t^-),(a,b)\notin E(t)\right)=e^{-B_{ij}T\on},
$$
which is equivalent to the following lemma.

\begin{lemma}\label{lem:T_on_general}
Let $T\on:=\int_{t-K}^{t-T} 1_{(a,b)}(\sigma)d\sigma$ denote the total duration of time for which the edge $(a,b)$ exists in the network during $[t-K,t-T]$. Then, for all $\kappa\in [0,t]$, we have
$$
    \Pr\left(b\stackrel{0,t }{\centernot\rightsquigarrow}a\,\bigg\lvert\, K=\kappa, \{1_{(a,b)}(\tau):0\leq \tau\leq t\}, b\in\I(t^-), (a,b)\notin E(t)\right) = e^{-B_{ij}T\on}.
$$
\end{lemma}

\begin{lemma}\label{lem:equation_only} Recall from Lemma~\ref{lem:T_on_general} that $T\on=\int_{t-K}^{t-T} 1_{(a,b)}(\sigma)d\sigma$. Then for all $\kappa,\tau\in [0,t]$, we have
\begin{align*}
    &\Pr\left((\S(t),\I(t))=(\S_0,\I_0)\,\Big \lvert\, (K,T,T\on)=(\kappa,\tau,\tau\on), b\stackrel{0,t  }{ \centernot \rightsquigarrow }a, (a,b)\notin E(t), b\in \I(t^-) \right)\cr
    &= \Pr\left((\S(t),\I(t))=(\S_0,\I_0)\,\Big \lvert\, K=\kappa, b\stackrel{0,t  }{ \centernot \rightsquigarrow }a, (a,b)\notin E(t), b\in\I(t^-) \right).
\end{align*}
\end{lemma}

\begin{proof}
We first examine the following conditional probability for an arbitrary $(a,b)$-agnostic superstate $\mathbb Y$:
$$
    \Pr\left((\S(t),\I(t))=(\S_0,\I_0), \X(\tilde K)\in\mathbb Y, K= \kappa, b\stackrel{0,t  }{ \centernot \rightsquigarrow }a, b\in \I(t^-) \mid \{1_{(a,b)}(\tau):0\leq \tau\leq t\}, (a,b)\notin E(t) \right).
$$
To begin, note that the proof of Lemma~\ref{lem:simple_but_difficult}, Remark~\ref{rem:tilde_K}, and the fact that $K$ is a function of $\tilde K$ together imply that
\begin{align*}
    &f_{K\mid \{1_{(a,b)}(\tau):0\leq \tau\leq t\}, (a,b)\notin E(t)}( \kappa ) = f_K(\kappa)
\end{align*}
and that 
\begin{align}\label{eq:to_be_used}
    \Pr(\X(\tilde K)\in\mathbb Y\mid K=\kappa, \{1_{(a,b)}(\tau):0\leq \tau\leq t\},(a,b)\notin E(t) )=\Pr(\X(\tilde K)\in\mathbb Y\mid K=\kappa).
\end{align}
Next, for the event $\{b\in\I(t^-)\}$, we have
\begin{align}\label{eq:to_be_used_2}
    &\Pr\left(b\in\I(t^-)\mid  K= \kappa,\X(\tilde K)\in\mathbb Y,\{1_{(a,b)}(\tau):0\leq \tau\leq t\},(a,b)\notin E(t)\right) \cr
    &\stackrel{(a)}= e^{-\gamma_j(t-\kappa)} \cr
    &\stackrel{(b)}= \Pr\left(b\in\I(t^-)\mid   K=\kappa\right) 
\end{align}
where $(a)$ and $(b)$ follow from our modelling assumption that $\Q(\x,\x_{\recover b})=\gamma_j$ for all $\x$ satisfying $x_b=1$, which means that the recovery time of $b$ depends only on the time of infection of $b$ and is conditionally independent of all other disease states and all the edge states.
Similarly, we have
\begin{align}\label{eq:to_be_used_3}
    &\Pr\left(b\stackrel{0,t}{\centernot\rightsquigarrow }a\mid \X(\tilde K)\in\mathbb Y, K= \kappa, \{1_{(a,b)}(\tau):0\leq \tau\leq t\},(a,b)\notin E(t),b\in\I(t^-)\right)\cr
    &\stackrel{(a)}=\Pr\left(b\stackrel{0,t}{\centernot\rightsquigarrow}a\mid   K= \kappa, \{1_{(a,b)}(\tau):0\leq \tau\leq t\},(a,b)\notin E(t),b\in\I(t^-)\right) \cr
    &\stackrel{(b)}= e^{-B_{ij}T\on } \cr
\end{align}
where $(a)$ follows from our modelling assumptions, which imply that the rate of infection transmission along an edge depends only on the edge state of the transmitting edge and the disease state of the transmitting node and is conditionally independent of other disease states and edge states (which are captured by the $(a,b)$-agnostic superstate of the chain) and $(b)$ follows from Lemma~\ref{lem:T_on_general}. Note that $T\on$ is a function of $T$ and hence also of $\{1_{(a,b)}(\tau):0\leq \tau\leq t\}$.

It remains for us to analyze 
$$
    \Pr\left((\S(t),\I(t))=(\S_0,\I_0) \mid \X(\tilde K)\in\mathbb Y,K= \kappa, b\stackrel{0,t  }{ \centernot \rightsquigarrow }a, b\in \I(t^-) , \{1_{(a,b)}(\tau):0\leq \tau\leq t\}, (a,b)\notin E(t) \right).
$$
To do so, we first let $L_N$ denote the time of the first $(a,b)$-agnostic jump to occur after $b $ gets infected, i.e., $N:=\inf\{\ell\in\N:L_\ell\geq \tilde K\}$, and we note the following: given the conditioning events and variables above (including the event that $b$ does not infect $a$ during $[0,t]$), the total conditional rate at which $a$ receives pathogens at any time $\tau\leq L_N$ is
$$
\sum_{q=1}^m\sum_{d\in\I_q(\X(\tilde K )\setminus\{b\}}B_{iq}1_{(a,d)}(\X(\tilde K) ),
$$
which is determined uniquely by $\mathbb Y$, the $(a,b)$-agnostic superstate of the chain at time $\tilde K$. Therefore, this rate is conditionally independent of $1_{(a,b)}(\tau)$ for any $\tau$. Similarly, for all age groups $\ell\in[m]$, given the conditioning events and variables above, the conditional rate at which a node $d\in\I_\ell(\X(\tilde K))$ recovers, which equals $\gamma_\ell$, and the total conditional rate at which a node $c\in\A_\ell\setminus\{a\}$ receives pathogens, which equals $\sum_{q=1}^m\sum_{d\in\I_q(\X(\tilde K) )} B_{\ell q}1_{(c,d)}(\X(\tilde K))$, are both conditionally independent of $1_{(a,b)}(\tau)$ given that $\X(\tilde K)\in\mathbb Y$. Therefore, by using arguments similar to those made in the proof of Lemma~\ref{lem:simple_but_difficult}, we can show that $(\S(L_N),\I(L_N))$ is conditionally independent of $\{1_{(a,b)}(\tau):0\leq \tau\leq t\}$ given the rest of the conditioning events and variables. Moreover, by repeating the above for subsequent $(a,b)$-agnostic jumps, we can generalize this conditional independence assertion to $(\S(t),\I(t))$, which means that
\begin{align}\label{eq:hopefully_last}
    \Pr\Big((\S(t&),\I( t))=(\S_0,\I_0) \mid \X(\tilde K)\in\mathbb Y,K= \kappa, b\stackrel{0,t  }{ \centernot \rightsquigarrow }a, b\in \I(t^-) , \{1_{(a,b)}(\tau):0\leq \tau\leq t\}, (a,b)\notin E(t) \Big)\nonumber\\
    &=\Pr\left((\S(t),\I(t))=(\S_0,\I_0) \mid \X(\tilde K)\in\mathbb Y,K= \kappa, b\stackrel{0,t  }{ \centernot \rightsquigarrow }a, b\in \I(t^-) , (a,b)\notin E(t) \right).
\end{align}

Combining~\eqref{eq:to_be_used},~\eqref{eq:to_be_used_2},~\eqref{eq:to_be_used_3} and~\eqref{eq:hopefully_last} now yields
\begin{align*}
    &\Pr\left((\S(t),\I(t))=(\S_0,\I_0), \X(\tilde K)\in\mathbb Y, b\stackrel{0,t  }{ \centernot \rightsquigarrow }a, b\in \I(t^-) \mid K= \kappa, \{1_{(a,b)}(\tau):0\leq \tau\leq t\}, (a,b)\notin E(t) \right)\nonumber\\
    &=\Pr\left((\S(t),\I(t))=(\S_0,\I_0) \mid \X(\tilde K)\in\mathbb Y,K= \kappa, b\stackrel{0,t  }{ \centernot \rightsquigarrow }a, b\in \I(t^-) , (a,b)\notin E(t) \right)\cr
    &\quad\times e^{-B_{ij}T\on} \times \Pr(b\in \I(t^-)\mid K=\kappa)\times\Pr(\X(\tilde K)\in\mathbb Y\mid K=\kappa)
\end{align*}
Summing both the sides of the above equation over the space of all $(a,b)$-agnostic superstates $\mathbb Y$ gives
\begin{align}\label{eq:numerator}
    &\Pr\left((\S(t),\I(t))=(\S_0,\I_0), b\stackrel{0,t  }{ \centernot \rightsquigarrow }a, b\in \I(t^-) \mid K= \kappa, \{1_{(a,b)}(\tau):0\leq \tau\leq t\}, (a,b)\notin E(t) \right)\nonumber\\
    &= e^{-B_{ij}T\on} \times \Pr(b\in \I(t^-)\mid K=\kappa)\cr
    &\quad\times \sum_{\mathbb Y}\Bigg( \Pr\left((\S(t),\I(t))=(\S_0,\I_0) \mid \X(\tilde K)\in\mathbb Y, K= \kappa, b\stackrel{0,t  }{ \centernot \rightsquigarrow }a, b\in \I(t^-) , (a,b)\notin E(t) \right)\cr
    &\quad\quad\quad\quad\quad\cdot\Pr(\X(\tilde K)\in\mathbb Y\mid K=\kappa)\Bigg).
\end{align}

Here, we recall from our earlier arguments that
\begin{align*}
    &e^{-B_{ij}T\on} \times \Pr(b\in \I(t^-)\mid K=\kappa)\cr
    &=\Pr\left(b\stackrel{0,t}{\centernot\rightsquigarrow}a\mid   K= \kappa, \{1_{(a,b)}(\tau):0\leq \tau\leq t\},(a,b)\notin E(t),b\in\I(t^-)\right)\cr
    &\quad\times \Pr\left(b\in\I(t^-)\mid   K= \kappa, \{1_{(a,b)}(\tau):0\leq \tau\leq t\},(a,b)\notin E(t)\right)\cr
    &=\Pr\left(b\stackrel{0,t  }{ \centernot \rightsquigarrow }a, b\in \I(t^-) \mid \{1_{(a,b)}(\tau):0\leq \tau\leq t\},(a,b)\notin E(t)\right).
\end{align*}
In light of~\eqref{eq:numerator}, this means that 
\begin{align*}
    &\Pr\left((\S(t),\I(t))=(\S_0,\I_0), b\stackrel{0,t  }{ \centernot \rightsquigarrow }a, b\in \I(t^-) \mid K= \kappa, \{1_{(a,b)}(\tau):0\leq \tau\leq t\}, (a,b)\notin E(t) \right)\nonumber\\
    &= \Pr\left(b\stackrel{0,t  }{ \centernot \rightsquigarrow }a, b\in \I(t^-) \mid \{1_{(a,b)}(\tau):0\leq \tau\leq t\},(a,b)\notin E(t)\right)\cr
    &\quad\times \sum_{\mathbb Y}\Bigg( \Pr\left((\S(t),\I(t))=(\S_0,\I_0) \mid \X(\tilde K)\in\mathbb Y, K= \kappa, b\stackrel{0,t  }{ \centernot \rightsquigarrow }a, b\in \I(t^-) , (a,b)\notin E(t) \right)\cr
    &\quad\quad\quad\quad\quad\cdot\Pr(\X(\tilde K)\in\mathbb Y\mid K=\kappa)\Bigg).
\end{align*}
Dividing both the sides of this equation by $\Pr\left(b\stackrel{0,t  }{ \centernot \rightsquigarrow }a, b\in \I(t^-) \mid \{1_{(a,b)}(\tau):0\leq \tau\leq t\},(a,b)\notin E(t)\right)$ gives
\begin{align*}
    &\Pr\left((\S(t),\I(t))=(\S_0,\I_0)\mid K= \kappa, b\stackrel{0,t  }{ \centernot \rightsquigarrow }a, b\in \I(t^-), \{1_{(a,b)}(\tau):0\leq \tau\leq t\}, (a,b)\notin E(t) \right)\cr
    &= \sum_{\mathbb Y}\Bigg( \Pr\left((\S(t),\I(t))=(\S_0,\I_0) \mid \X(\tilde K)\in\mathbb Y, K= \kappa, b\stackrel{0,t  }{ \centernot \rightsquigarrow }a, b\in \I(t^-) , (a,b)\notin E(t) \right)\cr
    &\quad\quad\quad\quad\cdot\Pr(\X(\tilde K)\in\mathbb Y\mid K=\kappa)\Bigg)\cr
    &=\Pr\left((\S(t),\I(t))=(\S_0,\I_0)\mid K= \kappa, b\stackrel{0,t  }{ \centernot \rightsquigarrow }a, b\in \I(t^-),  (a,b)\notin E(t) \right),
\end{align*}
where the last step holds because the summation is independent of $\{1_{(a,b)}(\tau):0\leq \tau\leq t\}$ given that $(a,b)\notin E(t)$. We have thus shown the following: given $K= \kappa, b\stackrel{0,t  }{ \centernot \rightsquigarrow }a$, and $ b\in \I(t^-)$, the event $\{(\S(t),\I(t))= (\S_0, \I_0  )\}$ is conditionally independent of $\{1_{(a,b)}(\tau):0\leq \tau\leq t\}$. Since $T$ and $T\on$ are functions of $\{1_{(a,b)}(\tau):0\leq \tau\leq t\}$, the assertion of the lemma follows.
\end{proof}

\subsection*{Proof of Proposition~\ref{prop:bounds}}

Before we prove Proposition~\ref{prop:bounds}, we recall that for any transition sequence $F=\{\x^{(0)}\stackrel{t_1}{\rightarrow}\x^{(1)}\stackrel{t_2}{\rightarrow}\cdots \stackrel{t_r}{\rightarrow} \x^{(r)}\stackrel{t}{\rightarrow}\x^{(r)}\}$ on a time interval $[0,t]$, the index $\Lambda_{(a,b)}(F)$
indexes the transition in which $(a.b)$ is updated for the last time during $[0,t]$ given that $F$ occurs. We now define another similar index below:
$$
    \Gamma{(a,b)}(F)=
    \begin{cases}
        \min\left\{\ell\in [r]: \x^{(\ell)}=\x^{(\ell-1)}_{\infect b} \right\} \quad &\text{if } \left\{\ell\in [r]: \x^{(\ell)}=\x^{(\ell-1)}_{\infect b} \right\}\neq\emptyset\\
        0 \quad &\text{otherwise.}
    \end{cases}
$$
Observe that $\Gamma_b(F)$ indexes the transition in which $b$ gets infected given that $F$ occurs.
\begin{proof}
Consider any realization $(\S_0, \I_0)$ of $(\S(t),\I(t))$, and let $\F$ be the set of all the transition sequences for $[0,t]$ that result in the occurrence of $\{(\S(t),\I(t))=(\S_0,\I_0)\}$, so that $\{(\S(t),\I(t))=(\S_0,\I_0)\} = \cup_{F\in\F}F$. 

Consider now any pair of nodes $(a,b)\in\A_i\cap \S_0\times \A_j\cap\I_0$ (so that we have $a\in\S_i(t)$ and $b\in\I_j(t)$ in the event that $(\S(t),\I(t)) = (\S_0,\I_0)$), and note that for any transition sequence $F\in\F$, we have $F\abcomp\in \F$, because both $F\abcomp$ and $F$ involve the same node recoveries and disease transmissions (all of which occur along edges other than $(a,b)$). Therefore, $F\abagn\subset \F$ for each $F\in\F$, and it follows that $\{(\S(t),\I(t))=(\S_0,\I_0)\}=\cup_{F\in\F}F\abagn$. 

Hence, we can derive bounds on $\chi_{ij}(t)$ (defined to be $\Pr((a,b)\in E(t)\mid \S(t),\I(t))$) by bounding $\Pr((a,b)\in E(t)\mid (\S(t),\I(t)) = (\S_0,\I_0)) = \Pr((a,b)\in E(t)\mid \cup_{F\in\F}F\abagn)$. To this end, we pick $F\in \F$ and $\delta>0$, and apply Bayes' rule to $\Pr((a,b)\in E(t)\mid F\abagn^\delta)$ as follows:
\begin{align}\label{eq:bayes_with_question}
    \Pr((a,b)\in E(t)\mid F\abagn^\delta)&=\left(\frac{ \Pr(F\abagn^\delta)   }{ \Pr(F\abagn^\delta \mid (a,b)\in E(t))\cdot\Pr((a,b)\in E(t)) }\right)^{-1}\cr
    &=\left(1 + \frac{\Pr(F\abagn^\delta\mid (a,b)\notin E(t) ) } {   \Pr(F\abagn^\delta \mid (a,b)\in E(t))}\cdot \frac{\Pr((a,b)\notin E(t)) }{\Pr((a,b)\in E(t))} \right)^{-1}\cr
    &\stackrel{(a)}{=}\left(1+ \frac{\Pr(F\abagn^\delta\mid (a,b)\notin E(t) ) } {   \Pr(F\abagn^\delta \mid (a,b)\in E(t))} \cdot \frac{ 1 - \rho_{ij}/n  }{ \rho_{ij}/n } \right)^{-1},
\end{align}
where $(a)$ holds because $\Pr((a,b)\in E(t))=\frac{\rho_{ij} }{n}$, which is the probability that the edge $(a,b)$ exists in the network after the last of the updates of $1_{(a,b)}$ to occur during $[0,t]$.  

We now estimate $\frac{\Pr(F\abagn^\delta\mid (a,b)\notin E(t) ) } {   \Pr(F\abagn^\delta \mid (a,b)\in E(t))}$. Note that if $\delta$ is small enough, either $F^\delta \subset\{(a,b)\in E(t)\}$ or $F^\delta\subset\{(a,b)\notin E(t) \}$. Assume w.l.o.g. that $F^\delta\subset \{(a,b)\notin E(t)\}$ (equivalently, $F^\delta\abcomp\subset \{(a,b)\in E(t)\}$), and observe that
\begin{align}\label{eq:delta_question}
    \frac{\Pr(F\abagn^\delta\mid (a,b)\notin E(t) ) } {\Pr(F\abagn^\delta \mid (a,b)\in E(t))}&=\frac{\Pr(F\abagn^\delta\cap\{(a,b)\notin E(t)\})}{\Pr(F\abagn^\delta\cap\{(a,b)\in E(t)\})}\cdot\frac{\Pr(a,b)\in E(t) }{\Pr(a,b)\notin E(t)}\cr
    &=\frac{\Pr(F^\delta)}{\Pr(F^\delta\abcomp)}\left(\frac{ \rho_{ij}/{n} }{ 1- \rho_{ij}/{n}}\right).
\end{align}

Thus, the next step is to evaluate $\frac{\Pr(F^\delta)}{\Pr(F\abcomp^\delta)}$. To do so, suppose $F=\{\x^{(0)}\stackrel{t_1}{\rightarrow}\x^{(1)}\stackrel{t_2}{\rightarrow}\cdots\stackrel{t_r}{\rightarrow}\x^{(r)}\stackrel{t_{r+1} }{\rightarrow}
x^{(r)}\}$  with $t_{r+1}:=t$, $\Lambda_{(a,b)}(F)=\zeta\in \{0,1,\ldots,r\}$, $\Gamma_b(F)=\xi\in\{0,1,\ldots, r\}$, $t_{\Lambda_{(a,b)}(F)}=t_\zeta=t-\tau$ for some $\tau\in [0,t]$, and $t_{\Gamma_b(F)}=t_\xi=t-\kappa$ for some $\kappa\in[0,\tau]$. Then the assumption $F^\delta\subset\{(a,b)\notin E(t) \}$ implies that $\x^{(\zeta)}=\x^{(\zeta-1)}_{\nonexistence(a,b)}$ and hence also that $\x^{(\zeta)}\abcomp=\x^{(\zeta-1)}_{\existence(a,b)}$. As a result, 
\begin{align*}
    F\abcomp &=\Big\{\mathbf{x}^{(0)} \stackrel{t_{1}}{\rightarrow} \cdots \stackrel{t_{\zeta-1}}{\rightarrow} \mathbf{x}^{(\zeta-1)} \stackrel{t-\tau}{\rightarrow} \mathbf{x}_{\uparrow(a, b)}^{(\zeta-1)} \stackrel{t_{\zeta+1}}{\rightarrow} \mathbf{x}_{(\overline{a, b})}^{(\zeta+1)}\stackrel{t_{\zeta+2}}{\rightarrow} \cdots \stackrel{t_{r}}{\rightarrow} \mathbf{x}_{(\overline{a,b})}^{(r)} \stackrel{t}{\rightarrow} \mathbf{x}_{(\overline{a, b})}^{(r)}\Big\}.
\end{align*}
It now follows from Lemma~\ref{lem:delta_approx} that 
\begin{align}\label{eq:ratio_intermediate}
    \frac{ \Pr(F^{\delta}) }{\Pr(F^\delta\abcomp)}=\frac{ e^{-q_r(t-t_r)}\prod_{\ell=\zeta }^r q_{\ell-1, \ell} (e^{-q_{\ell-1}(t_\ell-t_{\ell-1})}\delta + o(\delta))  }{  e^{-\bar q_r(t-t_r)}\prod_{\ell=\zeta }^r \bar q_{\ell-1, \ell} (e^{-\bar q_{\ell-1}(t_\ell-t_{\ell-1})}\delta + o(\delta))  },
\end{align}
where $q_{\ell-1,\ell}:=\Q(\x^{(\ell-1)},\x^{(\ell)})$ and $\bar q_{\ell-1,\ell}:=\Q(\x^{(\ell-1)}\abcomp,\x^{(\ell)}\abcomp)$ for all $\ell\in\{\zeta+1,\ldots, r\}$,   $q_{\ell}:=-\Q(\x^{(\ell)},\x^{(\ell)}) $ and $\bar q_{\ell}:=-\Q(\x^{(\ell)}\abcomp ,\x^{(\ell)}\abcomp )$ for all $\ell\in \{\zeta,\ldots, r\}$,  $q_{\zeta-1,\zeta} := \Q(\x^{(\zeta-1)}, \x^{(\zeta-1)}_{\nonexistence(a,b)}) = \lambda\left(1 -\frac{\rho_{ij} }{n} \right) $, $\bar q_{\zeta-1, \zeta}:= \Q(\x^{(\zeta-1)}, \x^{(\zeta-1)}_{\existence(a,b)}) = \lambda\frac{\rho_{ij}}{n}$, and $\bar q_{\zeta-1}= q_{\zeta-1}:=-\Q(\x^{(\zeta-1)}, \x^{(\zeta-1)})$.

The above definitions imply that $\frac{q_{\zeta-1,\zeta}}{\bar q_{\zeta-1,\zeta}}=\frac{1-\rho_{ij}/n }{\rho_{ij}/n}$. To evaluate $\frac{q_{\ell-1,\ell}}{\bar q_{\ell-1,\ell}}$ for $\ell\in\{\zeta+1,\ldots,r\}$, observe that by the definition of $\Lambda_{(a,b)}(F)$, we have $\x^{(\ell)}\notin \{\x^{(\ell-1)}_{\existence(a,b)}, \x^{(\ell-1)}_{\nonexistence(a,b)}\}$ for $\ell>\zeta=\Lambda_{(a,b)}(F)$. Moreover, the facts $F\subset \{\S(t)=\S_0\}$ and $a\in \S_0$ together imply that $\x^{(\ell)}\neq \x^{(\ell-1)}_{\infect a}$ for all $\ell\in [r]$. Hence, $\x^{(\ell)}\notin \{\x^{(\ell-1)}_{\existence(a,b)}, \x^{(\ell-1)}_{\nonexistence(a,b)},\x^{(\ell-1)}_{\infect a} \}$ for all $\ell>\zeta$. It now follows from the definition of $\Q$ (Section~\ref{sec:formulation}) that $\Q(\x^{(\ell-1)},\x^{(\ell)})=\Q(\x^{(\ell-1)}\abcomp,\x^{(\ell)}\abcomp)$ for all $\ell\in\{\zeta+1,\ldots, r\}$. Thus, $\frac{\prod_{\ell=\zeta}^r q_{\ell-1,\ell} }{\prod_{\ell=\zeta}^r \bar q_{\ell-1,\ell}}= \frac{1-\rho_{ij}/n }{\rho_{ij}/n}$.

To relate $\bar q_\ell$ to $ q_\ell$, note that $F\subset\{(a,b)\notin E(t)\}$ implies that $\x^{(\ell)}_{\langle a,b\rangle}=0$ and hence also that $\left(\x^{(\ell)}\abcomp\right)_{\langle a,b\rangle}=1$ for $\ell\geq \zeta$. Since $\x^{(\ell)}_u=\left(\x^{(\ell)}\abcomp\right)_u$ for all $u\neq\langle a,b\rangle$, we have 
\begin{align*}
    \Q\left (\x^{(\ell)}\abcomp, \left(\x^{(\ell)}\abcomp \right)_{\infect a}\right)&= \sum_{k=1}^m B_{ik }E_k^{(a)} \left (\x^{(\ell)}\abcomp\right)\cr
    &=\sum_{k=1}^m \sum_{c\in \I_k\left(\x^{(\ell)}\abcomp\right)} B_{ik} 1_{(a,c)}\left(\x^{(\ell)}\abcomp\right)\cr
    &=\sum_{k=1}^m \sum_{c\in \I_k\left(\x^{(\ell)}\right)} B_{ik} 1_{(a,c)}\left(\x^{(\ell)}\right) + B_{ij}1_{(a,b)}\left(\x^{(\ell)}\abcomp\right)\cr
    &=\Q\left( \x^{(\ell)} , \x^{(\ell)}_{\infect a} \right) + B_{ij}
\end{align*}
for all $\ell\in\{\zeta,\;\ldots, r\}$ such that $b\in\I_j(\x^{(\ell)})$, and
\begin{align*}
    \Q\left(\x^{(\ell)}\abcomp, \left(\x^{(\ell)}\abcomp \right)_{\infect a} \right) = \Q\left(\x^{(\ell)}, \x^{(\ell)}_{\infect a}\right)
\end{align*}
for all $\ell\in\{\zeta,\ldots, r\}$ such that $b\notin \I_j(\x^{(\ell)})$. As a result, 
\begin{align} \label{eq:plus_B_ij}
    \Q\left (\x^{(\ell)}\abcomp, \left(\x^{(\ell)}\abcomp \right)_{\infect a}\right)=
    \begin{cases}
        \Q\left( \x^{(\ell)} , \x^{(\ell)}_{\infect a} \right)\quad&\text{if }\zeta\leq \ell< \xi\\
        \Q\left( \x^{(\ell)} , \x^{(\ell)}_{\infect a} \right) + B_{ij}\quad&\text{if }\max\{\zeta,\xi\}\leq \ell\leq r.
    \end{cases}
\end{align}
Moreover, using the definition of $\Q$ one can easily verify that regardless of whether the network is in state $\x^{(\ell)}$ or in state $\x^{(\ell)}\abcomp$, the rates of infection of nodes in $\S(\x^{(\ell)})\setminus\{a\}$, the recovery rates of nodes in $\I(x^{(\ell)})$, and the edge update rate (which is $\lambda$) are the same. Given that $\Q(\x,\x)=-\sum_{\y\in\mathbb S\setminus\{\x\} }\Q(\x,\y)$ for all $\x\in\mathbb S$, it now follows from~\eqref{eq:plus_B_ij} that
\begin{align} 
    \bar q_\ell-q_\ell=
    \begin{cases}
       0\quad&\text{if }\zeta\leq \ell< \xi\\
       B_{ij}\quad&\text{if }\max\{\zeta,\xi\}\leq \ell\leq r.
    \end{cases}
\end{align}

Combining the above observations with~\eqref{eq:ratio_intermediate} yields
\begin{align*}
    \frac{\Pr(F^\delta) }{\Pr(F^\delta\abcomp)}&=\left(\frac{1-{\rho_{ij} }/{n} }{{\rho_{ij} }/{n}}\right)e^{(\bar q_r-q_r)(t-t_r)}\prod_{\ell=\zeta-1}^{r-1} \left( e^{(\bar q_{\ell}-q_{\ell})(t_{\ell+1}-t_{\ell}) }\delta + o(\delta) \right) \cr
    &\stackrel{(a)}= \left(\frac{1-{\rho_{ij} }/{n} }{{\rho_{ij} }/{n}}\right)\prod_{\ell=\zeta-1}^{r} \left( e^{(\bar q_{\ell}-q_{\ell})(t_{\ell+1}-t_{\ell}) }\delta + o(\delta) \right)\cr
    &= \left(\frac{1-{\rho_{ij} }/{n} }{{\rho_{ij} }/{n}}\right)\prod_{\ell=\zeta-1}^{\max\{\zeta,\xi\}-1} (1+o(\delta))\prod_{\ell=\max\{\zeta,\xi\}}^r \left( e^{B_{ij} (t_{\ell+1}-t_{\ell}) }\delta + o(\delta) \right)\cr 
    &=\left(\frac{1-{\rho_{ij}}/{n} }{{\rho_{ij} }/{n}}\right)e^{B_{ij}\left(t-t_{\max\{\zeta,\xi\}}\right)}+o(\delta),
\end{align*} 
which means that
\begin{align}\label{eq:this_will_clinch_it}
    \frac{\Pr(F^\delta) }{\Pr(F^\delta\abcomp)}=\left(\frac{1-{\rho_{ij}}/{n} }{{\rho_{ij} }/{n}}\right)e^{B_{ij}\min\{\tau,\kappa\}}+o(\delta).
\end{align}

We now use~\eqref{eq:bayes_with_question} and~\eqref{eq:delta_question} along with~\eqref{eq:this_will_clinch_it} to show that
\begin{align*}
    \Pr((a,b)\in E(t)\mid F^\delta\abagn)&=\left(1+\left(\frac{1-\rho_{ij}/n }{\rho_{ij}/n} \right)e^{B_{ij}\min\{\tau,\kappa\} }+o(\delta)\right)^{-1}\cr
    &= \frac{\rho_{ij}  }{n}\cdot\frac{1}{ \frac{\rho_{ij}  }{n}\cdot 1 + \left(1-\frac{\rho_{ij}  }{n}\right) e^{B_{ij}\min\{\tau,\kappa\} } }+o(\delta).
\end{align*}
In the limit as $\delta\rightarrow 0$, this yields
\begin{align}\label{eq:lower_bound}
    \Pr((a,b)\in E(t)\mid F\abagn )&=\frac{\rho_{ij}  }{n}\cdot\frac{1}{ \frac{\rho_{ij}  }{n}\cdot 1 + \left(1-\frac{\rho_{ij}  }{n}\right) e^{B_{ij}\min\{\tau,\kappa\} } }
\end{align}

Note that these bounds hold for all $F\in\F$ satisfying $t-t_{\Lambda_{(a,b)}(F) }=\tau$ and $t-t_{\Gamma_b(F)}=\kappa$. We now recall that $T$ is the difference between $t$ and the time at which $1_{(a,b)}$ is updated for the last time during $[0,t]$, so that we have $T=t- t_{\Lambda_{(a,b)}(F)}$ whenever $F\abagn$ occurs. Likewise, $K=t-t_{\Gamma_b(F)}$ on $F\abagn$. Therefore,~\eqref{eq:lower_bound} holds for all $F\in\F$ satisfying $F\subset\{T=\tau\}\cap\{K=\kappa\}$. As a result, we have
$$
    \Pr((a,b)\in E(t)\mid \cup_{F\in\F } F\abagn, T=\tau,K=\kappa)=\frac{\rho_{ij}  }{n}\cdot\frac{1}{ \frac{\rho_{ij}  }{n}\cdot 1 + \left(1-\frac{\rho_{ij}  }{n}\right) e^{B_{ij}\min\{\tau,\kappa\} } }.
$$
Since $\cup_{F\in\F}F\abagn=\{ (\S(t),\I(t))=(\S_0,\I_0) \} $ as argued earlier, it follows that
\begin{align}\label{eq:to_be_used_in_remarks}
    \Pr((a,b)\in E(t)\mid (\S(t),\I(t))=(\S_0,\I_0), T=\tau,K=\kappa)=\frac{\rho_{ij}  }{n}\cdot\frac{1}{ \frac{\rho_{ij}  }{n}\cdot 1 + \left(1-\frac{\rho_{ij}  }{n}\right) e^{B_{ij}\min\{\tau,\kappa\} } }.
\end{align}
Observe that $0\leq \min\{\kappa,\tau\}\leq \tau$, which means that 
\begin{align}
    \frac{\rho_{ij}  }{n}&\geq \Pr((a,b)\in E(t)\mid (\S(t),\I(t))=(\S_0,\I_0), T=\tau,K=\kappa)\cr
    &\geq \frac{\rho_{ij}  }{n}\cdot\frac{1}{ \frac{\rho_{ij}  }{n}\cdot 1 + \left(1-\frac{\rho_{ij}  }{n}\right) e^{B_{ij}\tau} }\cr
    &\geq\frac{\rho_{ij} }{n}e^{-B_{ij}\tau}.
\end{align}
Furthermore, since the above bounds do not depend on $\kappa$, we can remove the conditioning on $K$ to obtain
\begin{align}
    \frac{\rho_{ij}  }{n}&\geq \Pr((a,b)\in E(t)\mid (\S(t),\I(t))=(\S_0,\I_0), T=\tau)
    \geq\frac{\rho_{ij} }{n}e^{-B_{ij}\tau}.
\end{align}
Consequently, 
\begin{align*}
    \frac{\rho_{ij} }{n} &\geq \Pr((a,b)\in E(t)\mid (\S(t),\I(t))=(\S_0,\I_0))\cr
    &\geq \frac{\rho_{ij} }{n} \int_0^t e^{-B_{ij}\tau} f_{T\mid (\S(t),\I(t)) = (\S_0, \I_0)} (\tau)d\tau\cr
    &\geq\frac{\rho_{ij}}{n}\left(1-\frac{B_{ij} }{\lambda}(1-e^{-\lambda t}) \right),
\end{align*}
where the second inequality follows from Lemmas~\ref{lem:semi-final_integral_lemma} and~\ref{lem:final_integral_lemma}. Since $(\S_0,\I_0)$ is an arbitrary realization of $(\S(t),\I(t))$, the assertion of the proposition follows.
\end{proof}

\begin{remark}\label{rem:for_converse_result} The proof of Proposition~\ref{prop:bounds} enables us to make a stronger statement about the conditional probability of $b$ being in contact with $a$ at time $t$. Indeed, consider~\eqref{eq:to_be_used_in_remarks} and observe that it holds for all realizations $(\S_0,\I_0)$ of $(\S(t),\I(t))$ that satisfy $a\in\S_0\cap \A_i$ and $b\in\I_0\cap\A_j$. It follows that
$$
    \Pr((a,b)\in E(t)\mid \S(t),\I(t), T=\tau,K=\kappa)=\frac{\rho_{ij}  }{n}\cdot\frac{1}{ \frac{\rho_{ij}  }{n}\cdot 1 + \left(1-\frac{\rho_{ij}  }{n}\right) e^{B_{ij}\min\{\tau,\kappa\} } }
$$
holds for all node pairs $(a,b)\in\S_i(t)\times\I_j(t)$. 
Equivalently, the following holds for all $(a,b)\in\S_i(t)\times\I_j(t)$:
$$
    \Pr((a,b)\in E(t)\mid \S(t),\I(t), T,K)=\frac{\rho_{ij}  }{n}\cdot\frac{1}{ \frac{\rho_{ij}  }{n}\cdot 1 + \left(1-\frac{\rho_{ij}  }{n}\right) e^{B_{ij}\min\{T,K\} } }.
$$
\end{remark}

\begin{lemma}\label{lem:last_needed}
Let $\tau_1,\tau_2\in [0,t]$. Then
$$
    1\leq \frac{\Pr\bigg(b\stackrel{0,t}{\centernot\rightsquigarrow}a \,\Big\lvert\, b\in\I(t^-), K=\kappa, T=\tau_2, (a,b)\notin E(t) \bigg) }{\Pr\bigg(b\stackrel{0,t}{\centernot\rightsquigarrow}a \,\Big\lvert\, b\in\I(t^-), K=\kappa, T=\tau_1, (a,b)\notin E(t) \bigg)}\leq e^{B_{ij}(\tau_2-\tau_1)}.
$$
\end{lemma}

\begin{proof}
Consider $T\on^{(\tau_2)}:=\left(\int_{t-\kappa}^{t-\tau_2}1_{(a,b)}(\tau')d\tau'\right)_+$, which denotes the duration of time for which $b$ is in contact with $a$ during the interval $[t-\kappa,t-\tau_2)$. Then, for any $\sigma\in[0,(\kappa-\tau_2)_+]$, we have
\begin{align}\label{eq:with_sigma}
    &\Pr\bigg(b\stackrel{0,t}{\centernot\rightsquigarrow}a \,\Big\lvert\, T\on^{(\tau_2)}=\sigma, b\in\I(t^-),K=\kappa,  T=\tau_1, (a,b)\notin E(t) \bigg)\cr
    &=\Pr\bigg(b\stackrel{t-\kappa,\kappa-\tau_1}{\centernot\rightsquigarrow}a \,\Big\lvert\, T\on^{(\tau_2)}=\sigma, b\in\I(t^-), K=\kappa, T=\tau_1, (a,b)\notin E(t) \bigg)\cr
    &=\Pr\bigg(b\stackrel{t-\kappa,\kappa-\tau_2}{\centernot\rightsquigarrow}a \,\Big\lvert\, T\on^{(\tau_2)}=\sigma, b\in\I(t^-), K=\kappa, T=\tau_1, (a,b)\notin E(t) \bigg)\cr
    &\quad\times \Pr\bigg(b\stackrel{t-\tau_2,\tau_2-\tau_1}{\centernot\rightsquigarrow}a \,\Big\lvert\, T\on^{(\tau_2)}=\sigma, b\in\I(t^-), K=\kappa, T=\tau_1, (a,b)\notin E(t)\bigg)\cr
    &\stackrel{(a)}{\geq}\Pr\bigg(b\stackrel{t-\kappa,\kappa-\tau_2}{\centernot\rightsquigarrow}a \,\Big\lvert\, T\on^{(\tau_2)}=\sigma, b\in\I(t^-), K=\kappa, T=\tau_1, (a,b)\notin E(t) \bigg) e^{-B_{ij}(\tau_2-\tau_1)}\cr
    &\stackrel{(b)}{\geq} e^{-B_{ij}\sigma}\cdot e^{-B_{ij}(\tau_2-\tau_1)}\cr
    &\stackrel{(c)}{=} \Pr\bigg(b\stackrel{t-\kappa,\kappa-\tau_2}{\centernot\rightsquigarrow}a \,\Big\lvert\, T\on^{(\tau_2)}=\sigma, b\in\I(t^-), K=\kappa, T=\tau_2, (a,b)\notin E(t) \bigg) e^{-B_{ij}(\tau_2-\tau_1)}\cr
    &=\Pr\bigg(b\stackrel{0,t}{\centernot\rightsquigarrow}a \,\Big\lvert\, T\on^{(\tau_2)}=\sigma, b\in\I(t^-),K=\kappa,  T=\tau_2, (a,b)\notin E(t) \bigg)e^{-B_{ij}(\tau_2-\tau_1)},
\end{align}
where (a) follows from the fact that $\int_{t-\tau_2}^{t-\tau_1} 1_{(a,b)}(\tau')d\tau'\leq \tau_2-\tau_1$ and from the observation that
\begin{align*}
    &\Pr\bigg(b\stackrel{t-\tau_2,\tau_2-\tau_1}{\centernot\rightsquigarrow}a \,\Big\lvert\, T\on^{(\tau_2)}=\sigma, b\in\I(t^-), K=\kappa, T=\tau_1, (a,b)\notin E(t),\int_{t-\tau_2}^{t-\tau_1} 1_{(a,b)}(\tau')d\tau'\bigg)\cr
    &=e^{-B_{ij}\left(\int_{t-\tau_2}^{t-\tau_1} 1_{(a,b)}(\tau')d\tau'\right)},
\end{align*}
the proof of which parallels the proof of Lemma~\ref{lem:T_on}, and $(b)$ and $(c)$ follow from the fact that
\begin{align*}
    &\Pr\bigg(b\stackrel{t-\kappa,\kappa-\tau_2}{\centernot\rightsquigarrow}a \,\Big\lvert\, T\on^{(\tau_2)}=\sigma, b\in\I(t^-), K=\kappa, T=\tau', (a,b)\notin E(t) \bigg) = e^{-B_{ij}\sigma}
\end{align*}
holds for all $\tau'\in[0,t]$, the proof of which also parallels the proof of Lemma~\ref{lem:T_on}.

We now eliminate $T\on^{(\tau_2)}$ from~\eqref{eq:with_sigma}. To do so, we first note the following: given that $t-T\geq t-\tau_2$, the random variable $T\on^{(\tau_2)}$ is by definition conditionally independent of $t-T$ (the time of the last edge update of $(a,b)$ during $[t-\tau_2,t]$) because the edge update process for $(a,b)$ is a Poisson process and hence, for any collection of disjoint time intervals, the times at which $1_{(a,b)}$ is updated during the intervals are independent of each other. Since $\{T=\tau_2\},\{T=\tau_1\}\subset\{t-T\geq t-\tau_2\}$ and since $\{t-T\geq t-\tau_2\}=\{T\leq \tau_2\}$, it follows that
\begin{align}\label{eq:late_label_1}
    &\Pr\bigg(b\stackrel{0,t}{\centernot\rightsquigarrow}a \,\Big\lvert\, b\in\I(t^-), K=\kappa, T=\tau_1, (a,b)\notin E(t) \bigg)\cr
    &=\int_0^{(\kappa-\tau_2)_+}\Pr\bigg(b\stackrel{0,t}{\centernot\rightsquigarrow}a \,\Big\lvert\, T\on^{(\tau_2)}=\sigma, b\in\I(t^-), K=\kappa,  T=\tau_1, (a,b)\notin E(t) \bigg)\cr
    &\quad\quad\quad\quad\quad\quad \cdot f_{T\on^{(\tau_2)}\mid b\in\I(t^-), K=\kappa,T=\tau_1,(a,b)\notin E(t) }(\sigma)d\sigma\cr
    &=\int_0^{(\kappa-\tau_2)_+}\Pr\bigg(b\stackrel{0,t}{\centernot\rightsquigarrow}a \,\Big\lvert\, T\on^{(\tau_2)}=\sigma, b\in\I(t^-), K=\kappa,  T=\tau_1, (a,b)\notin E(t) \bigg)\cr
    &\quad\quad\quad\quad\quad\quad \cdot f_{T\on^{(\tau_2)}\mid b\in\I(t^-), K=\kappa,T\leq \tau_2,(a,b)\notin E(t) }(\sigma)d\sigma,
\end{align}
and likewise,
\begin{align*}\label{eq:late_label_2}
    &\Pr\bigg(b\stackrel{0,t}{\centernot\rightsquigarrow}a \,\Big\lvert\, b\in\I(t^-), K=\kappa, T=\tau_2, (a,b)\notin E(t) \bigg)\cr
    &=\int_0^{(\kappa-\tau_2)_+}\Pr\bigg(b\stackrel{0,t}{\centernot\rightsquigarrow}a \,\Big\lvert\, T\on^{(\tau_2)}=\sigma, b\in\I(t^-), K=\kappa, T=\tau_2, (a,b)\notin E(t) \bigg)\cr
    &\quad\quad\quad\quad\quad \quad\cdot f_{T\on^{(\tau_2)}\mid b\in\I(t^-), K=\kappa,T\leq \tau_2,(a,b)\notin E(t) }(\sigma)d\sigma.
\end{align*}
Therefore, taking conditional expectations on both sides of~\eqref{eq:with_sigma} (w.r.t. the PDF $f_{T\on^{(\tau_2)}\mid b\in\I(t^-), K=\kappa,T\leq \tau_2,(a,b)\notin E(t) }$) yields
\begin{align*}
    &\Pr\bigg(b\stackrel{0,t}{\centernot\rightsquigarrow}a \,\Big\lvert\, b\in\I(t^-), K=\kappa, T=\tau_1, (a,b)\notin E(t) \bigg)\cr
    &\geq \Pr\bigg(b\stackrel{0,t}{\centernot\rightsquigarrow}a \,\Big\lvert\, b\in\I(t^-), K=\kappa, T=\tau_2, (a,b)\notin E(t) \bigg)e^{-B_{ij}(\tau_2-\tau_1)},
\end{align*}
which proves the upper bound. For the lower bound, we again proceed as in~\eqref{eq:with_sigma}, but reverse the inequality signs:
\begin{align*}
    &\Pr\bigg(b\stackrel{0,t}{\centernot\rightsquigarrow}a \,\Big\lvert\, T\on^{(\tau_2)}=\sigma, b\in\I(t^-),K=\kappa,  T=\tau_1, (a,b)\notin E(t) \bigg)\cr
    &=\Pr\bigg(b\stackrel{t-\kappa,\kappa-\tau_2}{\centernot\rightsquigarrow}a \,\Big\lvert\, T\on^{(\tau_2)}=\sigma, b\in\I(t^-), K=\kappa, T=\tau_1, (a,b)\notin E(t) \bigg)\cr
    &\quad\times \Pr\bigg(b\stackrel{t-\tau_2,\tau_2-\tau_1}{\centernot\rightsquigarrow}a \,\Big\lvert\, T\on^{(\tau_2)}=\sigma, b\in\I(t^-), K=\kappa, T=\tau_1, (a,b)\notin E(t)\bigg)\cr
    &\stackrel{}{\leq}\Pr\bigg(b\stackrel{t-\kappa,\kappa-\tau_2}{\centernot\rightsquigarrow}a \,\Big\lvert\, T\on^{(\tau_2)}=\sigma, b\in\I(t^-), K=\kappa, T=\tau_1, (a,b)\notin E(t) \bigg)\cr
    &\stackrel{}{=} e^{-B_{ij}\sigma}\cr
    &\stackrel{}{=} \Pr\bigg(b\stackrel{t-\kappa,\kappa-\tau_2}{\centernot\rightsquigarrow}a \,\Big\lvert\, T\on^{(\tau_2)}=\sigma, b\in\I(t^-), K=\kappa, T=\tau_2, (a,b)\notin E(t) \bigg)\cr
    &=\Pr\bigg(b\stackrel{0,t}{\centernot\rightsquigarrow}a \,\Big\lvert\, T\on^{(\tau_2)}=\sigma, b\in\I(t^-),K=\kappa,  T=\tau_2, (a,b)\notin E(t) \bigg)
\end{align*}
In light of~\eqref{eq:late_label_1} and~\eqref{eq:late_label_2}, the above inequality implies the required lower bound.
\end{proof}

\begin{lemma}\label{lem:semi-final_integral_lemma} Let $T$ denote the random time defined earlier. Then
\begin{gather*}
    \int_0^t e^{-B_{ij}\tau}f_{T\mid (\S(t),\I(t)) = (\S_0,\I_0), (a,b)\notin E(t)}(\tau)d\tau 
    \geq 1 - \frac{B_{ij} }{\lambda}(1-e^{-\lambda t}).
\end{gather*}
\end{lemma}
\begin{proof}
We first use Bayes' rule to note that for any $\kappa\in[0,t]$,
\begin{align}\label{eq:label_unknown}
    &f_{T\mid K=\kappa, (\S(t),\I(t))=(\S_0,\I_0),(a,b)\notin E(t)}(\tau) \cr
    &= \Pr((\S(t),\I(t))=(\S_0,\I_0)\mid T=\tau, K=\kappa, (a,b)\notin E(t))\cdot f_{K\mid T=\tau, (a,b)\notin E(t)}(\kappa)\cdot f_{T\mid (a,b)\notin E(t)}(\tau)\cr
    &\quad\cdot \bigg(\int_0^t \Pr((\S(t),\I(t))=(\S_0,\I_0)\mid T=\tau', K=\kappa, (a,b)\notin E(t))\cdot f_{K\mid T=\tau', (a,b)\notin E(t)}(\kappa)\cdot f_{T\mid (a,b)\notin E(t)}(\tau') d\tau'\bigg)^{-1}.\nonumber\\
\end{align}
We consider each multiplicand one by one. First, we use Lemmas~\ref{lem:difficult} and~\ref{lem:simple_but_difficult} to note that $f_{K\mid T=\tau,(a,b)\notin E(t)}(\kappa)\cdot f_{T\mid (a,b)\notin E(t) }(\tau)=f_K(\kappa)f_T(\tau)$. To deal effectively with the other multiplicands, we let $T\on:=\left(\int_{t-K}^{t-T} 1_{(a,b)}(\tau')d\tau'\right)_+$ denote the total duration of time for which $b$ is in contact with $a$ during the time interval $[t-K,t-T]$, and we observe that  
\begin{align*}
    &\Pr((\S(t),\I(t))=(\S_0,\I_0)\mid T=\tau, K=\kappa, (a,b)\notin E(t))\cr
    &=\int_0^{(\kappa-\tau)_+}\Big( \Pr((\S(t),\I(t))=(\S_0,\I_0)\mid T=\tau, K=\kappa, T\on=\tau\on, (a,b)\notin E(t))\cr
    &\quad\quad\quad\quad\quad\quad\cdot f_{T\on\mid K=\kappa, T=\tau, (a,b)\notin E(t)}(\tau\on)\Big) d\tau\on\cr
    &=\int_0^{(\kappa-\tau)_+}\Big( \Pr((\S(t),\I(t))=(\S_0,\I_0)\mid T=\tau, K=\kappa,T\on=\tau\on, (a,b)\notin E(t), b\stackrel{0,t}{\centernot\rightsquigarrow}a,b\in \I(t^-) )\cr
    &\quad\quad\quad\quad\quad\quad\cdot \Pr(b\stackrel{0,t}{\centernot \rightsquigarrow}a\mid b\in\I(t^-), T=\tau,K=\kappa, T\on=\tau\on, (a,b)\notin E(t))\cr
    &\quad\quad\quad\quad\quad\quad\cdot \Pr(b\in\I(t^-)\mid T=\tau, K=\kappa, T\on=\tau\on, (a,b)\notin E(t) )\cr
    &\quad\quad\quad\quad\quad\quad\cdot f_{T\on\mid K=\kappa, T=\tau, (a,b)\notin E(t)}(\tau\on)\Big) d\tau\on\cr
    &\stackrel{(a)}{=}\int_0^{(\kappa-\tau)_+}\Big( \Pr((\S(t),\I(t))=(\S_0,\I_0)\mid K=\kappa, (a,b)\notin E(t), b\stackrel{0,t}{\centernot\rightsquigarrow}a,b\in \I(t^-) )\cr
    &\quad\quad\quad\quad\quad\quad\cdot e^{-B_{ij}\tau\on}\cdot e^{-\gamma_j\kappa} \cdot f_{T\on\mid K=\kappa, T=\tau, (a,b)\notin E(t)}(\tau\on)\Big) d\tau\on,
\end{align*}
where $(a)$ follows from Lemmas~\ref{lem:T_on} and~\ref{lem:equation_only} and from the modelling assumption that $b$ recovers at rate $\gamma_j$ independently of any edge state.  On substituting the above expression into~\eqref{eq:label_unknown}, we obtain
\begin{align}\label{eq:num_denom}
    &f_{T\mid K=\kappa, (\S(t),\I(t))=(\S_0,\I_0), (a,b)\notin E(t) }(\tau)=\frac{ \left(\int_0^{(\kappa-\tau)_+} e^{-B_{ij}\tau\on} \psi_{\kappa,\tau}(\tau\on)d\tau\on \right) f_T(\tau) }{  \int_0^t \left(\int_0^{(\kappa-\tau')_+} e^{-B_{ij}\tau\on} \psi_{\kappa,\tau'}(\tau\on)d\tau\on \right) f_T(\tau') d\tau'  },
\end{align}
where $\psi_{\kappa,\tau}(\cdot):=f_{T\on\mid K=\kappa,T=\tau,(a,b)\notin E(t)}(\cdot)$. Now, Lemma~\ref{lem:T_on} implies that
\begin{align*}
    &\int_0^{(\kappa-\tau)_+} e^{-B_{ij}\tau\on} \psi_{\kappa,\tau}(\tau\on)d\tau\on\cr
    &=\int_0^{(\kappa-\tau)_+}\Pr\bigg(b\stackrel{0,t}{\centernot\rightsquigarrow}a \,\Big\lvert\, T\on=\tau\on, b\in\I(t^-), K=\kappa, T=\tau, (a,b)\notin E(t) \bigg)\cdot f_{T\on\mid K=\kappa, T=\tau, (a,b)\notin E(t)}(\tau\on) d\tau\on\cr
    &\stackrel{(a)}{=}\int_0^{(\kappa-\tau)_+}\Pr\bigg(b\stackrel{0,t}{\centernot\rightsquigarrow}a \,\Big\lvert\, T\on=\tau\on, b\in\I(t^-), K=\kappa, T=\tau, (a,b)\notin E(t) \bigg)\cr
    &\quad\quad\quad\quad\quad\,\,\cdot f_{T\on\mid b\in\I(t^-), K=\kappa, T=\tau, (a,b)\notin E(t)}(\tau\on) d\tau\on\cr
    &=\Pr\bigg(b\stackrel{0,t}{\centernot\rightsquigarrow}a \,\Big\lvert\, b\in\I(t^-), K=\kappa, T=\tau, (a,b)\notin E(t) \bigg),
\end{align*}
where $(a)$ holds because the recovery time of $b$ is conditionally independent of $T\on$ given $K=\kappa$ (recall that $b$ recovers at rate $\gamma_j$ independently of any edge state (and hence independently of $T\on$), and  $\{b\in\I(t^-)\}$ is precisely the event that the recovery time of $b$ is at least $K$). Hence,~\eqref{eq:num_denom} implies that for any $\tau_1,\tau_2\in [0,t]$ satisfying $\tau_1\leq \tau_2$, we have
\begin{align}\label{eq:fractions}
    &\frac{g(\tau_2)}{g(\tau_1)}=\frac{\Pr\bigg(b\stackrel{0,t}{\centernot\rightsquigarrow}a \,\Big\lvert\, b\in\I(t^-), K=\kappa, T=\tau_2, (a,b)\notin E(t) \bigg) }{\Pr\bigg(b\stackrel{0,t}{\centernot\rightsquigarrow}a \,\Big\lvert\, b\in\I(t^-), K=\kappa, T=\tau_1, (a,b)\notin E(t) \bigg)}\cdot \frac{f_T(\tau_2) }{f_T(\tau_1)},
\end{align}
where $g(\cdot):=f_{T\mid K=\kappa (\S(t),\I(t))=(\S_0,\I_0),(a,b)\notin E(t)} (\cdot)$.

As a consequence of~\eqref{eq:fractions} and Lemma~\ref{lem:last_needed}, we have
\begin{align*}
    \frac{g(\tau_2)}{g(\tau_1)} \leq e^{B_{ij}(\tau_2-\tau_1)} \frac{f_T(\tau_2)}{f_T(\tau_1)}.
\end{align*}
Since $f_{T}(\tau)=\lambda e^{-\lambda\tau}+e^{-\lambda t}\delta(\tau-t)$ for $\tau\in[0,t]$ (see Lemma~\ref{lem:pdf_of_T}), we have the following for all $0\leq \tau_1\leq \tau_2<t$:
\begin{align}\label{eq:first_important}
    g(\tau_2)\leq e^{-(\lambda-B_{ij})(\tau_2-\tau_1)}g(\tau_1),
\end{align}
and for all $0\leq \tau<t$, we have
\begin{align}\label{eq:second_important} 
    \tilde g(t)\leq \frac{e^{-(\lambda-B_{ij})(t-\tau)}}{ \lambda }g(\tau),
\end{align}
where $\tilde g(t)$ scales $\delta(0)$ so that $g(t)=\tilde g(t)\delta(0)$. Since $\delta$ is the Dirac-delta function,~\eqref{eq:second_important} simply means that 
\begin{gather*}
    \Pr(T=t \mid K=\kappa, (\S(t),\I(t))=(\S_0,\I_0),(a,b)\notin E(t))
    \leq \frac{e^{-(\lambda-B_{ij})(t-\tau)}}{ \lambda }g(\tau).
\end{gather*}
Our next goal is to use~\eqref{eq:first_important} and~\eqref{eq:second_important} to show that
\begin{align}\label{eq:which_is_to_be_shown}
    \int_0^t e^{-B_{ij}\tau} g(\tau)d\tau \geq \int_0^t e^{-B_{ij}\tau} \varphi(\tau) d\tau,
\end{align}
where $\varphi$ is the probability density function defined by $\varphi(\tau):=(\lambda-B_{ij})e^{-(\lambda-B_{ij})\tau}+e^{-(\lambda-B_{ij})t}\delta(\tau-t)$ for all $\tau\in[0,t]$ and $\varphi(\tau)=0$ for $\tau>t$. To this end, we compare $g$ with $\varphi$ under the following two cases.

\textit{Case 1: }There exists a time $ \tau_0\in [0,t)$ such that $g(\tau_0)<\varphi(\tau_0)$. In this case,~\eqref{eq:first_important} implies that  for all $\tau\in [\tau_0,t)$, 
\begin{align*}
    g(\tau)&\leq e^{-(\lambda-B_{ij})(\tau-\tau_0 )}g(\tau_0)\cr
    &< e^{-(\lambda-B_{ij})(\tau-\tau_0 )}(\lambda-B_{ij})e^{-(\lambda-B_{ij})(\tau_0)}\cr
    &=\varphi(\tau),
\end{align*}
which means that the set $\{\tau\in[0,t):g(\tau)<\varphi(\tau)\}$ is either $[\tau^*,t)$ or  $(\tau^*,t)$, where $\tau^*:=\inf\{\tau:g(\tau)<\varphi(\tau)\}$. Also, by the definition of $\tau^*$, we have $g(\tau)\geq \varphi(\tau)$ for all $\tau\in[0,\tau^*)$. Next, to compare $g$ and $\varphi$ at $\tau=t$, we use~\eqref{eq:second_important} to note that
\begin{align}\label{eq:separator}
    g(t)&\leq \frac{ e^{-(\lambda-B_{ij} )(t-\tau_0)}  }{\lambda}g(\tau_0)\delta(0)\cr
    &\leq \frac{ e^{-(\lambda-B_{ij} )(t-\tau_0)}  }{\lambda} (\lambda-B_{ij}) e^{-(\lambda-B_{ij})\tau_0}\delta(0)\cr
    &=\left(1-\frac{B_{ij}}{\lambda}\right)e^{-(\lambda-B_{ij})t}\delta(0)\cr
    &\leq \varphi(t).
\end{align}
Thus, $g(\tau)-\varphi(\tau)\geq 0$ for all $\tau\in[0,\tau^*)$ and $g(\tau)-\varphi(\tau)\leq 0$ for all $\tau\in(\tau^*,t]$. Now, since $g$ and $\varphi$ are both PDFs, we must have $\int_0^\infty (g(\tau)-\varphi(\tau))d\tau=0$ or equivalently,
$$
    \int_0^{\tau^*} (g(\tau)-\varphi(\tau))d\tau=\int_{\tau^*}^\infty (\varphi(\tau)-g(\tau))d\tau.
$$
Since both the integrands above are non-negative, we have
\begin{align*}
    &\int_0^{\tau^*} e^{-B_{ij}\tau} (g(\tau)-\varphi(\tau))d\tau\cr
    &\geq e^{-B_{ij}\tau^*}\int_0^{\tau^*}(g(\tau)-\varphi(\tau))d\tau\cr
    &=e^{-B_{ij}\tau^*}\int_{\tau^*}^\infty (\varphi(\tau)-g(\tau))d\tau\cr
    &\geq \int_{\tau^*}^\infty e^{-B_{ij}\tau}(\varphi(\tau)-g(\tau))d\tau.
\end{align*}
Adding $\int_{\tau^*}^\infty e^{-B_{ij}\tau} g(\tau)d\tau+ \int_0^{\tau^*} e^{-B_{ij}\tau}\varphi(\tau)d\tau$ to both sides now yields~\eqref{eq:which_is_to_be_shown}.

\textit{Case 2:} $g(\tau)\geq \varphi(\tau)$ for all $\tau\in [0,t)$. In this case, we can simply set $\tau^*=t$ and repeat the arguments following ~\eqref{eq:separator} in Case 1 to show that~\eqref{eq:which_is_to_be_shown} holds. 

Next, we use the definition of $\varphi$ to evaluate $\int_0^t e^{-B_{ij}\tau }\varphi(\tau)d\tau$, and we then restate \eqref{eq:which_is_to_be_shown} as follows:
\begin{align}
    \int_0^t e^{-B_{ij}\tau} f_{T\mid K=\kappa,(\S(t),\I(t))=(\S_0,\I_0),(a,b)\notin E(t)}(\tau)d\tau\geq 1-\frac{B_{ij} }{\lambda}(1-e^{-\lambda t}).
\end{align}
Since this holds for all $\kappa\in[0,t)$, the assertion of  the lemma follows. 
\end{proof}

Using arguments very similar to the proof above, we can prove the following result.

\begin{lemma}\label{lem:final_integral_lemma}
Let $T$ denote the random time defined earlier. Then
\begin{gather*}
    \int_0^t e^{-B_{ij}\tau}f_{T\mid (\S(t),\I(t)) = (\S_0,\I_0), (a,b)\in E(t)}(\tau)d\tau \geq  1 - \frac{B_{ij} }{\lambda}(1-e^{-\lambda t}).
\end{gather*}
\end{lemma}

\begin{remark}\label{rem:should_be_the_last}
In the proof of Lemma~\ref{lem:semi-final_integral_lemma}, if, instead of using~\eqref{eq:fractions} along with the upper bound in Lemma~\ref{lem:last_needed}, we had used~\eqref{eq:fractions} along with the lower bound in Lemma~\ref{lem:last_needed}, we would have obtained
$$
    \frac{g(\tau_2)}{g(\tau_1)}\geq \frac{f_T(\tau_2)}{f_{T}(\tau_1)}=e^{-\lambda(\tau_2-\tau_1)}.
$$
In addition, if we had subsequently replaced $t$ with a generic $\tau\in [0,t)$ and the weighting function $[0,\infty)\ni\tau\rightarrow e^{-B_{ij}\tau}\in (0,\infty)$ by the constant function $1$, and if we had defined $\varphi$ by $\varphi(\tau):=\lambda e^{-\lambda\tau}+e^{-\lambda t}\delta(\tau-t)$, then using the same arguments but with reversed inequality signs, we would have been able to prove that
$$
    \int_0^\tau 1\cdot g(\tau')d\tau'\leq \int_0^\tau 1\cdot \varphi(\tau')d\tau'.
$$
Since the integral on the left-hand-side is $\Pr(T\leq \tau\mid K=\kappa,(\S(t),\I(t))=(\S_0,\I_0),(a,b)\notin E(t))$ and since the right-hand-side evaluates to $1-e^{-\lambda \tau}$, we conclude that
$$
    \Pr(T\leq \tau\mid K=\kappa,(\S(t),\I(t))=(\S_0,\I_0),(a,b)\notin E(t))\leq 1 - e^{-\lambda \tau}
$$
for all $\tau\in[0,t)$ and all $\kappa\in[0,t]$.
\end{remark}

\subsection*{Some Auxiliary Lemmas}

In addition to the above results, the proof of Theorem~\ref{thm:main} relies on the following lemmas, which we reproduce from~\cite{armbruster2017elementary}.

\begin{lemma}\label{lem:sum_var}
For random variables $Y$ and $Z$, we have $\variance[Y+Z]\leq 2(\variance[Y]+\variance[Z])$.
\end{lemma}

\begin{lemma} \label{lem:sq_var}
For a random variable $Y\in [0,1]$, we have $\variance[Y^2]\leq 4\variance[Y]$.
\end{lemma}

\begin{lemma}\label{lem:split}
For random variables $Y$ and $Z$ in $[0,1]$,
$$
    |\mathrm{E}[Y Z]-\mathrm{E}[Y] \mathrm{E}[Z]| \leq(\operatorname{Var}[Y]+\operatorname{Var}[Z]) / 2
$$
$$
    \left|\mathrm{E}\left[Y^{2} Z\right]-\mathrm{E}[Y]^{2} \mathrm{E}[Z]\right| \leq 2(\operatorname{Var}[Y]+\operatorname{Var}[Z]).
$$
\end{lemma}

The following result is a straightforward consequence of the above lemmas.
\begin{corollary}\label{cor:straightforward}
For non-negative random variables $Y$ and $Z$ satisfying $0\leq Y+Z\leq 1$, we have
$$
\variance[YZ]\leq 8(\variance[Y]+\variance[Z]).
$$
\end{corollary}
\begin{proof}
Note that
\begin{align*}
    4\variance[YZ] = \variance[2YZ] &= \variance[(Y+Z)^2 + (-1)(Y^2+Z^2)]\cr
    &\stackrel{(a)}{\leq} 2(\variance[(Y+Z)^2] + (-1)^2\variance[Y^2 + Z^2])\cr
    &\stackrel{(b)}{\leq} 2( 4\variance[Y+Z] + 2( \variance[Y^2] + \variance[Z^2] ) )\cr
    &\stackrel{(c)}{\leq} 2( 4\variance[Y+Z] + 2( 4\variance[Y] + 4\variance[Z] ) ),
\end{align*}
where (a) follows from Lemma~\ref{lem:sum_var}, (b) from both Lemma~\ref{lem:sum_var} and Lemma~\ref{lem:sq_var}, and (c) from Lemma~\ref{lem:sq_var} alone. Thus,
$$
\variance[YZ] \leq 2(\variance[Y+Z] + 2(\variance[Y]+\variance[Z]))\leq 8(\variance[Y] + \variance[Z]),
$$
where the last inequality follows from Lemma~\ref{lem:sum_var}.
\end{proof}

\subsection*{Proof of Theorem~\ref{thm:main}}
\begin{proof}
The proof is based on Proposition~\ref{prop:bounds} and it follows the approach used in~\cite{armbruster2017elementary}. We first modify Equations~\eqref{item:init_1} -~\eqref{item:init_4} (Proposition~\ref{prop:basic}) by expressing the expectations of cross-terms such as $\E[s_i\beta_j]$ in terms of expectations of individual terms such as $\E[s_i^2]$ and $\E[\beta_j]$. To begin, we apply Lemma~\ref{lem:split} to $\E[s_i(t)\beta_j(t)]$ and obtain
\begin{align*}
    |\E[s_i(t)\beta_j(t)] - \E[s_i(t)]\E[\beta_j(t)]|\leq  \frac{1}{2}(\variance[s_i(t)]+\variance[\beta_j(t)]).
\end{align*}
Therefore, there exists a function $h_{i,j,1,n}:[0,\infty)\rightarrow[-1,1]$ such that 
\begin{align*}
    \E[s_i(t)\beta_j(t)] &= \E[s_i(t)]\E[\beta_j(t)] + \frac{h_{i,j,1,n}(t)}{2}(\variance[s_i(t)]+\variance[\beta_j(t)]).
\end{align*}

Similarly, we can use Lemma~\ref{lem:split} to show that there exists a function $h_{i,j,2,n}:[0,\infty)\rightarrow[-1,1]$ such that 
\begin{align*}
    \E[s_i^2(t)\beta_j(t)]&= \E[s_i(t)]^2\E[\beta_j(t)]+ 2h_{i,j,2,n}(t)(\variance[s_i(t)]+\variance[\beta_j(t)]).
\end{align*}
Next, we use Corollary~\ref{cor:straightforward} to express $\E[s_i(t)\beta_j(t)\beta_i(t)]$ as
\begin{align*}
    \E[s_i(t)\beta_j(t)\beta_i(t)]&= \E[s_i(t)\beta_j(t)]\E[\beta_i(t)] + \frac{h_{i,j,5,n}(t) }{2}(\variance[s_i(t)\beta_j(t)] + \variance[\beta_i(t)])\\
    &= \Big(  \E[s_i(t)]\E[\beta_j(t)]+ \frac{h_{i,j,1,n}(t)}{2}(\variance[s_i(t)]+\variance[\beta_j(t)]) \Big)\E[\beta_i(t)]\cr
    &\quad+ \frac{h_{i,j,5,n}(t)}{2}( h_{i,j,6,n}(t) (8\variance[s_i(t)]+8\variance[\beta_j(t)]) +\variance[\beta_i(t)]),
\end{align*}
where $h_{i,j,5,n}(t)\in[-1,1]$ and $h_{i,j,6,n}(t)\in[0,1]$.

We thus obtain the following relations:
\begin{enumerate} [(I)]
    \item \label{item:simplify_1} 
    $\E[s_i\beta_j] = \E[s_i]\E[\beta_j] + \frac{h_{i,j,1,n}}{2}(\variance[s_i]+\variance[\beta_j])$,
    \item \label{item:simplify_2}
    $\E[s_i^2\beta_j]= \E[s_i]^2\E[\beta_j] + 2h_{i,j,2,n}(\variance[s_i]+\variance[\beta_j])$,
    \item \label{item:simplify_3} 
        $\E[s_i\beta_j \beta_i ] = \Big(  \E[s_i ]\E[\beta_j ] + \frac{h_{i,j,1,n} }{2}(\variance[s_i ]+\variance[\beta_j ]) \Big)+ \frac{h_{i,j,5,n} }{2}( h_{i,j,6,n}  (8\variance[s_i ]+8\variance[\beta_j ])+\variance[\beta_i ]).$
\end{enumerate}

To handle terms of the form $B_{ij}\E[n\cdot\chi_{ij}(t,\S, \I)\cdot U]$ where $U$ is some random variable, we use Proposition~\ref{prop:bounds} to obtain
\begin{align*}
    A_{ij}\left(1-\frac{B_{ij} }{\lambda^{(n)} }(1-e^{-\lambda^{(n)} t}) \right)\E[U]\leq B_{ij} \E[n\chi_{ij}(t,\S,\I) U ] \leq A_{ij}\E[U].
\end{align*}
As a result, if $\Pr(|U|\leq 1) = 1$, then there exists a function $h_{i,j,U,n}:[0,\infty)\rightarrow [0, B_{ij}A_{ij}]$ such that
\begin{align}\label{eq:scary}
    B_{ij}\E[n\chi_{ij}(t, \S, \I)U] = A_{ij}\E[U] - \frac{h_{i,j,U,n}(t)}{\lambda^{(n)}}.
\end{align}
By making the above substitutions in (\ref{item:init_1}) - (\ref{item:init_4}), and by using the identity $\variance[Y]'=\E[Y^2]'-2\E[Y]\E[Y]'$, we obtain the following differential equations:
\begin{enumerate} [(I)]
    \item $\E[s_i]'=\sum_{j=1}^m \frac{h_{i,j, 7, n}}{\lambda^{(n)}} -\sum_{j=1}^m A_{ij}\left( \E[s_i]\E[\beta_j] + \frac{h_{i,j,1,n}}{2}(\variance[s_i]+\variance[\beta_j]) \right)$,
    \item
    $\E[\beta_i]'=\sum_{j=1}^m A_{ij}\left( \E[s_i]\E[\beta_j] + \frac{h_{i,j,1,n}}{2}(\variance[s_i]+\variance[\beta_j]) \right)-\sum_{j=1}^m \frac{h_{i,j, 7, n}}{\lambda^{(n)}} - \gamma_i\E[\beta_i]$,
    \item \begin{align*}
        \variance[s_i]'&=-2\sum_{j=1}^m A_{ij} \big( \E[s_i]^2 \E[\beta_j] + 2h_{i,j,2,n}(\variance[s_i] +\variance[\beta_j] )  \big) \cr
        &\quad+ \sum_{j=1}^m A_{ij} \left( \E[s_i]\E[\beta_j] + \frac{h_{i,j,1,n}}{2}( \variance[s_i] +\variance[\beta_j])  \right) \left (2\E[s_i] +\frac{1}{n} \right)\cr
        &\quad+\sum_{j=1}^m\left(\frac{2h_{i,j,8,n} }{\lambda^{(n)}} - \frac{h_{i,j,7,n} }{n\lambda^{(n)} } - 2\E[s_i]\frac{h_{i,j,7,n}}{\lambda^{(n)}}\right) ,
    \end{align*}
    \item \begin{align*}
        \variance[\beta_i]'&=2\sum_{j=1}^m A_{ij}\bigg( \E[s_i] \E[\beta_j]\E[\beta_i]+\frac{h_{i,j,1,n}}{2} (\variance[s_i] +\variance[\beta_j])\E[\beta_i] \cr
        &\quad\quad\quad\quad\quad\quad+ \frac{h_{i,j,5,n}}{2}\left( h_{i,j,6,n} ( 8\variance[s_i] +8\variance[\beta_j] ) +\variance[\beta_i]  \right)                     \bigg)\cr
        &\quad+\sum_{j=1}^m A_{ij}\left( \frac{1}{n} - 2\E[\beta_i] \right)\Big( \E[s_i]\E[\beta_j]\frac{ h_{i,j,1,n} }{2} (\variance[s_i] +\variance[\beta_j] )  \Big) - 2\gamma_i\variance[\beta_i] + \gamma_i\frac{ \E[\beta_i]}{n}\cr
        &\quad-\sum_{j=1}^m \left( \frac{2h_{i,j,9,n} }{\lambda^{(n)}} + \frac{h_{i,j,7,n} }{n\lambda^{(n)}} \right),
    \end{align*}
\end{enumerate}
where $h_{i,j,7,n}$, $h_{i,j,8,n}$, and $h_{i,j,9,n}$ are functions from $[0,\infty)$ to $[0, B_{ij}A_{ij} ]$ and are defined on the basis of~\eqref{eq:scary}.

The above equations constitute a proper system of differential equations with the same variables $\{\E[s_i]\}_{i=1}^m$,  $\{\variance[s_i]\}_{i=1}^m$,  $\{\E[\beta_i]\}_{i=1}^m$, and  $\{\variance[\beta_i]\}_{i=1}^m$ appearing on both the sides. To express these equations compactly, we define $z^{(n)}\in [0,1]^{4m}$ as the vector whose entries are given by $z_{i,1}^{(n)}:= z_{4(i-1)+1}^{(n)} := \E[s_i^{(n)}]$, $z_{i,2}^{(n)}:= z_{4(i-1)+2}^{(n)} := \E[\beta_i^{(n)}]$, $z_{i,3}^{(n)}:= z_{4(i-1)+3}^{(n)} := \variance[s_i^{(n)}]$, and $z_{i,4}^{(n)}:= z_{4i}^{(n)} := \variance[\beta_i^{(n)}].$ Then $z^{(n)}(t)$ is a solution to the initial value problem $(z^{(n)})' =g_n(t,z^{(n)};1/n,1/\lambda^{(n)})$ and $z(0)=z_0^{(n)}$, where
\begin{enumerate} [(I)]
    \item  $g_{i,n}^{(1)} (t,z;\varepsilon_1,\varepsilon_2):=-\sum_{j=1}^m A_{ij}\Big( z_{i,1} z_{j,2} + \frac{h_{i,j,1,n}}{2}( z_{i,3} + z_{j,4}  ) \Big)+\varepsilon_2\sum_{j=1}^m h_{i,j,7, n} $,
    \item 
    $g_{i,n}^{(2)}(t,z;\varepsilon_1,\varepsilon_2):=\sum_{j=1}^m A_{ij}\left( z_{i,1} z_{j,2} + \frac{h_{i,j,1,n}}{2}( z_{i,3} + z_{j,4}  ) \right) - \gamma_i z_{i,2} -\varepsilon_2 \sum_{j=1}^m h_{i,j,7, n} $,
    \item 
    \begin{align*}
        g_{i,n}^{(3)}(t,z;\varepsilon_1,\varepsilon_2)&:=-2\sum_{j=1}^m A_{ij} ( (z_{i,1})^2 z_{j,2}  + 2h_{i,j,2,n}(z_{i,3} + z_{j,4} ) )\cr 
        &\quad+\sum_{j=1}^m A_{ij} \Big( z_{i,1}z_{j,2}+\frac{h_{i,j,1,n}}{2}(z_{i,3}+z_{j,4})  \Big)\left( 2z_{i,1} + \varepsilon_1  \right)\cr
        &\quad+\sum_{j=1}^m\left({2h_{i,j,8,n} }{\varepsilon_2} - {h_{i,j,7,n} }{\varepsilon_1\varepsilon_2} -2\E[s_1] h_{i,j,7,n}\varepsilon_2 \right),
    \end{align*}
    \item
    \begin{align*}
        g_{i,n}^{(4)}(t,z;\varepsilon_1,\varepsilon_2)&= 2\sum_{j=1}^m A_{ij}\Bigg( z_{i,1} z_{j,2}z_{i,2} +\frac{h_{i,j,1,n}}{2} (z_{i,3} +z_{j,4})z_{i,2}\cr &\quad\quad\quad\quad\quad\quad+\frac{h_{i,j,5,n}}{2}\left( h_{i,j,6,n} ( 8z_{i,3} +8z_{j,4} ) +z_{i,4}  \Bigg)                     \right)\cr
        &\quad+ \sum_{j=1}^m A_{ij}\left( \varepsilon_1 - 2z_{i,2} \right) \left( z_{i,1}z_{j,2} +\frac{ h_{i,j,1,n} }{2} (z_{i,3} +z_{j,4} )  \right)\cr
        &\quad- 2\gamma_iz_{i,4} + \varepsilon_1\gamma_iz_{i,2} - \sum_{j=1}^m \left(2 {h_{i,j,9,n} }{\varepsilon_2} + {h_{i,j,7,n} }{\varepsilon_1\varepsilon_2} \right),
    \end{align*}
    \item  $z_0^{(n)}=(s_1^{(n)}(0), \beta_{1}^{(n)}(0),0,0, s_2^{(n)}(0), \beta_{2}^{(n)}(0),0,0, \ldots,s_{m}^{(n)}(0), \beta_{m}^{(n)}(0),0,0)$.
\end{enumerate}

Observe that irrespective of the value of $n$, the solution $(\bar z_{i,1}(t) , \bar z_{i,2}(t), \bar z_{i,3} (t), \bar z_{i,4} (t)):=(y_i(t), w_i(t), 0, 0)$ solves the initial value problem $z' = g_n(t,z;0,0)$ and $z(0)=z_0$, where
$$
    z_0:= (s_{1,0}, \beta_{1,0}, 0, 0,s_{2,0}, \beta_{2,0}, 0, 0,\ldots, s_{m,0}, \beta_{m,0}, 0, 0 ).
$$

Next, we need to bound $\left\|z^{(n)}(t)-\bar{z}(t)\right\|$  (where $\bar z(t)\in [0,1]^{4m}$ is the unique vector satisfying $\bar z_{4(i-1)+\ell}(t)=\bar z_{i,\ell}(t)$ for all $i\in [m]$ and $\ell\in[4]$). For this purpose, we will need the following lemma, which we borrow from~\cite{armbruster2017elementary}.

\begin{lemma}\label{lem:final}
Consider the initial value problems $x^{\prime}=f_{1}(t, x)$, $x(0)=x_{1}$ and $x^{\prime}=f_{2}(t, x)$, $x(0)=x_{2}$
with solutions $\varphi_{1}(t)$ and $\varphi_{2}(t)$ respectively. If $f_{1}$ is Lipschitz in $x$ with constant $L$ and
$\|f_{1}(t, x)-f_{2}(t, x)\| \leq M,$ then $\left\|\varphi_{1}(t)-\varphi_{2}(t)\right\| \leq\left(\left\|x_{1}-x_{2}\right\|+M / L\right) e^{L t}-M / L$.
\end{lemma}

Now, note that the domain of $z$ for $g_n(t,z;\varepsilon_1,\varepsilon_2)$ can be chosen to be bounded because $0\leq \E[s_i], \E[\beta_i] \leq 1$ and $\variance[s_i] \leq \E[s_i^2] \leq 1$. Similarly, $\variance[\beta_i]\le1$. Also, we let $\varepsilon_1,\varepsilon_2\in(0,1)$ and define $\varepsilon:=\max\{\varepsilon_1,\varepsilon_2\}$. Since $g_n(t,z;0,0)$ is a polynomial in $z$, it is Lipschitz-continuous with some Lipschitz constant $L\in(0,\infty)$. In addition, we use the bounds on $z$ and the functions $\{h_{i,j,\ell,n}:1\leq \ell\leq 9\}$ as follows:
\begin{align*}
    \|g_n(t,z;\varepsilon_1,\varepsilon_2) - g_n(t,z;0,0)\| 
    &\leq 2\sum_{i=1}^m \sum_{j=1}^m A_{ij}\varepsilon \left | z_{i,1}z_{j,2} +\frac{ h_{i,j,1,n} }{2} (z_{i,3} +z_{j,4} )  \right |+\sum_{i=1}^m\gamma_i\varepsilon + 10\sum_{i=1}^m\sum_{j=1}^m  A_{ij}B_{ij}\varepsilon\cr
    &\leq \left(\sum_{i=1}^m \sum_{j=1}^m A_{ij}(4 + 10B_{ij}) + \sum_{i=1}^m \gamma_i\right)\varepsilon,
\end{align*}
i.e., 
\begin{align*}
    &\|g_n(t,z;\varepsilon_1,\varepsilon_2) - g_n(t,z;0,0)\| \leq M(\varepsilon),
\end{align*}
where $M(\varepsilon):=\left(\sum_{i=1}^m \sum_{j=1}^m A_{ij}(4 + 10B_{ij}) + \sum_{i=1}^m \gamma_i\right)\varepsilon$.

We now apply Lemma~\ref{lem:final} after setting
\begin{gather*}
    f_1(t,x)=g_n(t,x;0,0), \quad f_2(t,x)=g_n(t,x;1/n,1/\lambda^{(n)}),\quad x_1 = z_0,\quad x_2 = z_0^{(n)}.
\end{gather*}
Also, we let $\varphi_1 = \bar z$ and $\varphi_2 = z^{(n)}$. Then we have
$$
\left\|z^{(n)}(t)-\bar{z}(t)\right\| \leq\left(\left\|z_{0}-z_{0}^{(n)}\right\|+\frac{M\alpha_n }{L}\right) e^{L t}-\frac{M\alpha_n}{L},
$$
where $\alpha_n:=\max\left\{\frac{1}{n},\frac{1}{\lambda^{(n)}}\right\}$.
Thus, for all $t \leq T$,
\begin{align}\label{eq:copied}
    \left\|z^{(n)}(t)-\bar{z}(t)\right\| \leq\left(\left\|z_{0}-z_{0}^{(n)}\right\|+\frac{M\alpha_n}{L}\right) e^{L T}-\frac{M\alpha_n}{L}.
\end{align}
Since $\lim_{\varepsilon\rightarrow0}M(\varepsilon)=0$, $\lim_{n\rightarrow\infty} z_{0}^{(n)} = z_{0}$, and $\lim_{n\rightarrow\infty}\alpha_n=\lim_{n\rightarrow\infty}\max\left\{\frac{1}{n},\frac{1}{\lambda^{(n)}}\right\}=0$, the right hand side of~\eqref{eq:copied}  goes to zero as $n \rightarrow \infty .$ Hence we have the uniform convergence $z^{(n)} \rightarrow \bar{z}$ over any finite time interval $[0, T]$. The last step is to show that $z^{(n)}\rightarrow\bar z$ implies $L^2$-convergence, i.e., $\E[\|(s_i^{(n)} - y_i, \beta_i^{(n)} - w_i)\|_2] \rightarrow 0$ as $n\rightarrow\infty$. To this end, we have
\begin{align*}
    \E[\|(s_i^{(n)} - y_i, \beta_i^{(n)} - w_i)\|_2^2] &= \E[(s_i^{(n)} - y_i)^2] + \E[(\beta_i^{(n)} - w_i)^2]\cr
    & = (\E[s_i^{(n)}]-y_i])^2 + (\E[\beta_i^{(n)}] - w_i)^2 + \variance[s_i^{(n)}] + \variance[\beta_i^{(n)}]\cr
    &\leq |\E[s_i^{(n)}]-y_i]| + |\E[\beta_i^{(n)}]-w_i]| +  \variance[s_i^{(n)}] + \variance[\beta_i^{(n)}]\cr
    &= |z_{i,1}^{(n)} - \bar z_{i,1}| + |z_{i,2}^{(n)} - \bar z_{i,2}| + |z_{i,3}^{(n)} - \bar z_{i,3}| + |z_{i,4}^{(n)} - \bar z_{i,4}|,
\end{align*}
where we used that $\bar z_{i,3} = \bar z_{i,4} = 0$, and the inequality holds because $y_i, \E[ s_i^{(n)}], w_i, \E[\beta_i^{(n)} ] \in [0,1]$. Thus, the uniform convergence of $z^{(n)}$ to $ \bar{z}$ over $[0,T]$ proves that $\E[\|(s_i^{(n)} - y_i, \beta_i^{(n)} - w_i)\|_2] \rightarrow 0$ as $n\rightarrow\infty$.

\end{proof}

\end{document}